\documentclass[letterpaper,twoside]{article}

\usepackage[letterpaper,left=1in,right=1in,top=1.5in,bottom=1.5in]{geometry}

\newcommand{\myauthor}{Benjamin Antieau, David Gepner, and Jeremiah Heller}

\newcommand{\mytitle}{$K$-theoretic obstructions to bounded
$t$-structures}

\title{\mytitle}
\author{Benjamin Antieau\footnote{Benjamin Antieau was supported
by NSF Grants DMS-1461847 and DMS-1552766.}~, David Gepner\footnote{David
Gepner was supported
by NSF Grants DMS-1406529 and DMS-1714273.}~, and Jeremiah Heller\footnote{Jeremiah Heller was supported by NSF Grant  DMS-1710966.}}
\date{}
 
\usepackage[pdfstartview=FitH,
            pdfauthor={\myauthor},
            pdftitle={\mytitle},
            colorlinks,
            linkcolor=reference,
            citecolor=citation,
            urlcolor=e-mail,
            backref]{hyperref}

\usepackage{amsmath}
\usepackage{amscd}
\usepackage{amsbsy}
\usepackage{amssymb}
\usepackage{verbatim}
\usepackage{eufrak}
\usepackage{eucal}
\usepackage{microtype}
\usepackage{hyperref}
\usepackage{mathrsfs}
\usepackage{amsthm}
\usepackage[all,cmtip]{xy}
\usepackage{tikz}
\usetikzlibrary{matrix,arrows}
\usepackage[abbrev,lite,alphabetic]{amsrefs}
\usepackage{dsfont}

\usepackage{fancyhdr}
\usepackage{makeidx}
\makeindex
\newcommand{\df}[1]{{\bf #1} \index{#1}}

\usepackage{color}
\definecolor{todo}{rgb}{1,0,0}
\definecolor{conditional}{rgb}{0,1,0}
\definecolor{e-mail}{rgb}{0,.40,.80}
\definecolor{reference}{rgb}{.20,.60,.22}
\definecolor{mrnumber}{rgb}{.80,.40,0}
\definecolor{citation}{rgb}{0,.40,.80}

\let\oldmarginpar\marginpar
\renewcommand\marginpar[1]{\-\oldmarginpar[\raggedleft\footnotesize #1]%
{\raggedright\footnotesize #1}}

\newcommand{\Ascr}{\mathcal{A}}
\newcommand{\Bscr}{\mathcal{B}}
\newcommand{\Cscr}{\mathcal{C}}
\newcommand{\Dscr}{\mathcal{D}}
\newcommand{\Escr}{\mathcal{E}}
\newcommand{\Fscr}{\mathcal{F}}
\newcommand{\Gscr}{\mathcal{G}}
\newcommand{\Hscr}{\mathcal{H}}

\newcommand{\Oscr}{\mathcal{O}}
\newcommand{\Pscr}{\mathcal{P}}

\newcommand{\Tscr}{\mathcal{T}}

\newcommand{\C}{\mathrm{C}}
\newcommand{\D}{\mathrm{D}}
\newcommand{\E}{\mathrm{E}}
\newcommand{\F}{\mathrm{F}}
\renewcommand{\H}{\mathrm{H}}
\newcommand{\K}{\mathrm{K}}
\renewcommand{\L}{\mathrm{L}}

\newcommand{\N}{\mathrm{N}}

\renewcommand{\S}{\mathrm{S}}

\renewcommand{\AA}{\mathds{A}}
\newcommand{\CC}{\mathds{C}}
\newcommand{\EE}{\mathds{E}}

\newcommand{\NN}{\mathds{N}}
\newcommand{\PP}{\mathds{P}}
\newcommand{\QQ}{\mathds{Q}}
\newcommand{\RR}{\mathds{R}}
\renewcommand{\SS}{\mathds{S}}
\newcommand{\ZZ}{\mathds{Z}}

\newcommand{\dR}{\mathrm{dR}}

\newcommand{\proj}{\mathrm{proj}}

\newcommand{\dg}{\mathrm{dg}}

\newcommand{\op}{\mathrm{op}}
\newcommand{\ex}{\mathrm{ex}}
\newcommand{\lex}{\mathrm{lex}}

\newcommand{\cn}{\mathrm{cn}}
\newcommand{\perf}{\mathrm{perf}}

\newcommand{\st}{\mathrm{st}}
\newcommand{\sing}{\mathrm{sing}}

\newcommand{\Sp}{\mathrm{Sp}}

\newcommand{\heart}{\heartsuit}

\newcommand{\ev}{\mathrm{ev}}
\newcommand{\coev}{\mathrm{coev}}

\DeclareMathOperator{\id}{id}

\newcommand{\coker}{\mathrm{coker}}
\newcommand{\pfp}{\pi_*\textrm{-fp}}
\renewcommand{\geq}{\geqslant}
\renewcommand{\leq}{\leqslant}
\newcommand{\Ho}{\mathrm{Ho}}
\newcommand{\cone}{\mathrm{cone}}

\newcommand{\Cat}{\mathrm{Cat}}

\newcommand{\Frob}{\mathscr{F}\mathrm{rob}}

\DeclareMathOperator{\Ext}{Ext}

\newcommand{\Kbar}{\mathrm{K}^{\mathrm{Bar}}}

\newcommand{\KU}{\mathrm{KU}}
\newcommand{\ko}{\mathrm{ko}}
\newcommand{\KO}{\mathrm{KO}}
 
\DeclareMathOperator{\Fun}{Fun}

\DeclareMathOperator{\End}{End}
\DeclareMathOperator{\Hom}{Hom}
\newcommand{\Map}{\mathrm{Map}}

\newcommand{\Loc}{\mathrm{Loc}}
\newcommand{\Ch}{\mathrm{Ch}}
\newcommand{\Ac}{\mathrm{Ac}}
\newcommand{\Mod}{\mathrm{Mod}}

\newcommand{\Perf}{\mathrm{Perf}}
\newcommand{\Ind}{\mathrm{Ind}}

\newcommand{\Alg}{\mathrm{Alg}}
\newcommand{\CAlg}{\mathrm{CAlg}}

\newcommand{\coh}{\mathrm{coh}}

\newcommand{\Gm}{\mathds{G}_{m}}

\DeclareMathOperator*{\colim}{colim}

\DeclareMathOperator{\Spec}{Spec}

\newcommand{\we}{\simeq}
\newcommand{\iso}{\cong}
\newcommand{\rcof}{\rightarrowtail}
\newcommand{\rfib}{\twoheadrightarrow}
\newcommand{\lcof}{\leftarrowtail}
\newcommand{\lfib}{\twoheadleftarrow}

\theoremstyle{plain}
\newtheorem{theorem}{Theorem}[section]
\newtheorem*{theorem*}{Theorem}
\newtheorem{lemma}[theorem]{Lemma}

\newtheorem{proposition}[theorem]{Proposition}

\newtheorem{conjecture}[theorem]{Conjecture}
\newtheorem{corollary}[theorem]{Corollary}

\newtheoremstyle{named}{}{}{\itshape}{}{\bfseries}{.}{.5em}{#1 \thmnote{#3}}
\theoremstyle{named}

\newtheorem*{namedconjecture}{Conjecture}

\theoremstyle{definition}
\newtheorem{definition}[theorem]{Definition}
\newtheorem{notation}[theorem]{Notation}

\newtheorem{example}[theorem]{Example}
\newtheorem{question}[theorem]{Question}
\newtheorem{construction}[theorem]{Construction}
\newtheorem{remark}[theorem]{Remark}

\begin{document}

\maketitle

\begin{abstract}
    \noindent
    Schlichting conjectured that the negative $K$-groups of small abelian
    categories vanish and proved this for noetherian abelian categories and
    for all abelian categories in degree $-1$. The main results of this paper
    are that $\K_{-1}(E)$ vanishes when $E$ is a small stable $\infty$-category
    with a bounded $t$-structure and that $\K_{-n}(E)$ vanishes for all $n\geq 1$
    when additionally the heart of $E$ is noetherian. It follows that Barwick's
    theorem of the heart holds for nonconnective $K$-theory spectra when the
    heart is noetherian. We give several applications, to non-existence results
    for bounded $t$-structures and stability conditions, to possible
    $K$-theoretic obstructions to the existence of the motivic $t$-structure,
    and to vanishing results for the negative $K$-groups of a large class of dg
    algebras and ring spectra.
%     On the other, we prove that various localization sequences
%     \`a la Blumberg-Mandell and Barwick-Lawson hold in
%     nonconnective $K$-theory. This leads to the first computations in
%     negative $K$-theory of nonconnective ring spectra arising in chromatic homotopy
%     theory. For example, $\K_{-n}(\KU)$ vanishes for all $n\geq 1$.

    \paragraph{Key Words.}
    Negative $K$-theory, $t$-structures, abelian categories, $K$-theory of dg
    algebras and ring spectra.

    \paragraph{Mathematics Subject Classification 2010.}
    Primary:
    \href{http://www.ams.org/mathscinet/msc/msc2010.html?t=16Exx&btn=Current}{16E45},
    \href{http://www.ams.org/mathscinet/msc/msc2010.html?t=18Exx&btn=Current}{18E30},
    \href{http://www.ams.org/mathscinet/msc/msc2010.html?t=19Dxx&btn=Current}{19D35}.
    Secondary:
    \href{http://www.ams.org/mathscinet/msc/msc2010.html?t=16Pxx&btn=Current}{16P40},
    \href{http://www.ams.org/mathscinet/msc/msc2010.html?t=18Exx&btn=Current}{18E10},
    \href{http://www.ams.org/mathscinet/msc/msc2010.html?t=55Pxx&btn=Current}{55P43}.
\end{abstract}

\tableofcontents

\section{Introduction}\label{sec:introduction}

We prove the following theorems about negative and nonconnective $K$-theory.

\begin{theorem}\label{thm:introkm1}
    If $E$ is a small stable $\infty$-category\footnote{The conjectures and
    results of this paper apply equally well to any triangulated category
    with a bounded $t$-structure that admits a model, either as a dg
    category or a stable $\infty$-category. This includes all
    examples of triangulated categories with bounded $t$-structures we have
    found in the literature. For background on stable $\infty$-categories, see
    Section~\ref{sec:stable}.} with a bounded $t$-structure, then $\K_{-1}(E)=0$.
\end{theorem}

\begin{theorem}\label{thm:mainintro}
    If $E$ is a small stable $\infty$-category equipped with a bounded
    $t$-structure such that $E^{\heartsuit}$ is noetherian, then $\K_{-n}(E)=0$
    for $n\geq 1$.
\end{theorem}

\begin{theorem}[Nonconnective theorem of the heart]\label{thm:heartintro}
    If $E$ is a small stable $\infty$-category with a bounded $t$-structure
    such that $E^\heartsuit$ is noetherian,
    then the natural map $$\K(E^\heartsuit)\xrightarrow{\we}\K(E)$$ is an equivalence.
\end{theorem}

The first two theorems generalize results of Schlichting
from~\cite{schlichting}, who proved the theorems in the special case where
$E\we\Dscr^b(A)$, the bounded derived $\infty$-category of a small abelian
category $A$. Note that our theorems are much more general than Schlichting's
results, as stable $\infty$-categories with bounded $t$-structures are
typically not bounded derived $\infty$-categories. The third result follows from the first two and Barwick's theorem
of the heart for connective $K$-theory~\cite{barwick}.

The proof of Theorem~\ref{thm:mainintro} is based on induction, with the base
case provided by Theorem~\ref{thm:introkm1}. The proof of Schlichting's result that $\K_{-1}(A)=0$
for general abelian categories $A$ is not hard, but the proof of
Theorem~\ref{thm:introkm1} is
more difficult as it is necessary to find an excisive square
playing the same role for $E$ that the square
\begin{equation*}
    \xymatrix{
        \Dscr^b(A)\ar[r]\ar[d]  &   \Dscr^+(A)\ar[d]\\
        \Dscr^-(A)\ar[r]        &   \Dscr(A)
    }
\end{equation*}
plays for $\Dscr^b(A)$.

In the inductive step, we use stable $\infty$-categories
of endomorphisms and automorphisms of $E$. We construct an exact sequence
$$\Dscr_{\{0\}}(\AA^1,\Cscr)^\omega\rightarrow\Dscr(\AA^1,\Cscr)^\omega\rightarrow\Dscr(\Gm,\Cscr)^\omega$$
of small idempotent complete stable $\infty$-categories, where $\Cscr=\Ind(E)$
is the ind-completion of $E$, the superscript $\omega$ denotes the
subcategory of compact objects, and
$\Dscr(\AA^1,\Cscr)\we\Mod_{\SS[s]}\otimes\Cscr$, and similarly for
$\Dscr(\Gm,\Cscr)$. The subscript $\{0\}$ denotes the full subcategory
$\Dscr_{\{0\}}(\AA^1,\Cscr)\subseteq\Dscr(\AA^1,\Cscr)$ of objects killed by
inverting the endomorphism $s$. Note that the $\infty$-category
$\Dscr(\AA^1,\Cscr)^\omega$ differs from the $\infty$-category used
in~\cite{bgt3} to define the \emph{$K$-theory of endomorphisms}. Indeed, the
$K$-theory of endomorphisms of $E$ takes as input the $\infty$-category of
endomorphisms of objects of $E$. But, these need not be compact in
$\Dscr(\AA^1,\Cscr)$. Conversely, the underlying object of a compact object of $\Dscr(\AA^1,\Cscr)$ need not be compact in $\Cscr$.

The technical input for the inductive step, proven in
Corollary~\ref{cor:tstructuregm}, is
that if $E$ is a small stable $\infty$-category with a bounded $t$-structure
such that $E^\heartsuit$ is noetherian, then the same is true for
$\Dscr(\Gm,\Cscr)^\omega$. This allows an inductive argument because $\K(E)$ is
a summand of $\K(\Dscr_{\{0\}}(\AA^1,\Cscr)^\omega)$ and this summand maps
trivially to $\K(\Dscr(\AA^1,\Cscr)^\omega)$.

Theorem~\ref{thm:mainintro} can be extended to the case
where $E^\heart$ is merely stably coherent; we do so in Section~\ref{sub:fj}.
We discuss in Sections~\ref{sub:counterexamples} and~\ref{sub:serrecones}
a counterexample to our approach when $E^\heartsuit$ is not noetherian as well
as several possible approaches for circumventing this problem. We hope that
these more speculative sections will serve to pique the interest of readers
thinking about related problems.

\paragraph{Conjectures.}
Schlichting made the following conjecture in~\cite{schlichting}.

\begin{namedconjecture}[A]\hypertarget{conj:a}
    If $A$ is a small abelian category, then $\K_{-n}(A)=0$ for $n\geq 1$.
\end{namedconjecture}

Motivated by this, we pose the next two conjectures.

\begin{namedconjecture}[B]\hypertarget{conj:b}
    If $E$ is a small stable $\infty$-category
    with a bounded $t$-structure, then $\K_{-n}(E)=0$ for $n\geq 1$.
\end{namedconjecture}

\begin{namedconjecture}[C]\hypertarget{conj:c}
    If $E$ is a small stable $\infty$-category with a bounded $t$-structure,
    then the natural map $\K(E^\heartsuit)\rightarrow\K(E)$ is an equivalence of nonconnective
    $K$-theory spectra.
\end{namedconjecture}

Conjecture~\hyperlink{conj:a}{A} is a special case of the second
conjecture by setting $E=\Dscr^b(A)$, the bounded derived $\infty$-category of
$A$, and we will therefore refer to the second as the generalized Schlichting conjecture. 
The connective part of Conjecture~\hyperlink{conj:c}{C} is Barwick's theorem~\cite{barwick} for connective
$K$-theory: $\K^{\cn}(E^\heartsuit)\we\K^{\cn}(E)$, which generalizes the
Gillet-Waldhausen theorem~\cite{thomason-trobaugh}*{Theorem~1.11.7} in the case
that $E=\Dscr^b(A)$. So, the open part of that conjecture may be rephrased as
saying that
$\K_{-n}(E^\heartsuit)\rightarrow\K_{-n}(E)$ is an isomorphism for all $n\geq
1$. Of course, this would follow from Conjecture~\hyperlink{conj:b}{B} together with Conjecture~\hyperlink{conj:a}{A}.
In fact, Conjecture~\hyperlink{conj:b}{B} holds if and only if
Conjectures~\hyperlink{conj:a}{A} and~\hyperlink{conj:c}{C} hold.

Note that there are examples of stable $\infty$-categories $E$ with two
different bounded $t$-structures, one having a noetherian heart and the other having a
non-noetherian heart. The standard example, due to Thomas and written down
in~\cite{abramovich-polishchuk}, is $\Dscr^b(\PP^1)$, and was pointed out to us by Calabrese.

\paragraph{Applications.}
There are three major areas of application of the work in this paper:
obstructions to the existence of $t$-structures (and hence to stability
conditions) on $\Perf(X)$ when $X$ is a singular scheme, possible obstructions
to the existence of the conjectural motivic $t$-structure, and vanishing results
for the negative $K$-theory of nonconnective dg algebras and ring spectra.
We describe the first two areas briefly below and leave the extensive vanishing
results for ring spectra to Section~\ref{sec:negative}.

\paragraph{Stability conditions.}
Bridgeland introduced in~\cite{bridgeland} the notion of stability conditions on abelian and triangulated categories.
Moreover, he proved in~\cite{bridgeland}*{Proposition~5.3} that giving a stability condition on a
triangulated category $\Tscr$ is equivalent to giving a bounded $t$-structure
on $\Tscr$ together with a stability condition on $\Tscr^\heart$. A crucial and
open problem in the theory of stability conditions is when $\Perf(X)$ admits
any stability conditions at all for $X$ a smooth scheme over $\CC$.
This is open even for general smooth proper threefolds (see for example~\cite{bayer-macri-toda}).

Our methods give $K$-theoretic obstructions to the existence of bounded $t$-structures
and hence to stability conditions. As far as we are aware these are the first
obstructions of any kind to the existence of bounded $t$-structures.

\begin{corollary}\label{cor:stability}
    Let $X$ be a scheme such that $\K_{-1}(X)\neq 0$. Then, there exists no
    bounded $t$-structure (and hence no stability condition) on $\Perf(X)$. If $\K_{-n}(X)\neq 0$ for some $n\geq
    2$, then there exists no bounded $t$-structure on $\Perf(X)$ with
    noetherian heart.
\end{corollary}

The corollary applies to a wide variety of singular schemes, even such simple
examples as nodal cubic curves, where $\K_{-1}(X)\iso\ZZ$. Note that when $X$ is noetherian and singular, it is
easy to see that the canonical bounded $t$-structure on $\Dscr^b(X)$ does not restrict
to one on $\Perf(X)\subseteq\Dscr^b(X)$. A priori there could be other,
exotic $t$-structures. We propose the following conjecture, which generalizes
Corollary~\ref{cor:stability}.

\begin{conjecture}
    Let $X$ be a noetherian scheme of finite Krull dimension. If $X$ is not
    regular, then $\Perf(X)$ admits no bounded $t$-structure.
\end{conjecture}

Based on Corollary~\ref{cor:stability}, when $X$ is
singular, $\Dscr^b(X)$ appears more natural from the point of view of stability
conditions.

\newcommand{\gm}{\mathrm{gm}}
\newcommand{\eff}{\mathrm{eff}}

\paragraph{Motivic $t$-structures.}
One of the major open problems in motives
(see~\cite{kahn-conjectures}*{Section~4.4.3}) is to construct a bounded
$t$-structure on Voevodsky's triangulated category
$\mathrm{DM}_{\gm}^{\eff}(k)_\QQ$ of
rational effective geometric motives over a field $k$. The heart of this
$t$-structure would be the abelian category of mixed motives. Voevodsky observed
in~\cite{voevodsky-triangulated} that there can be no integral motivic
$t$-structure when there are smooth projective conic curves over $k$ with no
rational points (thus for example when $k=\QQ$), although potentially there
could be other bounded $t$-structures that do not satisfy all of the expected
properties. Our next corollary implies a possible approach to proving
non-existence of any motivic $t$-structure. Note that the heart of the motivic
$t$-structure is expected to be noetherian.

\begin{corollary}
    If $\K_{-n}(\mathrm{DM}_{\gm}^{\eff}(k)_\QQ)\neq 0$ for some $n\geq 1$, then there is no
    motivic $t$-structure.
\end{corollary}

Using our work, Sosnilo has proved in~\cite{sosnilo} that in fact a
different conjecture of Voevodsky, the nilpotence conjecture
of~\cite{voevodsky-nilpotence}, would imply
$\K_{-n}(\mathrm{DM}_{\gm}^{\eff}(k)_\QQ)=0$ for all $n\geq 1$.
Put another way, if $\K_{-n}(\mathrm{DM}_{\gm}^{\eff}(k)_\QQ)\neq 0$ for some
$n\geq 1$, then the nilpotence conjecture would also be false.

\paragraph{Outline.}
Section~\ref{sec:t} is dedicated to background on $t$-structures, proving
several new inheritance results about $t$-structures, $K$-theoretic
excisive squares, and the proof of Theorem~\ref{thm:introkm1}.
Section~\ref{sec:induction} contains the proofs of Theorem~\ref{thm:mainintro}
and~\ref{thm:heartintro}
as well as our thoughts of how one might attempt to prove
Conjecture~\hyperlink{conj:b}{B} in general. Section~\ref{sec:negative}
contains our applications to the negative $K$-theory of ring spectra.
In Appendix~\ref{sec:frobenius}, we construct a functorial $\infty$-categorical model of
the stable category of a Frobenius category. This is needed to check that the
definition of negative K-theory we use agrees with Schlichting's.

\paragraph{Notation.}
Throughout, unless otherwise stated, we use homological indexing for chain complexes and objects in
stable $\infty$-categories. The $\infty$-category of small stable
$\infty$-categories and exact functors is written $\Cat_{\infty}^{\ex}$, while the full subcategory of
small \emph{idempotent complete} stable $\infty$-categories is written
$\Cat_{\infty}^\perf$. Given a small stable $\infty$-category $E$, we denote by
$\widetilde{E}$ or $E^\sim$ the idempotent completion of $E$. If $E\subseteq F$
is a fully faithful inclusion such that $E$ is idempotent complete in $F$, then
$F/E$ denotes the Verdier quotient (the cofiber in $\Cat_{\infty}^\ex$).

If $\Cscr$ is an $\infty$-category, $\Map_{\Cscr}(M,N)$
is the mapping space of morphisms from $M$ to $N$ in $\Cscr$. Given an
idempotent complete stable $\infty$-category $E$, $\K(E)$ always
denotes the \emph{nonconnective} $K$-theory spectrum of $E$, as defined
in~\cite{bgt1}. We use $\K^{\cn}(E)$ for the connective cover of $\K(E)$, the
connective $K$-theory spectrum of $E$. Finally, if $R$ is a ring spectrum,
$\Mod_R$, $\Alg_R$, and $\CAlg_R$ denote the $\infty$-categories of $R$-module
spectra, $\EE_1$-$R$-algebra spectra (if $R$ is commutative), and $\EE_\infty$-$R$-algebra spectra (if
$R$ is commutative), respectively (even if $R$ is discrete). If $R$ is
discrete, we let $\Mod_R^\heartsuit$,
$\Alg_R^{\heartsuit}$, and $\CAlg_R^{\heartsuit}$ denote the ordinary
categories of discrete right $R$-modules, discrete associative $R$-algebras (if
$R$ is commutative),
and discrete commutative $R$-algebras (if $R$ is commutative), respectively;
this notation reflects the fact that the abelian category of discrete right
$R$-modules is equivalent to the heart of the standard $t$-structure on the
stable $\infty$-category $\Mod_R$.

\paragraph{Acknowledgments.}
BA and JH thank the Hausdorff Institute
for Mathematics in Bonn and DG thanks the Max Planck Institute for Mathematics:
these were our hosts during the summer of 2015, when this
project was conceived. BA thanks Akhil Mathew for several
conversations that summer at HIM, especially about bounded $t$-structures for
compact modules over cochain algebras.

BA thanks John Calabrese, Denis-Charles Cisinski, Michael Gr\"ochenig, Jacob Lurie, Matthew
Morrow, Marco Schlichting, Jesse Wolfson, and Matthew Woolf for conversations
and emails about material
related to this paper. DG thanks Andrew Blumberg and Markus Spitzweck for
conversations about material related to this paper. Both BA and DG would
especially like to thank Benjamin Hennion for explaining Tate objects and the
subtleties behind excisive squares.
We all are very grateful for detailed, helpful comments from an anonymous
referee.

We also thank the UIC Visitors' Fund, Purdue University, UIUC, and Lars Hesselholt
for supporting collaborative visits in 2016.

\section{$t$-structures}\label{sec:t}

We give some background on stable $\infty$-categories in Section~\ref{sec:stable}.
After recalling $t$-structures in
Section~\ref{sub:recollections}, we study induced $t$-structures
on ind-completions and localizations in
Section~\ref{sub:induced}. In some cases,
our results extend results in~\cite{bbd} beyond
the setting in which all functors admit left and right adjoints that preserve
compact objects (the main assumption in~\cite{bbd}). The ability to construct a $t$-structure on
a localization in certain circumstances will be used
later in the paper when we perform the inductive step in our generalization of
Schlichting's theorem.

In  Section~\ref{sub:mv}, we study excisive squares in algebraic
$K$-theory and their connection to adjointability.
We prove that $\K_{-1}(E)=0$ when $E$ is a small stable $\infty$-category with
a bounded $t$-structure in Section~\ref{sub:km1}.

\subsection{Stable $\infty$-categories}\label{sec:stable}

For the purposes of studying $K$-theory, it has been known for some time that
triangulated categories are not sufficient. This was the result of work of
Schlichting~\cite{schlichting-triangulated}, which gave an example of two stable
model categories with triangulated equivalent homotopy categories but different
$K$-theories. On the other hand, To\"en and Vezzosi~\cite{toen-vezzosi} showed that $K$-theory is a
good invariant of simplicial localizations of Waldhausen categories in the
following sense. If $C$ and $D$ are good Waldhausen categories and if the
simplicial localizations $L^HC$ and $L^HD$
are equivalent simplicial categories, then $\K(C)\we\K(D)$. Thus, the
simplicial localization loses some information, like passing to the homotopy
category, but not so much that $K$-theory is inaccessible. These simplicial
localizations are a kind of enhancement of the triangulated homotopy
categories, and it is now well-understood that $K$-theory requires some kind of
enhancement.

Unfortunately, computations are difficult in the model categories of simplicial
categories and dg categories, and it is much easier to work in the setting of
$\infty$-categories. The $K$-theory of $\infty$-categories is studied
in~\cite{bgt1} and~\cite{barwick-ktheory} and it agrees in all cases with Waldhausen
$K$-theory when both are defined. So, this setting provides a
best-of-both-worlds approach to $K$-theory, where we can not only compute
$K$-theory correctly but we can also compute maps between the inputs.
The theory of $\infty$-categories is not the only way of doing this, but it is
by now the most well-developed and it is the most well-suited for the problems
we study. 

A pointed $\infty$-category is an $\infty$-category $E$ with an object $0$ that is
both initial and final. It is called a zero object of $E$.
A cofiber sequence in a pointed $\infty$-category is a commutative
diagram\begin{equation*}\xymatrix{a\ar[r]\ar[d]&b\ar[d]\\0\ar[r]&d}\end{equation*}
which is a pushout diagram in the sense of colimits in
$\infty$-categories as developed in~\cite{htt}. It is standard
practice to abbreviate and write $a\rightarrow b\rightarrow c$
for a cofiber sequence. If $f:a\rightarrow b$ is a morphism in $E$, then a
cofiber for $f$ is a cofiber sequence $a\rightarrow b\rightarrow c$. Cofibers
for $f$ are unique up to homotopy. Fiber sequences and fibers are defined
similarly.

By definition,
a pointed $\infty$-category is {\bf stable} if it has all cofibers and fibers
and if a triangle in $E$ is a fiber sequence if and only if it is a cofiber
sequence. It turns out that this definition is equivalent to asking for a
pointed $\infty$-category to have all finite colimits and for
the suspension functor $\Sigma:E\rightarrow E$ to be an
equivalence (see~\cite{ha}*{Corollary~1.4.2.27}).

Unlike the case of triangulated categories in which the
triangulation is extra \emph{structure} which must be specified, stable
$\infty$-categories are $\infty$-categories with certain \emph{properties}, and
the homotopy category $\Ho(E)$ of a stable $\infty$-categories is an ordinary
category equipped with a \emph{canonical} triangulation.
If $\Cscr$ is stable, \cite{ha}*{Theorem~1.1.2.15} says that a sequence
$a\rightarrow b\rightarrow c$ determines a cofiber sequence if
and only $a\rightarrow b\rightarrow c$ is a distinguished
triangle in the triangulated homotopy category $\Ho(E)$.
For additional details and background about stable $\infty$-categories, see~\cite{ha}*{Chapter~1}.

\subsection{Definitions and first properties}\label{sub:recollections}

The notion of a $t$-structure appears in
Be\u{\i}linson-Bernstein-Deligne~\cite{bbd}*{Definition~1.3.1}. However, as we will
work with homological indexing, Lurie's treatment
in~\cite{ha}*{Definition~1.2.1.1} is more a
convenient reference.
If $E$ is a stable $\infty$-category and $x\in E$, we will typically write
$x[n]$ for the $n$-fold suspension $\Sigma^n x$ of $x$. If $F\subseteq E$ is a
full subcategory, we will also write $F[n]\subseteq E$ for the full subcategory
spanned by the objects of the form $x[n]$, where $x$ is an object of $F$.

\begin{definition}\label{def:t}
    A \df{$\mathbf{t}$-structure} on a stable $\infty$-category $E$ consists of a pair of full
    subcategories $E_{\geq 0}\subseteq E$ and $E_{\leq 0}\subseteq E$ satisfying the following
    conditions:
    \begin{enumerate}
        \item[(1)]   $E_{\geq 0}[1]\subseteq E_{\geq 0}$ and $E_{\leq 0}\subseteq E_{\leq 0}[1]$;
        \item[(2)]   if $x\in E_{\geq 0}$ and $y\in E_{\leq 0}$, then
            $\Hom_E(x,y[-1])=0$;
        \item[(3)]   every $x\in E$ fits into a cofiber sequence $\tau_{\geq
                0}x\rightarrow x\rightarrow\tau_{\leq -1}x$ where $\tau_{\geq 0}x\in
                E_{\geq 0}$ and $\tau_{\leq -1}x\in E_{\leq 0}[-1]$.
    \end{enumerate}
    An exact functor $E\rightarrow F$ between stable $\infty$-categories
    equipped with $t$-structures is \df{left $\mathbf{t}$-exact} (resp.
    {\bf right $\mathbf{t}$-exact}) if it sends $E_{\leq 0}$ to $F_{\leq 0}$
    (resp. $E_{\geq 0}$ to $F_{\geq 0}$). An exact functor is $t$-exact if is
    both left and right $t$-exact. We set $E_{\geq n}=E_{\geq 0}[n]$ and $E_{\leq
    n}=E_{\leq 0}[n]$.
\end{definition}

\begin{example}\label{ex:t}
    \begin{enumerate}
        \item[(a)]   If $A$ is a small abelian category, then the bounded derived
            $\infty$-category $\Dscr^b(A)$ (see Definition~\ref{def:bounded}) admits a
            canonical $t$-structure, where $\Dscr^b(A)_{\geq n}$ consists of the
            complexes $x$ such that $\H_i(x)=0$ for $i<n$, and similarly for
            $\Dscr^b(A)_{\leq n}$.
        \item[(b)]   If $A$ is a Grothendieck abelian category, then the
            derived $\infty$-category $\Dscr(A)$ admits
            a $t$-structure with the same description as the previous example.
            This stable $\infty$-category and its $t$-structure are studied
            in~\cite{ha}*{Section~1.3.5}.
        \item[(c)]   If $R$ is a connective $\EE_1$-ring spectrum, then the stable
            presentable $\infty$-category $\Mod_R$ of right $R$-module spectra admits
            a $t$-structure with $\left(\Mod_R\right)_{\geq 0}\we\Mod_R^{\cn}$,
            the $\infty$-category of connective $R$-module spectra. See for
            example~\cite{ha}*{Proposition~1.4.3.6}. We call this the {\bf
            Postnikov $t$-structure}.
    \end{enumerate}
\end{example}

Condition (2) implies in fact that the mapping spaces $\Map_E(x,y[-1])$ are
contractible for $x\in E_{\geq 0}$ and $y\in E_{\leq 0}$. This is not generally
the case for the mapping spectra. Indeed, if $A$ is a
Grothendieck abelian category, then
$\pi_{0}\Map_{\Dscr(A)}(x[-n],y[-1])\iso\Ext^{n-1}_A(x,y)$ for $x,y\in A$.
(See~\cite{ha}*{Proposition~1.3.5.6}.)

\begin{lemma}
    The intersection $E_{\geq 0}\cap E_{\leq 0}$ is the full subcategory of
    $E_{\geq 0}$ consisting of discrete objects. Moreover, the intersection is
    an abelian category.
\end{lemma}

\begin{proof}
    See~\cite{ha}*{Warning~1.2.1.9} for the first statement
    and~\cite{bbd}*{Th\'eor\`eme~1.3.6} for the second.
\end{proof}

\begin{definition}
    The abelian category $E_{\geq 0}\cap E_{\leq 0}$ is called the \df{heart}
    of the $t$-structure $(E_{\geq 0},E_{\leq 0})$ on $E$, and is denoted
    $E^{\heartsuit}$.
\end{definition}

\begin{example}
    The hearts of the $t$-structures in Example~\ref{ex:t} are $A$ in (a), $A$
    in (b), and
    $\Mod_{\pi_0R}^{\heartsuit}$, the abelian category of right
    $\pi_0R$-modules, in (c).
\end{example}

The truncations $\tau_{\geq n}x$ and $\tau_{\leq n}x$ are functorial in the
sense that the inclusions $E_{\geq n}\rightarrow E$ and $E_{\leq n}\rightarrow
E$ admit right and left adjoints, respectively, by~\cite{ha}*{Corollary~1.2.1.6}. Let $\pi_nx=\tau_{\geq
n}\tau_{\leq n}x[-n]\in E^{\heartsuit}$. This functor is homological
by~\cite{bbd}*{Th\'eor\`eme~1.3.6}, meaning that there are long exact sequences
$$\cdots\rightarrow\pi_{n+1}z\rightarrow\pi_nx\rightarrow\pi_ny\rightarrow\pi_nz\rightarrow\pi_{n-1}x\rightarrow\cdots$$
in $E^\heartsuit$ whenever $x\rightarrow y\rightarrow z$ is a cofiber sequence in $E$.

\begin{definition}
    A $t$-structure $(E_{\geq 0},E_{\leq 0})$ on a stable $\infty$-category
    is \df{right separated} if $$\bigcap_{n\in\ZZ}E_{\leq n}=0.$$ \df{Left
    separated} $t$-structures are defined similarly. Left and right separated
    $t$-structures are called \df{non-degenerate} in~\cite{bbd}.
\end{definition}

\begin{definition}
    If $E$ is a stable $\infty$-category with a $t$-structure $(E_{\geq
    0},E_{\leq 0})$, we say that the $t$-structure is \df{bounded} if the inclusion
    $$E^b=\bigcup_{n\rightarrow\infty}E_{\geq -n}\cap E_{\leq n}\rightarrow E$$
    is an equivalence. Bounded $t$-structures are left and right separated.
\end{definition}

For example, the $t$-structure in Example~\ref{ex:t}(1) is bounded.

\begin{lemma}\label{lem:boundedsub}
    If $E$ is a stable $\infty$-category equipped with a
    $t$-structure $(E_{\geq 0},E_{\leq 0})$, then the full subcategory
    $E^b\subseteq E$ is stable and the $t$-structure on $E$ restricts to a bounded
    $t$-structure on $E^b$.
\end{lemma}

\begin{proof}
    Since $E^b\subseteq E$ is closed under translations (by part (1) of the
    definition of a $t$-structure), it is enough to show that it is closed
    under cofibers in $E$. Let $x\rightarrow y$ be a map in $E^b$ with cofiber
    $z$.  We must show that $z$ is bounded. We can assume first that $x$ and
    $y$ are in $E_{\geq 0}\cap E_{\leq n}$ for some $n>0$,
    in which case $z\in E_{\geq 0}$ since the inclusion $E_{\geq 0}\subseteq E$
    preserves and creates colimits. Moreover,
    $z\rightarrow x[1]\rightarrow y[1]$ is a fiber sequence in $E$ and
    $x[1]$ and $y[1]$ are in $E_{\leq n+1}$. Since the adjoint $E_{\leq
    n+1}\rightarrow E$ preserves limits, it follows that $z\in E_{\leq n+1}$.
    Hence, $z$ is bounded. To conclude, we must show that $E^b$ is closed under
    truncations in $E$, which will show that the $t$-structure on $E$ restricts
    to a $t$-structure on $E^b$. So, suppose that $w\in E^b$, and consider
    $\tau_{\geq 0}w$ in $E$. We have only to show that $\tau_{\geq 0}w$ is
    bounded above. Choose $m>0$ such that $\tau_{\leq m}w\we 0$. Such an $m$
    exists because $w$ is bounded. But, we now have $$\tau_{\leq m}\tau_{\geq
    0}w\we\tau_{\geq 0}\tau_{\leq m}w\we 0,$$ since the truncation functors
    commute by~\cite{ha}*{Proposition~1.2.1.10}
    or~\cite{bbd}*{Proposition~1.3.5}.
\end{proof}

\begin{lemma}\label{lem:exact}
    Suppose that $A=E^\heartsuit$ is the heart of a $t$-structure on a
    stable $\infty$-category $E$. If $0\rightarrow x\rightarrow y\rightarrow
    z\rightarrow 0$ is
    an exact sequence in $A$, then $x\rightarrow y\rightarrow z$ is a cofiber
    sequence in $E$.
\end{lemma}

\begin{proof}
    Note that using Lemma~\ref{lem:boundedsub} we can assume that $E$ is
    bounded. Let $w$ be the cofiber of $x\rightarrow y$ in $E$. Because $E_{\geq
    0}\subseteq E$ is a left adjoint, we can identify $w$ with the cofiber of
    $x\rightarrow y$ in $E_{\geq 0}$. As $E_{\geq
    0}\xrightarrow{\pi_0}E^{\heartsuit}\we A$ is a left adjoint, the sequence
    $x\rightarrow y\rightarrow\pi_0w\rightarrow 0$ is exact. But, it is also
    exact on the left by hypothesis, so that the cofiber $c$ of the natural map
    $w\rightarrow z$ has the property that $\pi_nc=0$ for all $n\in\ZZ$. Since
    bounded $t$-structures are non-degenerate, this implies $c\we 0$ and hence
    that $w\we z$, as desired.
%     There is a natural map
%     $w\rightarrow z$ induced from the fact that $x\rightarrow y\rightarrow z$
%     factors through $0$. The cofiber of $w\rightarrow z$ is an object $u$ such
%     that $\pi_nu=0$ for all $n\geq 0$. Since the $t$-structure on $E$ is
%     bounded, $u\in E_{\geq -n}\cap E_{\leq n}$ for some $n\geq 0$. But, since
%     $\pi_nu=0$, the cofiber sequence $\pi_nu[n]\rightarrow u\rightarrow\tau_{\leq
%     n-1}u$ shows that $u\we\tau_{\leq n-1}u$. Proceeding downwards, we find
%     eventually that $u\we\pi_{-n}u[-n]\we 0$.
\end{proof}

We leave the proof of the next lemma to the reader.

\begin{lemma}\label{lem:exactness}
    Let $E$ and $F$ be stable $\infty$-categories with $t$-structures.
    If $\varphi:E\rightarrow F$ is a right (resp. left) $t$-exact functor, then
    $\varphi$ induces a right (resp. left) exact functor $\pi_0\varphi:E^{\heartsuit}\rightarrow
    F^{\heartsuit}$.
\end{lemma}

% \marginpar{Do we use this?}
% \begin{lemma}
%     If $\Phi:E\rightarrow F$ is a $t$-exact functor, then $\Phi$ commutes with
%     truncation: the natural maps $\Phi(\tau_{\geq n}x)\rightarrow\tau_{\geq
%     n}\Phi(x)$ and $\tau_{\leq n-1}\Phi(x)\rightarrow\Phi(\tau_{\leq n-1}x)$ are
%     equivalences for all $x$.
% \end{lemma}
% 
% \begin{proof}
%     We can apply truncation $\tau_{\geq n}$ in $F$ to obtain a cofiber sequence $$\tau_{\geq
%     n}\Phi(\tau_{\geq n}x)\rightarrow\tau_{\geq n}\Phi(x)\rightarrow\tau_{\geq
%     n}\Phi(\tau_{\leq n-1}x).$$ However, $\Phi(\tau_{\leq n-1}x)$ is
%     $(n-1)$-coconnective by definition of $t$-exactness, so the right-most term is contractible.
%     The argument for $\tau_{\leq n-1}$ is the same.
% \end{proof}

% In Section~\ref{sub:abelian}, we define the $K$-theory $\K(A)$ of a small abelian category
% $A$ in a way that suits our purposes computationally, and which also agrees
% with all other definitions of nonconnective $K$-theory of an abelian category.
% Moreover, the connective cover $\K^{\cn}(A)$ is equivalent to the Quillen
% $K$-theory of $A$ as an exact category, where $A$ is viewed as an exact
% category with admissible sequences precisely those sequences
% \begin{equation}\label{seq1}
%     0\rightarrow x\rightarrow y\rightarrow z\rightarrow 0
% \end{equation} which are exact in $A$.
% See~\cite{quillen}.
Recall that if $A$ is an abelian category, then $\K_0(A)$ is the
Grothendieck group of $A$, which has generators $[x]$ for $x\in A$
and relations $[y]=[x]+[z]$ whenever $x,y,z$ fit into an exact sequence
$0\rightarrow x\rightarrow y\rightarrow z\rightarrow 0$.
Similarly, if $E$ is a small stable $\infty$-category, then $\K_0(E)$ is the
free abelian group on symbols $[x]$ for $x\in E$ modulo the relation
$[y]=[x]+[z]$ whenever $x\rightarrow y\rightarrow z$ is a cofiber sequence in
$E$.

It follows from Lemma~\ref{lem:exact} that there is a natural
map $\K_0(E^\heartsuit)\rightarrow\K_0(E)$ when $E$ is equipped with a
$t$-structure.

\begin{lemma}\label{lem:k0}
    If $E$ is a small stable $\infty$-category
    equipped with a bounded $t$-structure, then the natural map
    $\K_0(E^\heartsuit)\rightarrow\K_0(E)$ is an isomorphism.
\end{lemma}

\begin{proof}
    Using the boundedness of the $t$-structure, it is immediate that
    $\K_0(E^\heartsuit)\rightarrow\K_0(E)$ is surjective because every object
    of $E$ is a finite iterated extension of objects in $E^\heartsuit$.
    On the other hand, by assigning to $x\in E$ the sum
    $$\sum_{n\in\ZZ}(-1)^n[\pi_nx],$$ we obtain a map
    $\K_0(E)\rightarrow\K_0(E^\heartsuit)$, which splits the surjection.
\end{proof}

\subsection{Induced $t$-structures on ind-completions and localizations}\label{sub:induced}

We give several results about $t$-structures on stable $\infty$-categories.
Some of these, especially the equivalence of conditions (i) through (iv) in Proposition~\ref{prop:tlocalization},
have not, as far as we are aware, been proved before either for
$\infty$-categories or for triangulated categories, so we treat the subject
in greater detail than is strictly needed for the rest of the paper.
However, there is some overlap between this section and~\cite{sag}*{Appendix~C}
and~\cite{hennion-porta-vezzosi}.

A $t$-structure $(E_{\geq 0},E_{\leq 0})$ on a stable $\infty$-category $E$ is
\df{bounded below} if the natural map $$E^-=\bigcup_{n\in\ZZ}E_{\geq n}\rightarrow
E$$ is an equivalence and
\df{right complete} if the natural map
$$E\rightarrow\lim\left(\cdots\rightarrow E_{\geq m}\xrightarrow{\tau_{\geq m+1}} E_{\geq
m+1}\rightarrow\cdots\right)$$ is an equivalence.
{\bf Bounded above} and {\bf left complete} $t$-structures
are defined similarly. 
A bounded below $t$-structure is right separated as is a right complete
$t$-structure. Neither converse is true in general.

The following definitions were introduced in~\cite{ha}*{Section 1}.
A $t$-structure on a stable presentable $\infty$-category $E$ is
\df{accessible} if $E_{\geq 0}$ is presentable. A $t$-structure on a stable
presentable $\infty$-category $E$ is
\df{compatible with filtered colimits} if $E_{\leq 0}$ is closed under
filtered colimits in $E$.

\begin{example}
    Example~\ref{ex:t}(a) is bounded (above and below). It is neither left or
    right complete, nor is it accessible or compatible with filtered colimits,
    as these notions are reserved for presentable $\infty$-categories.
    Examples~\ref{ex:t}(b) and (c) are right complete, accessible, and compatible
    with filtered colimits. 
\end{example}

The following proposition also appears in~\cite{sag}*{Lemma C.2.4.3}.

\begin{proposition}\label{lem:textension}
    Suppose that $E$ is a small stable $\infty$-category
    with a $t$-structure. Then, $\Ind(E_{\geq 0})\subseteq\Ind(E)$ determines
    the non-negative part of an accessible $t$-structure on $\Ind(E)$ which is
    is compatible with filtered colimits and such that the
    inclusion functor $E\rightarrow\Ind(E)$ is $t$-exact. Moreover, if the $t$-structure on $E$ is bounded
    below, then $\Ind(E)$ is right complete.
\end{proposition}

\begin{proof}
    The functor $\Ind(E_{\geq 0})\rightarrow\Ind(E)$ is fully faithful
    by~\cite{htt}*{Proposition~5.3.5.11}, and we let $\Ind(E)_{\geq 0}$ denote the
    essential image. Similarly, let $\Ind(E)_{\leq -1}$ denote the essential
    image of the fully faithful functor $\Ind(E_{\leq -1})\rightarrow\Ind(E)$.
    We claim that this pair of subcategories defines a $t$-structure on $\Ind(E)$.
    Condition (1) of Definition~\ref{def:t} is immediate. Suppose that
    $x\we\colim_{i\in I} x_i$ is in $\Ind(E)_{\geq 0}$, where each $x_i$ is in
    $E_{\geq 0}$, and let $y\we\colim_{j\in J} y_j$ be in $\Ind(E)_{\leq -1}$, with each $y_j\in E_{\leq
    -1}$. Then, by definition of the ind-completion of $E$,
    $$\Map_{\Ind(E)}(x,y)\we\lim_i\colim_j\Map_E(x_i,y_j),$$ which is
    contractible since each $\Map_E(x_i,y_j)$ is contractible. Hence, (2)
    holds. To verify condition (3), note that if $x\we\colim_{i\in I}x_i$ is a
    filtered colimit of objects $x_i\in E$, then $$\colim_i\tau_{\geq
    0}x_i\rightarrow x\rightarrow\colim_i\tau_{\leq -1}x_i$$ is a cofiber
    sequence since cofiber sequences commute with colimits. Hence, (3) holds.

    To see that the $t$-structure is compatible with filtered colimits, note
    that $y\in\Ind(E)_{\leq -1}$ if and only if $\Map_{\Ind(E)}(x,y)\we 0$ for
    all $x\in\Ind(E)_{\geq 0}\we\Ind(E_{\geq 0})$.
    However, this latter condition holds if and only if
    $\Map_{\Ind(E)}(x,y)\we 0$ for all $x\in E_{\geq 0}$ since $\Ind(E_{\geq 0})$ is
    generated by $E_{\geq 0}$ under filtered colimits. Since the objects $x\in E_{\geq 0}\subseteq E$
    are compact, this condition is closed under filtered colimits in $y$, as
    desired.

    By construction, the functor $E\rightarrow\Ind(E)$ is $t$-exact, and the
    $t$-structure on $\Ind(E)$ is accessible as $\Ind(E)_{\geq
    0}\we\Ind(E_{\geq 0})$ is presentable.

    To finish the proof, we first show right separatedness.
    Suppose that $y$ is an object of $\bigcap_{n\in\ZZ}\Ind(E)_{\leq n}.$
    Since the objects of $E$ are compact generators for
    $\Ind(E)$, it is enough to show that the mapping spaces $\Map_{\Ind(E)}(x,y)\we 0$
    for all $x\in E$. Fix $x\in E$. We have for all $n$ a natural equivalence
    $$\Map_{\Ind(E)}(x,y)\we\Map_{\Ind(E)}(\tau_{\leq n}x,y).$$
    However, since the $t$-structure on $E$ is bounded below, $\tau_{\leq
    n}x\we 0$ for $n$ sufficiently small. Therefore, $\Map_{\Ind(E)}(x,y)\we
    0$. Hence, $y\we 0$.

    Since $\Ind(E)_{\leq 0}\rightarrow\Ind(E)$ is closed under finite
    coproducts and filtered colimits it is closed under countable coproducts.
    Therefore, it follows by
    the right complete version of~\cite{ha}*{Proposition~1.2.1.19} that
    $\Ind(E)$ is right separated if and only if it is right complete. This
    completes the proof.
\end{proof}

We will call the $t$-structure on $\Ind(E)$ constructed in
Proposition~\ref{lem:textension} the \df{induced} $t$-structure.
The proof of the proposition does not extend to show that bounded above $t$-structures on $E$
induce left complete $t$-structures on $\Ind(E)$. The obstruction is that the
inclusion of $\Ind(E)_{\geq n}$ is a left adjoint rather than a right
adjoint.

\begin{corollary}\label{cor:idemcomplete}
    Let $E$ be a small stable $\infty$-category with a bounded $t$-structure.
    Then, $E$ is idempotent complete.
\end{corollary}

\begin{proof}
    Let $F$ be the idempotent completion of $E$. Equivalently,
    $F\we\Ind(E)^\omega$, the full subcategory of compact objects of $\Ind(E)$.
    We claim that the $t$-structure on $E$ extends to a bounded $t$-structure
    on $F$. It is enough to check that the truncation functors
    $\tau_{\leq 0}$ and $\tau_{\geq 0}$ on $\Ind(E)$ preserve compact objects. But,
    if $x\in F$ is a summand of $y\in E$, it follows that $\tau_{\leq 0}x$ is a
    summand of $\tau_{\leq 0}y$, and similarly for $\tau_{\geq 0}x$. This
    proves that the $t$-structure on $\Ind(E)$ restricts to a bounded $t$-structure on
    $F$. The heart $F^{\heartsuit}$ must be the
    idempotent completion of $E^{\heartsuit}$. But, since abelian categories
    are idempotent complete,
    $E^\heartsuit\rightarrow F^{\heartsuit}$ is an equivalence. Hence, by
    Lemma~\ref{lem:k0}, $\K_0(E)\rightarrow\K_0(F)$ is an isomorphism. It
    follows from Thomason's classification of dense subcategories of
    triangulated categories that $E\we F$. See~\cite{thomason-classification}*{Theorem~2.1}.

    We can also avoid appealing to Thomason's result as follows. Given an
    object $x\in F$ and an integer $n\geq 0$, we say that $x$ has amplitude at most $n$ if there is an
    interval $[a,b]$ with $b-a\leq n$ and such that $\pi_ix=0$ for
    $i\notin[a,b]$. As the $t$-structure on $F$ is bounded, every object has
    amplitude at most $n$ for some integer $n\geq 0$. Since
    $E^\heartsuit\we F^\heartsuit$, if $x$ has amplitude at most $0$, then
    $x\in E$. We proceed by induction on the amplitude.
    Assume that for every object $y$ of $F$ of amplitude at most $n-1$, where
    $n\geq 1$, we have that $y$ is in the subcategory $E$. Fix $x\in F$ an
    object of amplitude at most $n$ and assume, possibly by suspending, that $\pi_ix=0$ for $i\notin[0,n]$. Consider the
    fiber sequence $\tau_{\geq 1}x\rightarrow x\rightarrow\pi_0x$. The objects
    $\tau_{\geq 1}x$ and $\pi_0x$ have amplitude at most $n-1$ and hence they
    are in $E$. But, $x$ is the fiber of $\pi_0x\rightarrow\tau_{\geq 1}x[1]$
    and $E\subseteq F$ is full. Since $E$ is stable, $x$ is in $E$, as desired.
\end{proof}

In the rest of this section, we establish an important device for checking
when a $t$-exact fully faithful functor $i:E\rightarrow F$ of small stable
$\infty$-categories induces a $t$-structure on the cofiber $G=\widetilde{F/E}$ in
$\Cat_{\infty}^{\perf}$, the $\infty$-category of small idempotent complete
stable $\infty$-categories and exact functors. Recall that $G$
is equivalent to the idempotent completion of the Verdier localization of $F$
by $E$ (see~\cite{bgt1}*{Proposition~5.13}).
We begin with a couple of easy lemmas.

\begin{lemma}\label{lem:inducedexact}
    If $i:E\rightarrow F$ is a $t$-exact (resp. right $t$-exact, resp. left
    $t$-exact) functor of stable $\infty$-categories equipped with
    $t$-structures, then the
    induced functor $i^*:\Ind(E)\rightarrow\Ind(F)$ is $t$-exact (resp. right
    $t$-exact, resp. left $t$-exact) with
    respect to the induced $t$-structures on $\Ind(E)$ and $\Ind(F)$.
\end{lemma}

\begin{proof}
    The exactness of $i^*$ is immediate as it preserves all small colimits and
    hence finite limits since $\Ind(E)$ and $\Ind(F)$ are stable.
    Because $\Ind(E)_{\geq 0}\we\Ind(E_{\geq 0})$ and
    $\Ind(F)_{\geq 0}\we\Ind(F_{\geq 0})$, it is immediate that $i^*:\Ind(E)\rightarrow\Ind(F)$ is
    right $t$-exact if $i$ is. The same holds for left $t$-exactness.
%     The description of the objects of $\Ind(E)_{\leq 0}$ as
%     filtered colimits in $\Ind(E)$ of objects in $E_{\leq 0}$ says that upon
%     applying $i$ and using that $i^*:\Ind(E)\rightarrow\Ind(F)$ preserves
%     colimits we get that $i^*$ is left $t$-exact as well.
\end{proof}

\begin{lemma}\label{lem:heartclosure}
    Let $i:E\rightarrow F$ be a $t$-exact fully faithful functor of
    stable $\infty$-categories equipped with $t$-structures. Then, the natural map
    $$\Ind(E)^{\heart}\rightarrow\Ind(F)^\heart\cap\Ind(E)$$ is an exact equivalence
    of abelian categories.
\end{lemma}

\begin{proof}
    Let $x\in\Ind(F)$ be an object of the intersection. Write $x=i^*y$ for some
    $y\in\Ind(E)$ (which is unique up to equivalence). The fact that $i^*$ is
    $t$-exact and fully
    faithful implies that $\tau_{\geq 1}y\we 0$ and $\tau_{\leq -1}y\we 0$. In particular, $y$ is
    contained in $\Ind(E)^\heart$. It follows that the map in the lemma is
    essentially surjective. That the map is fully faithful follows from the fact that
    $\Ind(E)\rightarrow\Ind(F)$ is fully faithful, while exactness again
    follows from Lemma~\ref{lem:inducedexact}.
\end{proof}

Recall from~\cite{ha}*{Proposition~1.4.4.11} that if $\Cscr$ is a stable
presentable $\infty$-category and $\Cscr'\subseteq\Cscr$ is a full presentable
subcategory closed under colimits and extensions in $\Cscr$, then
$\Cscr'\we\Cscr_{\geq 0}$ for some accessible $t$-structure on $\Cscr$. We will
say that the $t$-structure $(\Cscr_{\geq 0},\Cscr_{\leq 0})$ on $\Cscr$ is the
$t$-structure \df{generated} by $\Cscr'\subseteq\Cscr$. This
provides a way for defining many $t$-structures on stable presentable
$\infty$-categories. Note that if $\Cscr\we\Ind(E)$, where $E$ is equipped with
a $t$-structure $(E_{\geq 0},E_{\leq 0})$, then the induced $t$-structure on
$\Ind(E)$ is a special case of this phenomenon: it is generated by $\Ind(E_{\geq 0})$.

\begin{definition}
    Let $A\subseteq B$ be an exact fully faithful functor of abelian
    categories. We identify $A$ with its essential image in $B$.
    Say that $A$ is a \df{weak Serre subcategory} of $B$ if $A$ is closed under
    extensions in $B$. We say that $A$ is a \df{Serre subcategory} of $B$ (or a
    \df{localizing subcategory} of $B$) if
    $A$ is a weak Serre subcategory and $A$ is additionally closed under taking
    subobjects and quotient objects in $B$.
\end{definition}

\begin{example}
    Let $R$ be a right coherent ring. Then, the category
    $\Mod_R^{\heartsuit,\omega}$ of finitely presented (discrete) right
    $R$-modules is an abelian subcategory of $\Mod_R^\heartsuit$. It is always
    weak Serre, but it is Serre if and only if $R$ is right noetherian.
\end{example}

\begin{lemma}\label{lem:wserre}
    Let $E\rightarrow F$ be a $t$-exact fully faithful functor of stable
    $\infty$-categories equipped with $t$-structures. Then, the induced map
    $E^{\heartsuit}\rightarrow F^{\heartsuit}$ exhibits $E^{\heartsuit}$ as a
    weak Serre subcategory of $F^{\heartsuit}$.
\end{lemma}

\begin{proof}
    The fact that $E^\heartsuit\rightarrow F^{\heartsuit}$ is exact and fully
    faithful follows from Lemma~\ref{lem:exactness} and the fully faithfulness of
    $E\rightarrow F$. To check that $E^\heartsuit$ is closed under extensions
    in $F^\heartsuit$, consider an exact sequence $0\rightarrow x\rightarrow
    y\rightarrow z\rightarrow 0$ where $x,z\in E^{\heartsuit}$ and $y\in
    F^{\heartsuit}$. Then, by Lemma~\ref{lem:exact}, $x\rightarrow y\rightarrow
    z$ is a cofiber sequence in $F$. Hence, we can rewrite $y$ as the fiber of
    $z\rightarrow x[1]$. Since $E\rightarrow F$ is fully faithful and preserves
    fibers, it follows that $y$ is in the essential image of $E\rightarrow F$,
    as desired. We conclude by using Lemma~\ref{lem:heartclosure}.
\end{proof}

The first draft of this paper contained conditions (i) through (iv) of the next
proposition. Benjamin Hennion pointed out another condition, (v) below, which is
shown to be equivalent to condition (iii)
in~\cite{hennion-porta-vezzosi}*{Proposition~A.5}.

\begin{proposition}\label{prop:tlocalization}
    Let $i:E\rightarrow F$ be a $t$-exact fully faithful functor of
    stable
    $\infty$-categories equipped with bounded
    $t$-structures, and let $j:F\rightarrow G$ be the cofiber in $\Cat_{\infty}^{\perf}$.
    Provide $\Ind(G)$ with the accessible
    $t$-structure generated by the smallest extension-closed cocomplete
    subcategory of $\Ind(G)$ containing the image of $F_{\geq 0}$, and equip
    $\Ind(E)$ and $\Ind(F)$ with the induced $t$-structures of
    Proposition~\ref{lem:textension}. The
    following are equivalent:
    \begin{enumerate}
        \item[{\rm (i)}]   the essential image of the embedding
            $i^\heart:E^\heart\rightarrow F^\heart$ is a Serre subcategory of
            $F^\heart$;
        \item[{\rm (ii)}] the $t$-structure on $\Ind(G)$ restricts to a
            $t$-structure on $G$ such that $j:F\rightarrow G$ is $t$-exact;
        \item[{\rm (iii)}]   the induced functor $j^*:\Ind(F)\rightarrow\Ind(G)$ is $t$-exact;
        \item[{\rm (iv)}]   the essential image of the embedding
            $\Ind(E)^\heart\rightarrow\Ind(F)^\heart$ is a Serre subcategory of
            $\Ind(F)^\heart$;
        \item[{\rm (v)}] the counit map $i^*i_*x\rightarrow x$ induces a
            monomorphism $\pi_0(i^*i_*x)\rightarrow\pi_0(x)$ in
            $\Ind(F)^\heartsuit$ for every object $x$ of $\Ind(F)$, where $i_*$
            is the right adjoint of $i^*:\Ind(E)\rightarrow\Ind(F)$.
    \end{enumerate}
    If these conditions hold, then the $t$-structure on $G$ in (ii) is bounded.
\end{proposition}

\begin{proof}
    Assume (i). Write $G'=F/E$ for the Verdier quotient of $F$ by $E$. In particular, $G$ is
    the idempotent completion of $G'$. We will construct a bounded $t$-structure on
    $G'$ such that the functors $F\rightarrow G'$ and $G'\subseteq
    G\subseteq\Ind(G)$ are $t$-exact.
    By Corollary~\ref{cor:idemcomplete}, $G'$ will be idempotent complete. This
    will establish (ii).

    Let $L:F\rightarrow G'$ denote the
    quotient functor. We define $\tau_{\geq 0}Lx=L\tau_{\geq 0}x$, and
    similarly $\tau_{\leq 0}Lx=L\tau_{\leq 0}$. It follows that (1) and (3) from
    Definition~\ref{def:t} hold trivially. Now, consider $\Hom_{G'}(Lx,Ly[-1])$, where
    $x\in F_{\geq 0}$ and $y\in F_{\leq 0}$. Pick $f\in\Hom_{G'}(Lx,Ly[-1])$.
    We can represent $f$ by a zig-zag $x\leftarrow z\rightarrow y[-1]$, where
    the cofiber $c$ of $x\leftarrow z$ is in $E$. Now, consider the following
    diagram
    \begin{equation*}
        \xymatrix{
            \tau_{\geq 0}z\ar[r]\ar[d]  &   z\ar[r]\ar[d]   &   \tau_{\leq -1}z\ar[d]\\
            \tau_{\geq 0}x\ar[r]\ar[d]  &   x\ar[r]\ar[d]   &   \tau_{\leq -1}x\ar[d]\\
            \tau_{\geq 0}c\ar[r] &   c\ar[r]  &   \tau_{\leq -1}c
        }
    \end{equation*}
    of truncation sequences. (Warning: while the horizontal sequences are always
    cofiber sequences, only the central vertical sequence is a cofiber sequence
    in general.) The fact that $y\in F_{\leq 0}$ means that the map
    $z\rightarrow y[-1]$ factors through $\tau_{\leq -1}z$. Now, the fact that
    $x$ is connective means that $\pi_{-n}z\in E$ for all $n\geq 1$. This is
    where we use the fact that $E^{\heartsuit}$ is a Serre subcategory of
    $F^{\heartsuit}$, to ensure that the quotient $\pi_{-1}z$ of $\pi_0c$ is
    also in $E$. In particular, $\tau_{\geq 0}z\rightarrow\tau_{\geq 0}x\we x$
    has cofiber in $E$ (though it is not in general $\tau_{\geq 0}c$).
    The commutative diagram
    $$\xymatrix{
        &z\ar[dl]\ar[dr]&\\
        x&\tau_{\geq 0}z\ar[u]\ar[r]^0\ar[l]\ar[d]_=&y[-1]\\
        &\tau_{\geq 0}z\ar[ul]\ar[ur]_0&
    }$$
    shows that $f$ is nullhomotopic, which completes the construction of a
    bounded $t$-structure on $G'$, which after the fact is idempotent complete,
    so $G'\we G$.

    The inclusion
    $G\rightarrow\Ind(G)$ is evidently right $t$-exact with respect to the
    $t$-structure defined above on $G'\we G$ and the given $t$-structure on $\Ind(G)$.
    Let $x\in F_{\leq
    -1}$. To see left $t$-exactness, it suffices to check that $\Map_{\Ind(G)}(y,Lx)\we 0$ for all
    $y\in\Ind(G)_{\geq 0}$. But, since $\Ind(G)_{\geq 0}$ is generated under
    filtered colimits and extensions by
    images of the objects $z\in F_{\geq 0}$, this result follows from the
    computation above.
    Finally, by construction, $F\rightarrow G'\we G$ is $t$-exact.
    This completes the proof that (i) implies (ii).

    Assume (ii). By definition of the $t$-structure on $\Ind(G)$, the
    localization functor $L:\Ind(F)\rightarrow\Ind(G)$ is right $t$-exact. Let
    $x\in\Ind(F)_{\leq -1}$. We must check that $\Map_{\Ind(G)}(y,Lx)\we 0$ for
    all $y\in\Ind(G)_{\geq 0}$. To do so, it is enough to check this for $y$ of
    the form $Lz$ for some $z\in F_{\geq 0}$. However, we can write
    $x\we\colim_I x_i$ for a filtered $\infty$-category $I$ and some $x_i\in
    F_{\leq -1}$ since we use the $t$-structure on $\Ind(F)$ induced by
    $\Ind(F_{\geq 0})$. Hence,
    $$\Map_{\Ind(G)}(Lz,Lx)\we\colim_I\Map_{\Ind(G)}(Lz,Lx_i)$$ since $L$
    commutes with colimits and $Lz$ is compact in $\Ind(G)$. As $Lz\in
    G_{\geq 0}$ and $Lx_i\in G_{\leq -1}$, (ii) shows that each mapping
    space in the colimit on the right is contractible, as desired.
    Hence, (ii) implies (iii).

%     Condition (iii) will follow from (ii) by Lemma~\ref{lem:inducedexact} if we
%     show that the $t$-structure described defined on $\Ind(G)$ is the induced
%     $t$-structure. However, since the $t$-structure is determined by
%     $\Ind(G)_{\geq 0}$, it is enough to show that the natural map $\Ind(G_{\geq
%     0})\rightarrow\Ind(G)_{\geq 0}$ is an equivalence. It is en

    To see that (iii) implies (iv), note first that the $t$-structures on $\Ind(E)$
    and $\Ind(F)$ are right complete and hence right separated by
    Proposition~\ref{lem:textension}. 
    It follows from Lemma~\ref{lem:wserre} that
    $\Ind(E)^{\heartsuit}\subseteq\Ind(F)^{\heartsuit}$ is weak Serre.
    Denote by $i^*:\Ind(E)\rightarrow\Ind(F)$
    the induced functor, and let $x\subseteq i^*y$ be a subobject, where
    $x\in\Ind(F)^{\heartsuit}$ and $y\in\Ind(E)^{\heartsuit}$. Then, by
    $t$-exactness, $j^*x\subseteq j^*i^*y=0$, so $j^*x=0$ in
    $\Ind(G)^\heartsuit$. It
    follows that $j^*x$ is in $\Ind(F)^\heartsuit$ and in $\Ind(E)$. Hence, by
    Lemma~\ref{lem:heartclosure}, $x=i^*z$ for some $z\in\Ind(E)^{\heartsuit}$.
    Thus, (iv) holds.

    Now, suppose that (iv) holds, and let $x\subseteq iy$ for some $y\in
    E^{\heartsuit}$ and $x\in F^{\heartsuit}$. Then, $x\we i^*z$ for some
    $z\in\Ind(E)^\heartsuit$ by hypothesis (iv). However, as an object $z$ of
    $\Ind(E)$ is compact if and only if $i^*z$ is compact, it follows that in
    fact $z\in E$. Hence, (iv) implies (i).

    The equivalence of (iii) and (v)
    is~\cite{hennion-porta-vezzosi}*{Proposition~A.5}.

    Finally, the boundedness of the $t$-structure on $G$ assuming that (ii)
    holds follows from the boundedness of the $t$-structure on $F$, the
    essential surjectivity of $j$ up to retracts, and the $t$-exactness of $j$.
\end{proof}

\subsection{Excisive squares and adjointability}\label{sub:mv}

Consider a commutative square
\begin{equation}\label{eq:mv1}
    \xymatrix{
        E\ar[r]\ar[d]   &   F\ar[d]\\
        G\ar[r]         &   H
    }
\end{equation}
of small idempotent complete stable $\infty$-categories and fully faithful
functors. In this section, we
establish general conditions (Lemma~\ref{lem:adjff},
Proposition~\ref{prop:tatelike}, and
Theorem~\ref{thm:mv}) which guarantee that the induced map
\begin{equation}\label{eq:mv2}
    \xymatrix{
        \K(E)\ar[r]\ar[d]   &   \K(F)\ar[d]\\
        \K(G)\ar[r]         &   \K(H)
    }
\end{equation}
is a pushout square of spectra and hence gives a long exact
sequence
$$\cdots\rightarrow\K_n(E)\rightarrow\K_n(F)\oplus\K_n(G)\rightarrow\K_n(H)\rightarrow\K_{n-1}(E)\rightarrow\cdots$$
of $K$-groups. We check these conditions in two situations: for Tate
objects (as studied in~\cite{hennion}) later in this section and for
$t$-structures in the proof of Theorem~\ref{thm:minus1}. We include the
former for completeness, while the latter is what we need later in the paper. We begin with a
standard lemma about pushouts and cofibers.

\begin{lemma}\label{lem:cofiberspushout}
    Suppose that
    $$\xymatrix{
        M\ar[r]^f\ar[d]   &   N\ar[d]\\
        P\ar[r]^g         &   Q
    }$$ is a commutative diagram in a stable $\infty$-category. Then, the induced map
    $\mathrm{cofib}(f)\rightarrow\mathrm{cofib}(g)$ is an equivalence if and
    only if the square is a pushout square.
\end{lemma}

\begin{proof}
    If the square is a pushout square, then the horizontal cofibers are
    equivalent (see~\cite{htt}*{Lemma~4.4.2.1}). This is true in any
    $\infty$-category with pushouts and a terminal object. So, assume that
    $\mathrm{cofib}(f)\rightarrow\mathrm{cofib}(g)$ is an equivalence.
    Let $S$ be the pushout of $P$ and $N$ over $M$, and let $T$ be an arbitrary
    spectrum. Consider the commutative diagram
    $$\xymatrix{
        \Map(\mathrm{cofib}(g),T)\ar[r]\ar[d]   &
        \Map(Q,T)\ar[r]\ar[d]   & \Map(P,T)\ar[d]\\
        \Map(\mathrm{cofib}(f),T)\ar[r]   &   \Map(S,T)\ar[r]   &
        \Map(P,T)
    }$$
    of fiber sequences of mapping spaces. The outer vertical arrows are equivalences by
    hypothesis. In general, this does not in general let us conclude that the
    middle vertical arrow is an equivalence. However, because these are fiber
    sequences of infinite loop spaces, the
    long exact sequence in homotopy groups shows that
    $\Map_{\Sp}(Q,T)\rightarrow\Map_{\Sp}(S,T)$ is an equivalence for all $T$. Hence,
    $S\rightarrow Q$ is an equivalence.
\end{proof}

Let $\widetilde{F/E}$ and $\widetilde{H/G}$ denote the cofibers in
$\Cat_{\infty}^{\perf}$ of the horizontal maps
in~\eqref{eq:mv1}. Then, by localization in $K$-theory, there is a commutative diagram
\begin{equation}\label{eq:mv3}
    \xymatrix{
    \K(E)\ar[r]\ar[d]   &   \K(F)\ar[d]\ar[r]   &   \K(\widetilde{F/E})\ar[d]\\
    \K(G)\ar[r]         &   \K(H)\ar[r] &   \K(\widetilde{H/G})
    }
\end{equation}
in which the horizontal sequences are cofiber sequences. Hence, using
Lemma~\ref{lem:cofiberspushout}, in order to check that~\eqref{eq:mv2} is a
pushout square it suffices (and is necessary) to see that
$\K(\widetilde{F/E})\rightarrow\K(\widetilde{H/G})$ is an equivalence.
This occurs in particular when $F/E\rightarrow H/G$ is an equivalence after
idempotent completion.

\begin{definition}
    Say that a square as in~\eqref{eq:mv1} is an \df{excisive square} if
    $\widetilde{F/E}\rightarrow \widetilde{H/G}$ is an equivalence.
\end{definition}

\begin{remark}
    It is easy to check using the full faithfulness of $F/E\rightarrow H/G$
    that an excisive square is cartesian, so that
    $E\rightarrow F\cap G$ is an equivalence.
\end{remark}

\begin{example}
    \begin{enumerate}
        \item[(a)]   If~\eqref{eq:mv1} is a pushout square, then it is an
            excisive square.
        \item[(b)]   Suppose that $E=0$ and that $H=\langle F,G\rangle$ is a
            \df{semiorthogonal decomposition} of $H$. Recall that this means
            that $F$ and $G$ are full stable subcategories of $H$ such that
            \begin{enumerate}
                \item[(i)]   $F\cap G=0$,
                \item[(ii)]   every object $x\in H$ can be written in a cofiber
                    sequence $y\rightarrow x\rightarrow z$ where $y\in G$ and
                    $z\in F$, and
                \item[(iii)]  the mapping spaces $\Map_H(y,z)$ vanish for all $y\in G$ and all $z\in F$.
            \end{enumerate}
            Under these conditions, it is easy to check by hand that the induced map
            $F\rightarrow H/G$ is an equivalence, which induces a (split) localization
            sequence $\K(G)\rightarrow\K(H)\rightarrow\K(F)$. For more details,
            see~\cite{bgt1}.
    \end{enumerate}
\end{example}

\begin{remark}
    Note that despite conditions (i) and (ii), $H$ is not generally the
    coproduct in $\Cat_{\infty}^{\perf}$ of $F$ and $G$. The coproduct is $F\oplus
    G$, and in that category one has the additional criterion that $\Map_H(z,y)=0$
    for $y\in G$ and $z\in F$. That is, one has an {\bf orthogonal decomposition}.
    This is a much stronger hypothesis, but it is rarely satisfied in situations of
    interest. For example, Be\u{\i}linson's decomposition of
    $\Dscr^b(\PP^1_k)\we\langle \Oscr,\Oscr(1)\rangle$ gives a semiorthogonal
    decomposition of $\Dscr^b(\PP^1_k)$ which is not orthogonal
    (see~\cite{huybrechts}*{Corollary~8.29}).
\end{remark}

In Proposition~\ref{prop:tatelike} below,
we give a criterion for checking that certain squares~\eqref{eq:mv1} are
excisive squares. Our arguments are based on those of Benjamin
Hennion~\cite{hennion}*{Proposition~4.2}, which in turn are based on those of
Sho Saito~\cite{saito}. We need some preliminaries first.

\begin{definition}[See~\cite{ha}*{Definition~4.7.5.13}]\label{def:adjointable}
        Consider a commutative diagram
        $$\xymatrix{
            \Escr\ar[r]^{i^*}\ar[d]_{p^*}   &   \Fscr\ar[d]^{q^*}\\
            \Gscr\ar[r]^{j^*}               &   \Hscr
        }$$ of $\infty$-categories such that $i^*$ and $j^*$ admit right
        adjoints $i_*$ and $j_*$, respectively. Fix a natural equivalence
        $\alpha:j^*p^*\we q^*i^*$ (necessarily unique up to homotopy). The diagram is \df{right
        adjointable} if the natural map $p^* i_*\rightarrow
        j_*j^*p^*i_*\xrightarrow{\alpha} j_*q^*i^*i_*\rightarrow j_* q^*$ is an
        equivalence.
\end{definition}

\begin{remark}
    In general, the right adjointability of a diagram as in
    Definition~\ref{def:adjointable} is not equivalent to
    the adjointability of the transpose diagram.
\end{remark}

\begin{proposition}\label{prop:adjointability}
    Consider a commutative diagram
    \begin{equation*}
        \xymatrix{
            E\ar[r]^i\ar[d]_p   &   F\ar[d]^q\\
            G\ar[r]^j           &   H
        }
    \end{equation*}
    of fully faithful exact functors of stable idempotent complete $\infty$-categories. The following
    conditions are equivalent:
    \begin{enumerate}
        \item[\emph{(1)}]   the induced commutative diagram
            \begin{equation*}
                \xymatrix{
                    \Ind(F)\ar[r]^{f^*}\ar[d]_{q^*} &   \Ind(F/E)\ar[d]^{r^*}\\
                    \Ind(H)\ar[r]^{g^*}             &   \Ind(H/G)
                }
            \end{equation*}
            of stable presentable $\infty$-categories is right adjointable,
            where $f:F\rightarrow F/E$ and $g:H\rightarrow H/G$ are the quotient
            maps and $r:F/E\rightarrow H/G$ is the induced map on the quotients;
        \item[\emph{(2)}]  for any $x\in\Ind(F)$, if $i_*x\we 0$ in $\Ind(E)$, then
            $j_*q^*x\we 0$ in $\Ind(G)$, where $i_*$ and $j_*$ are right
            adjoint to $i^*$ and $j^*$, respectively.
    \end{enumerate}
\end{proposition}

The functors $f^*,g^*,i^*,\ldots$ all preserve colimits and hence admit right
adjoints which we will denote by $f_*,g_*,i_*,\ldots$
For the proof and the remainder of the section, we will make use of the
cofiber sequences $i^*i_*x\rightarrow x\rightarrow f_*f^*x$ in $\Ind(F)$ when
$x\in\Ind(F)$ and $\Ind(E)\xrightarrow{i^*}\Ind(F)\xrightarrow{f^*}\Ind(F/E)$
is a localization sequence.

\begin{proof}
    Assume (1). Choose $x\in\Ind(F)$ such that $i_*x\we 0$. Then, $x\we
    f_*f^*x$. Now, consider the cofiber sequence
    $$j^*j_*q^*f_*f^*x\rightarrow q^*f_*f^*x\rightarrow g_*g^*q^*f_*f^*x\we
    g_*r^*f^*f_*f^*x\we g_*r^*f^*x.$$ Adjointability means that the map
    $q^*f_*f^*x\rightarrow g_*r^*f^*x$ is an equivalence so that
    $j^*j_*q^*f_*f^*x\we 0$. Since $j^*$ is fully faithful, this means that
    $j_*q^*f_*f^*x\we j_*q^*x\we 0$, as desired.
    
    We prove (2) implies (1).
    Let $y\in\Ind(F/E)$. Then, the counit map $f^*f_*y\rightarrow y$ is an
    equivalence. Set $x=f_*y$. Consider the commutative diagram
    \begin{equation*}
        \xymatrix{
            q^*i^*i_*x\ar[r]\ar[d]   &   q^*x\ar[r]\ar[d]    & q^*f_*f^*x\ar[d]\\
            j^*j_*q^*x\ar[r]        &   q^*x\ar[r]          &   g_*g^*q^*x
        }
    \end{equation*}
    of cofiber sequences in $\Ind(G)$. Since $i_*x\we i_*f_*y\we 0$, we have
    that $j_*q^*x\we 0$ by hypothesis (3). Hence, both terms on the left vanish,
    so the map $q^*f_*f^*x\rightarrow g_*g^*q^*x$ is an equivalence. But,
    $g_*g^*q^*x\we g_*r^*f^*x$. In particular, $g_*r^*y\we q^*f_* y$ for all
    $y\in\Ind(F/E)$.
\end{proof}

\begin{lemma}\label{lem:adjff}
    Suppose that a commutative diagram
    \begin{equation*}
        \xymatrix{
            E\ar[r]^i\ar[d]^p   &   F\ar[d]^q\\
            G\ar[r]^j           &   H
        }
    \end{equation*}
    of fully faithful functors of stable idempotent complete
    $\infty$-categories satisfies the equivalent conditions of
    Proposition~\ref{prop:adjointability}. Then, $F/E\rightarrow H/G$ is fully
    faithful.
\end{lemma}

\begin{proof}
    We adopt the notation of the proof of the previous proposition.
    We show that the natural map
    $\Map_{\Ind(F/E)}(f^*x,f^*y)\rightarrow\Map_{\Ind(H/G)}(r^*f^*x,r^*f^*y)$
    is an equivalence for all $x,y\in F/E$. There are natural equivalences,
    \begin{align*}
        \Map_{\Ind(H/G)}(r^*f^*x,r^*f^*y)&\we\Map_{\Ind(H/G)}(g^*q^*x,r^*f^*y)\\
            &\we\Map_{\Ind(H)}(q^*x,g_*r^*f^*y)\\
            &\we\Map_{\Ind(H)}(q^*x,q^*f_*f^*y)\\
            &\we\Map_{\Ind(F)}(x,f_*f^*y)\\
            &\we\Map_{\Ind(F/E)}(f^*x,f^*y),
    \end{align*}
    where the third equivalence is via right adjointability and the fourth
    follows from the fact that $q$ is fully faithful.
\end{proof}

Now, we come to an important test for adjointability. We include it for
completeness, as it will not be used in the rest of the paper. Rather, when
needed, we will
check that the equivalent conditions of Proposition~\ref{prop:adjointability} are satisfied.
However, the proof is similar to one step in the proof of
Theorem~\ref{thm:minus1}.

%%%%%%%%%%%%%%% OLD VERSION %%%%%%%%%%%%%%
% \begin{proposition}\label{prop:tatelike}
%     Let
%     \begin{equation*}
%         \xymatrix{
%             E\ar[r]^i\ar[d]^p   &   F\ar[d]^q\\
%             G\ar[r]^j         &   H
%         }
%     \end{equation*}
%     be a commutative square of fully faithful functors in
%     $\Cat_{\infty}^{\perf}$ such that
%             \marginpar{Is condition (1) superfluous?}
%     \begin{enumerate}
%         \item   the natural map $E\rightarrow F\cap G$ is an equivalence,
%         \item   every object $x\in H$ is a retract of an object $x'$ that fits
%             into a cofiber sequence $y\rightarrow x' \rightarrow z$ where $y\in
%             F$ and $z\in G$,
%         \item   every object $y$ of $F$ is a filtered colimit $y\we\colim_A
%             i(y_\alpha)$ of objects in $E$ such that $q(y)\we\colim_A
%             q(i(y_\alpha))$,
%         \item   every object $z$ of $G$ is a filtered limit $z\we\lim_B p(z_\beta)$
%             of objects in $E$ such that $j(z)\we\lim_B j(p(z_\beta))$, and
%             \marginpar{Weaken (5) to actual filtered limits we need.}
%         \item   the essential image of $p$ consists of cocompact objects of
%             $G$.
%     \end{enumerate}
%     Then, the induced map $F/E\rightarrow H/G$ is an equivalence.
% \end{proposition}

\begin{proposition}\label{prop:tatelike}
    Let
    \begin{equation*}
        \xymatrix{
            E\ar[r]^i\ar[d]^p   &   F\ar[d]^q\\
            G\ar[r]^j         &   H
        }
    \end{equation*}
    be a commutative square of fully faithful functors in
    $\Cat_{\infty}^{\perf}$ such that
    \begin{enumerate}
        \item[\emph{(a)}]   every object $y$ of $G$ is a cofiltered limit
            $y\we\lim_B p(z_\beta)$ such that $jy\we\lim_B jpz_\beta$, and
        \item[\emph{(b)}]   the essential image of $q$ consists of $j$-cocompact objects of
            $H$, meaning that the natural map
            $$\colim_{B^{\op}}\Map_H(jy_\beta,qx)\rightarrow\Map_H(\lim_B
            jy_\beta,qx)$$ is an equivalence for all $x\in F$ whenever the limit $\lim_B
            y_\beta$ exists in $G$ and $j$ preserves the limit.
    \end{enumerate}
    Then, the induced map $F/E\rightarrow H/G$ is fully faithful.
\end{proposition}

\begin{proof}
    By Proposition~\ref{prop:adjointability} and Lemma~\ref{lem:adjff},
    it suffices to prove that $j_*q^*x\we 0$ for all
    $x\in\Ind(F)$ such that $i_*x\we 0$.

    So, assume that $i_*x\we 0$ for some $x\in\Ind(F)$.
    Note that $j_*q^*x\we 0$ if and only if
    $\Map_{\Ind(G)}(y,j_*q^*x)\we 0$ for all $y\in G$. Note also that $q^*$
    preserves filtered colimits. Pick one $y\in G$, and use condition (a) to write
    $y\we\lim_Bpz_\beta$ where $j$ preserves this limit. If we write $\colim_A x_\alpha\we x$ for
    some filtered $\infty$-category $A$ with $x_\alpha\in F$, then, using the
    compactness of $j^*y$, there is a chain of equivalences
    \begin{align*}
        \Map_{\Ind(G)}(y,j_*q^*x)&\we\Map_{\Ind(H)}(j^*y,q^*x)\\
            &\we\colim_A\Map_H(jy,qx_\alpha)\\
            &\we\colim_A\Map_H(j\lim_B pz_\beta,qx_\alpha)\\
            &\we\colim_A\Map_H(\lim_B jpz_\beta,qx_\alpha)\\
            &\we\colim_A\colim_{B^{\op}}\Map_H(jpz_\beta,qx_\alpha)\\
            &\we\colim_{B^{\op}}\colim_A\Map_F(iz_\beta,x_\alpha)\\
            &\we\colim_{B^{\op}}\Map_{\Ind(F)}(i^*z_\beta,x)\\
            &\we\colim_{B^{\op}}\Map_{\Ind(E)}(z_\beta,i_*x)\\
            &\we 0,
    \end{align*}
    where we use condition (b) to justify the fifth equivalence.
    This completes the proof.
% 
%     More precisely, suppose that $x\in\Ind(F)$ is of the form $\colim_A x_\alpha$
%     and that $i_*x\we 0$. Then,
%     \begin{align*}
%         \Map_{\Ind(G)}(y,\colim_A
%         j_*q^*x_\alpha)&\we\colim_{A}\Map_{\Ind(G)}(y,j_*q^*x_\alpha)\\
%         &\we\colim_A\Map_{\Ind(G)}(\lim_B p^*y_\beta,j_*q^*x_\alpha)\\
%         &\we\colim_A\Map_{\Ind(G)}(\lim_B p^*y_\beta,j_*q^*\colim_{C_\alpha}i^*x_{\alpha\gamma})\\
%         &\we\colim_A\Map_{\Ind(H)}(j^*\lim_B p^*y_\beta,\colim_{C_\alpha}q^*i^*x_{\alpha\gamma})\\
%         &\we\colim_A\Map_{\Ind(H)}(\lim_B j^*p^*y_\beta,\colim_{C_\alpha}q^*i^*x_{\alpha\gamma})\\
%         &\we\colim_A\colim_{C_\alpha}\Map_{\Ind(H)}(\lim_B j^*p^*y_\beta,q^*i^*x_{\alpha\gamma})\\
%         &\we\colim_A\colim_{C_\alpha}\Map_{\Ind(H)}(\lim_B j^*p^*y_\beta,j^*p^*x_{\alpha\gamma})\\
%         &\we\colim_A\colim_{C_\alpha}\Map_{H}(\lim_B j^*p^*y_\beta,j^*p^*x_{\alpha\gamma})\\
%         &\we\colim_A\colim_{C_\alpha}\colim_B\Map_{H}(j^*p^*y_\beta,j^*p^*x_{\alpha\gamma})\\
%         &\we\colim_A\colim_{C_\alpha}\colim_B\Map_{E}(y_\beta,x_{\alpha\gamma})\\
%         &\we\colim_C\colim_B\Map_E(y_\beta,x_\gamma)\\
%         &\we\colim_B\colim_C\Map_E(y_\beta,x_\gamma)\\
%         &\we\Map_{\Ind(E)}(y,i_*x)\we 0.
%     \end{align*}
\end{proof}

\begin{theorem}\label{thm:mv}
    Let
    \begin{equation*}
        \xymatrix{
            E\ar[r]^i\ar[d]^p   &   F\ar[d]^q\\
            G\ar[r]^j         &   H
        }
    \end{equation*}
    be a commutative square of fully faithful functors in
    $\Cat_{\infty}^{\perf}$ such that $F/E\rightarrow H/G$ is fully faithful
    and such that
    \begin{enumerate}
        \item[\emph{(c)}] every object $x$ of $H$ is a retract of an object
            $x'$ such that $x'$ fits in to a cofiber sequence $jy\rightarrow
            x'\rightarrow qz$ for some $y$ in $G$ and some $z$ in $F$.
    \end{enumerate}
    Then, the induced square
    \begin{equation*}
        \xymatrix{
            \K(E)\ar[r]\ar[d]   &   \K(F)\ar[d]\\
            \K(G)\ar[r]               &   \K(H)
        }
    \end{equation*}
    is a pushout square of spectra.
\end{theorem}

\begin{proof}
    By Lemma~\ref{lem:cofiberspushout}, it is enough to show that
    $\K(\widetilde{F/E})\rightarrow\K(\widetilde{H/G})$ is an equivalence, so it is enough to show that
    $\widetilde{F/E}\rightarrow\widetilde{H/G}$ is an equivalence.
    By hypothesis this functor is fully faithful, so it is enough to check essential surjectivity.
    Every object of $\widetilde{H/G}$ is a retract of the image of an object $x$ of $H$,
    which is in turn a
    retract of the image of an object $x'$ of $H$ fitting into a cofiber sequence
    as in (c). Since $H\rightarrow H/G$ kills $jy$, it follows that every object
    of $\widetilde{H/G}$ is a retract of the image of an object of $F$. Since
    $\widetilde{F/E}$ is idempotent complete by definition,
    $\widetilde{F/E}\rightarrow\widetilde{H/G}$ is essentially surjective.
\end{proof}

\begin{remark}
    Note that conditions (a) and (b) in Proposition~\ref{prop:tatelike} can be
    used to check fully faithfulness of $F/E\rightarrow H/G$.
\end{remark}

\begin{example}\label{ex:tate}
    Conditions (a) through (c) of Proposition~\ref{prop:tatelike} and
    Theorem~\ref{thm:mv} are meant to abstract the basic
    property of {\bf Tate objects.} Given a stable idempotent complete
    $\infty$-category $E$, the $\infty$-category $\mathrm{Tate}(E)$ of Tate objects
    in $E$ fits into a commutative square
    \begin{equation*}
        \xymatrix{
            E\ar[r]\ar[d]   &   \Ind(E)\ar[d]\\
            \mathrm{Pro}(E)\ar[r]   &   \mathrm{Tate}(E),
        }
    \end{equation*}
    and this square satisfies the properties of the theorem (for ease of exposition, we suppress set-theoretic
    issues and refer the reader to~\cite{hennion} for a careful treatment). To
    check condition (a), note that \emph{every} object of $\mathrm{Pro}(E)$ can
    be written as a cofiltered limit, and
    $\mathrm{Pro}(E)\rightarrow\mathrm{Tate}(E)$ preserves cofiltered limits
    by the universal property of Tate objects
    (see~\cite{hennion}*{Theorem~2.6}). Condition (b) follows from the fact
    that there is a natural embedding
    $\mathrm{Tate}(E)\rightarrow\mathrm{Pro}\,\Ind(E)$ which preserves
    cofiltered limits by definition of the mapping spaces in a
    pro-category. Condition (c) follows for example
    from~\cite{hennion}*{Corollary~3.4}.

    The key point about the $\infty$-category of Tate objects is that, by the
    theorem,
    $\K(\mathrm{Tate}(E))\we\Sigma\K(E)$. Indeed, $\K(\Ind(E))\we
    0\we\K(\mathrm{Pro}(E))$ because of the existence of countable (co)products, which means that
    \begin{equation*}
        \xymatrix{
            \K(E)\ar[r]\ar[d]   &   0\ar[d]\\
    0\ar[r]               &   \K(\mathrm{Tate}(E))
        }
    \end{equation*}
    is a pushout square.
\end{example}

\begin{remark}
    It is possible to build an $\infty$-category $\mathrm{Tate}^\kappa(E)$ of
    $\kappa$-Tate objects out of $\Ind(E)^\kappa$ and $^\kappa\mathrm{Pro}(E)$,
    the full subcategory of $\kappa$-cocompact objects in $\mathrm{Pro}(E)$.
    This construction has the same properties as $\mathrm{Tate}(E)$ but has the advantage
    that it is small and hence does not require working in a larger universe.
    Such an approach is closer to the spirit of this paper and is done for exact categories
    in~\cite{braunling-groechenig-wolfson}.
\end{remark}

\subsection{Vanishing of $\K_{-1}$}\label{sub:km1}

We prove our analogue of Schlichting's
theorem~\cite{schlichting}*{Theorem~6} in the case of a stable
$\infty$-category admitting a bounded $t$-structure. The proof
differs substantially from that of Schlichting.

\begin{theorem}\label{thm:minus1}
    If $E$ is a small stable $\infty$-category
    with a bounded $t$-structure, then $\K_{-1}(E)=0$.
\end{theorem}

\begin{proof}
    The $t$-structure on $E$ extends to a $t$-structure on $\Ind(E)$ with
    nonnegative objects $\Ind(E_{\geq 0})\we\Ind(E)_{\geq 0}$ by
    Proposition~\ref{lem:textension}.
%     Moreover, there is a $t$-completion
%     $\Ind(E)\rightarrow\widehat{\Ind}(E)$ where $\widehat{\Ind}(E)$ has an induced
%     $t$-structure making the functor $t$-exact. It has the property that
%     $\Ind(E)_{\leq 0}\we\widehat{\Ind}(E)_{\leq 0}$, but the bounded below objects
%     now have convergent Postnikov towers. Note that
%     $E\rightarrow\widehat{\Ind}(E)$ is fully faithful because any two objects
%     are contained in $\Ind(E)_{\leq n}$ for some $n$. It follows that the
%     objects of $E$ are not compact in $\widehat{\Ind}(E)$, for otherwise
%     $\Ind(E)\rightarrow\widehat{\Ind}(E)$ would be fully faithful.
% 
%     Nevertheless, $E\subseteq\wInd(E)^{\kappa}$ for large enough $\kappa$ since
%     the $t$-completion functor is accessible. Pick such a cardinal $\kappa>\omega$ such that
%     $E$ is additionally essentially $\kappa$-small.
% 
    Let $A=E^\heartsuit$ denote the heart of $E$, and fix $\kappa$ an uncountable regular
    cardinal such that $E$ is essentially $\kappa$-small. Consider
    the commutative diagram of fully faithful functors
    \begin{equation}\label{eq:schlichtinglike}
        \xymatrix{
            E\ar[r]^i\ar[d]_p   &   \Ind^+_{A}(E)\ar[d]^q\\
            \Ind^-_{A}(E)^\kappa\ar[r]^j    &   \Ind_{A}(E)^\kappa,
        }
    \end{equation}
    where
    \begin{itemize}
        \item   $\Ind_A^+(E)\subseteq\bigcup_{n}\Ind(E)_{\leq
            n}\we\bigcup_n\Ind(E)_{\leq n}$ is the full
            subcategory of bounded above objects $x$ with $\pi_nx\in A$ for all
            $n$,
        \item   $\Ind_A^-(E)^\kappa\subseteq\bigcup_n\Ind(E_{\geq n})^\kappa$ is the full
            subcategory of the $\kappa$-compact bounded below
            objects with $\pi_n(x)\in A$ for all $n$, and
        \item   $\Ind_A(E)^\kappa$ is the full subcategory of $\Ind(E)$ of
            objects $x\in\Ind(E)$ such that $\tau_{\leq n}x\in\Ind_A^+(E)$ and
            $\tau_{\geq n}x\in\Ind_A^-(E)^\kappa$ for all $n$.
    \end{itemize}
    Note that the inclusion $\Ind^+_A(E)\rightarrow\Ind(E)$
    factors through $\Ind_A(E)^\kappa$. Indeed, for $x\in\Ind^+_A(E)$, the
    truncations $\tau_{\leq n}x$ are in $\Ind^+_A(E)$ for all $n$.
    Moreover, $\tau_{\geq n}x$ is bounded and has homotopy objects all in $A$,
    so that $\tau_{\geq n}x$ is in fact in $E$; it follows that
    $\tau_{\geq n}x\in\Ind^-_A(E)^\kappa$.

    The objects of $\Ind_A^-(E)^\kappa$ are in fact $\kappa$-compact in
    $\Ind(E)$ because $\Ind(E_{\geq 0})\rightarrow\Ind(E)$ preserves
    $\kappa$-compact objects as the right adjoint preserves ($\omega$-)filtered colimits, and hence all
    $\kappa$-filtered colimits. For the same reason,
    $\Ind(E_{\geq n})\rightarrow\Ind(E_{\geq n-1})$ preserves $\kappa$-compact objects.
    Clearly, if $x\in\Ind_A(E)^\kappa$, then every truncation
    $\tau_{\geq n}x$ is $\kappa$-compact, however we do not claim that every
    $\kappa$-compact object $x$ in $\Ind(E)$ with $\pi_n(x)\in A$ is contained in
    $\Ind_A(E)^\kappa$.

%     \marginpar{It is not obvious that $\Ind_A(E)$ or $\Ind_A^-(E)$ are small.
%     This definition of $\Ind_A(E)$ is closed under arbitrary coproducts of
%     objects with vanishing homotopy groups. So, I guess it cannot be small unless
%     $\Ind(E)$ is also left complete.}

    We do claim that
    \begin{enumerate}
        \item[(1)] $\Ind_A(E)^\kappa$, $\Ind_A^+(E)$, and $\Ind_A^-(E)^\kappa$
            are essentially small idempotent complete stable subcategories of
            $\Ind(E)$ and
        \item[(2)]  that the $t$-structure on $\Ind(E)$ restricts to a $t$-structure on $\Ind_A(E)^\kappa$.
    \end{enumerate}
    After establishing these facts,
    we prove that we can apply Theorem~\ref{thm:mv} to the
    square~\eqref{eq:schlichtinglike}. This gives a pushout square of
    $K$-theory spectra which lets us prove in the end that $\K_{-1}(E)=0$.

    In fact $\Ind_A(E)^\kappa$ is contained in $\Ind(E)^\kappa$.
    Since the bounded below objects of $\Ind_A(E)^\kappa$ are $\kappa$-compact,
    it is enough to check that $\Ind_A^+(E)\subseteq\Ind(E)^\kappa$.
    Let $x\in\Ind_A^+(E)$, so that in particular $x\in\Ind(E)_{\leq
    n}$ for some $n$. There are maps \begin{equation}\label{eq:w}\pi_nx[n]\we\tau_{\geq
    n}x\rightarrow\tau_{\geq n-1}x\rightarrow\tau_{\geq
    n-2}x\rightarrow\cdots\rightarrow x.\end{equation} Since the induced
    $t$-structure on $\Ind(E)$ is right complete by Proposition~\ref{lem:textension},
    the colimit of the sequence is equivalent to $x$. To see this, note that it
    is enough to prove that $\colim_i\Map(y,\tau_{\geq n-i}x)\we\Map(y,x)$ for
    all $y\in E$. However, since the $t$-structure on $E$ is bounded, any such
    $y$ is contained in $E_{\geq n-i}$ for some $i$. Using the two cofiber
    sequences $$\tau_{\geq n-i}x\rightarrow\tau_{\geq
    n-j}x\rightarrow\tau_{\leq n-i-1}\tau_{\geq n-j}x$$ and $$\tau_{\geq n-i}x\rightarrow
    x\rightarrow\tau_{\leq n-i-1}x,$$ for $j\geq i$, we see that
    $$\Map(y,\tau_{\geq n-i}x)\we\Map(y,\tau_{\geq n-j}x)$$ and
    $$\Map(y,\tau_{\geq n-i}x)\we\Map(y,x)$$ for $j\geq i$. This proves that
    the colimit of~\eqref{eq:w} is indeed $x$.
    However, each object $\tau_{\geq m}x$ is actually in $E$, so this is a $\kappa$-small
    colimit of compact objects and hence of $\kappa$-compact objects in $\Ind(E)$. Thus,
    $x$ is $\kappa$-compact by~\cite{htt}*{Corollary~5.3.4.15}.

    That these three $\infty$-categories are essentially small follows from the
    fact that $\Ind_A(E)^\kappa\subseteq\Ind(E)^\kappa$ and the fact that
    $\Ind(E)^\kappa$ is essentially small because every object is the colimit
    in presheaves on $E$ of a $\kappa$-small
    diagram~\cite{htt}*{Proposition~5.3.4.17}. Moreover, $\Ind_A^+(E)$ is
    idempotent complete
    because $\Ind(E)_{\leq 0}$ and $A$ are idempotent complete, while each $\Ind(E_{\geq n})^\kappa$ is
    idempotent complete since it is closed under $\kappa$-small colimits, and
    in particular it is closed under idempotent completion because $\kappa$ is
    uncountable. It follows that
    $\Ind_A(E)^\kappa$ is idempotent complete as well since inclusion in
    $\Ind_A(E)^\kappa$ is given by a condition on the truncations.

    These three $\infty$-categories are closed under suspension and
    desuspension in $\Ind(E)$, so to see that they are stable, it is enough to
    show that they are closed under either taking fibers or cofibers in
    $\Ind(E)$ by~\cite{ha}*{Lemma~1.1.3.3} and its opposite version.
    We first note that if $z$ is the cofiber of a map $f:x\rightarrow y$ in
    $\Ind(E)$ between two objects such that $\pi_nx$ and $\pi_ny$ are in
    $A$ for all $n$, 
    then $\pi_nz$ is an extension of objects of $A$, namely of
    $\ker(\pi_{n+1}x\rightarrow\pi_{n+1}y)$ by $\coker(\pi_nx\rightarrow\pi_ny)$.
    Since $A$ is closed under extension in $\Ind(E)^{\heartsuit}$,
    by Lemma~\ref{lem:wserre}, we see that $\pi_nz\in A$.

%     It follows as in Lemma~\ref{lem:exact} that $\pi_nz$ is the cofiber in $\Ind(E)$ of a map
%     $$\ker(\pi_{n+1}x\rightarrow\pi_{n+1}y)[-1]\rightarrow\coker(\pi_nx\rightarrow\pi_ny).$$
%     However, since $E\rightarrow\Ind(E)$ is fully faithful, this map is in $E$,
%     and hence $\pi_nz\in E^{\heartsuit}\we A$, as desired.

    Stability of $\Ind_A^+(E)$ follows from the fact the cofiber of a map
    of bounded above objects is bounded above;
    stability of $\Ind_A^-(E)^\kappa$ follows from the fact that $\Ind(E_{\geq
    n})^\kappa$ is closed under cofibers in $\Ind(E)$.

    We show that $\Ind_A(E)^\kappa$ is stable, by showing that
    it is closed under taking fibers in $\Ind(E)$. If $z\rightarrow x\rightarrow
    y$ is a fiber sequence in $\Ind(E)$ where $x\rightarrow y$
    is in $\Ind_A(E)^\kappa$, then $\pi_nz\in A$ for all $n$, so that
    $\tau_{\leq n}z\in\Ind_A^+(E)$ for all $n$. Hence, it is enough to show
    that $\tau_{\geq n}z\in\Ind(E_{\geq n})^\kappa$. However, $\tau_{\geq
    n}:\Ind(E)\rightarrow\Ind(E_{\geq n})$ preserves limits as it is a right
    adjoint. Hence, $\tau_{\geq n}z$ is the fiber of $\tau_{\geq
    n}x\rightarrow\tau_{\geq n}y$ in $\Ind(E_{\geq 0})$. Now, as we have chosen
    $\kappa$ to be uncountable and such that $E$ is essentially $\kappa$-small,
    we see by~\cite{htt}*{Proposition~5.4.7.4} that the inclusion $\Ind(E_{\geq n})^\kappa\rightarrow\Ind(E_{\geq
    n})$ is closed under all finite limits in $\Ind(E_{\geq
    n})$, and in particular under fibers. Hence, $z\in\Ind_A(E)^\kappa$, which
    completes the proof of claim (1).

    Suppose that $x\in\Ind_A(E)^\kappa$. To show that the $t$-structure on
    $\Ind(E)$ restricts to $\Ind_A(E)^\kappa$, we show that $\tau_{\geq
    n}x$ and $\tau_{\leq n}x$ are in $\Ind_A(E)^\kappa$. In fact, by stability, it is
    sufficient to check only one of these. Moreover, $\tau_{\geq
    m}\tau_{\geq n}x\we\tau_{\geq m}x$ for $m\geq n$, so that if
    $x\in\Ind_A^-(E)^\kappa$, then so is $\tau_{\geq m}x$ for all $m$.
    Similarly, $\tau_{\leq m}\tau_{\geq n}x\in E\subseteq\Ind_A^+(E)$. So,
    $\tau_{\geq m}x\in\Ind_A^-(E)$, which proves that $\Ind_A(E)^\kappa$ inherits the induced $t$-structure from $\Ind(E)$.
    This proves (2).
    Note that by construction the truncation functors on $\Ind_A(E)^\kappa$
    preserve $\Ind_A^+(E)$ and $\Ind_A^-(E)^\kappa$, which therefore inherit
    compatible $t$-structures.

    To complete the proof, we will show that the
    square~\eqref{eq:schlichtinglike} satisfies the hypotheses of
    Theorem~\ref{thm:mv}. The validity of condition (c) in the theorem is due to the
    $t$-structure, which gives cofiber sequence $\tau_{\geq 0}x\rightarrow
    x\rightarrow\tau_{\leq -1}x$ for every $x\in\Ind_A(E)^\kappa$, where
    $\tau_{\geq 0}x\in\Ind_A^-(E)^{\kappa}$ and $\tau_{\leq
    -1}x\in\Ind_A^+(E)$. To prove that the induced functor
    $$\Ind_A^+(E)/E\rightarrow\Ind_A(E)^\kappa/\Ind_A^-(E)^\kappa$$ is fully
    faithful, we will prove directly that the square satisfies
    Proposition~\ref{prop:adjointability}(2) and
    invoke Lemma~\ref{lem:adjff}.

    Let $x\in\Ind(\Ind_A^+(E))$. We need to show that if $i_*x\we 0$ then
    $j_*q^*x\we 0$. It is enough to prove that
    $\Map_{\Ind(\Ind^-_A(E)^\kappa)}(y,j_*q^*x)\we 0$ for
    $y\in\Ind^-_A(E)^\kappa$. Choose a filtered $\infty$-category $B$ and an
    equivalence $\colim_B x_\beta\we x$ where $x_\beta\in\Ind_A^+(E)$ for all
    $\beta$ in $B$. Using adjunctions and
    compactness, we get a chain of equivalences
    \begin{align*}
            \Map_{\Ind(\Ind^-_A(E)^\kappa)}(y,j_*q^*x)&\we\Map_{\Ind(\Ind_A(E)^\kappa)}(j^*y,q^*x)\\
            &\we\colim_B\Map_{\Ind(E)^\kappa}(jy,qx_\beta)\\
            &\we\colim_B\Map_{\Ind(E)}(y,x_\beta)\\
            &\we\colim_B\colim_n\Map_{\Ind(E)}(\tau_{\leq n}y,x_\beta)\\
            &\we\colim_B\colim_n\Map_{\Ind_A^+(E)}(\tau_{\leq n}y,x_\beta)\\
            &\we\colim_n\colim_B\Map_{\Ind_A^+(E)}(\tau_{\leq n}y,x_\beta)\\
            &\we\colim_n\Map_{\Ind(\Ind_A^+(E))}(i^*\tau_{\leq n}y,x)\\
            &\we\colim_n\Map_{\Ind(E)}(\tau_{\leq n}y,i_*x)\\
            &\we 0,
%             &\we\\
%             &\we\colim_A\colim_{n\rightarrow\infty}\Map_{\Ind(E)}(\tau_{\leq n}jy,qx_\alpha)\\
%             &\we\colim_A\colim_{n\rightarrow\infty}\colim_{m\rightarrow -\infty}\Map_E(\tau_{\leq
%             n}jy,\tau_{\geq m}x_\alpha)\\
%             &\we\colim_{n\rightarrow\infty}\colim_A\colim_{m\rightarrow
%             -\infty}\Map_E(\tau_{\leq n}jy,\tau_{\geq m}x_\alpha)\\
%             &\we\colim_{n\rightarrow\infty}\Map_{\Ind(E)}(\tau_{\leq
%             n}jy,i_*x)\we 0
    \end{align*}
    where we use (i) the crucial fact that $x_\beta$ is bounded above as well as the
    $t$-structure to observe that the
    colimit $\colim_{n\rightarrow\infty}\Map_{\Ind(E)}(\tau_{\leq n}y,x_\beta)$
    stabilizes at $\Map_{\Ind(E)}(y,x_\beta)$ and (ii) that
    $\tau_{\leq n}jy$ is in
    $E\subseteq\Ind_A^+(E)$ and
    hence $i^*\tau_{\leq n}jy$ is compact in $\Ind(\Ind_A^+(E))$ to prove the eighth
    equivalence.

    It follows that there is a cofiber sequence
    $$\K(E)\rightarrow\K(\Ind_{A}^-(E)^\kappa)\oplus\K(\Ind_{A}^+(E))\rightarrow\K(\Ind_{A}(E)^\kappa).$$
    of $K$-theory spectra.
    It is easy to see that the $K$-theory spectra of the idempotent complete stable
    $\infty$-categories $\Ind_{A}^{+}(E)$ and $\Ind_A^-(E)^\kappa$ are zero. Indeed, there is an
    endofunctor $T:\Ind_A^-(E)^\kappa\rightarrow\Ind_A^-(E)^\kappa$ given by
    $T=\bigoplus_{n\geq 0}\id[2n]$ since $\Ind(E)^\kappa$ is closed under
    countable coproducts. We have an equivalence of endofunctors $T\we \id\oplus T[2]$. Hence, the
    identity map on $\K(\Ind_A^-(E)^\kappa)$ is nullhomotopic. The same argument but
    with desuspensions shows that $\K(\Ind_A^+(E))\we 0$. Hence,
    $$\K_0(\Ind_{A}(E)^\kappa)\iso\K_{-1}(E).$$
    Given an object $x$ of $\Ind_{A}(E)^\kappa$, we have a canonical triangle
    $$\tau_{\geq 0}x\rightarrow x\rightarrow\tau_{\leq -1}x$$ coming from the
    $t$-structure, where $\tau_{\geq 0}x$ is in $\Ind_{A}^-(E)^\kappa$ and
    $\tau_{\leq -1}x$ is in $\Ind_{A}^+(E)$. But, since $\K_0$ of each of the
    half-bounded categories is zero, it follows that the class of $x$ in
    $\K_{-1}(E)$ is also zero.
\end{proof}

\section{Induction}\label{sec:induction}

This section contains the proofs of the inductive step of our main theorem and
the nonconnective theorem of the heart in the noetherian case,
their relation to the Farrell-Jones conjecture in negative $K$-theory
for group rings, and a discussion of the major impediments to proving the
conjecture in general.

\subsection{Dualizability of compactly generated stable $\infty$-categories}\label{sub:prelims}

We discuss in this section some technical preliminaries about dualizability we
will need later.
The material here is basically well-known, but we include it for the sake of completeness.

Recall that an object $x$ in a symmetric monoidal $\infty$-category $\Pscr$ is
\df{dualizable} if there is another object, $\D x$ together with an evaluation
map $\ev:x\otimes \D x\rightarrow\mathds{1}$ and a coevaluation map
$\coev:\mathds{1}\rightarrow \D x\otimes x$ such that the composites
$$x\xrightarrow{\id_x\otimes\coev}x\otimes\D x\otimes
x\xrightarrow{\ev\otimes\id_x}x$$ and
$$\D x\xrightarrow{\coev\otimes\id_{\D x}}\D x\otimes
x\otimes\D x\xrightarrow{\id_{\D x}\otimes\ev}\D x$$ are equivalent
to the identities on $x$ and $\D x$, respectively.

In a closed symmetric monoidal $\infty$-category $\Pscr$, the endofunctor induced by
tensoring with a fixed object $x$ has a right adjoint taking $y$ to $y^x$ by definition.
Tensoring with $x$ has a left adjoint if and only if
$x$ is dualizable, in which case the unit and counit maps of the adjunction are given by tensoring
with $\coev$ and $\ev$, respectively. Moreover, when $x$ is a dualizable object
in the closed symmetric monoidal $\infty$-category $\Pscr$, there is a natural
equivalence $y\otimes x\we y^{\D x}$ for $y\in\Pscr$.

\begin{proposition}\label{prop:dual}
    If $\Cscr$ is a compactly generated stable $\infty$-category, then $\Cscr$
    is dualizable in $\Pr^\L_\st$ with dual $\Fun^\L(\Cscr,\Sp)$.
\end{proposition}

\begin{proof}
    We refer to~\cite{ha}*{Section~4.8} for information about the tensor product of
    stable presentable $\infty$-categories.
    Because colimits in $\Fun^{\L}(\Cscr,\Sp)$ are computed pointwise, the
    evaluation bifunctor $\Cscr\times\Fun^\L(\Cscr,\Sp)\to\Sp$ preserves
    colimits separately in each variable, so we obtain an evaluation map
    $\Cscr\otimes\Fun^\L(\Cscr,\Sp)\to\Sp$.
    We must define a coevaluation map $\Sp\to\Fun^\L(\Cscr,\Sp)\otimes\Cscr$, which is to say an object of
    \[
    \Fun^\L(\Cscr,\Sp)\otimes\Cscr\simeq\Fun^{\lim}(\Cscr^{\op},\Fun^\L(\Cscr,\Sp)),
    \]
    where $\Fun^{\lim}$ denotes the $\infty$-category of limit-preserving functors.
    Using the fact that $\Cscr$ is stable and compactly generated (i.e.
    $\Cscr\simeq\Ind(\Cscr^\omega)$) we have an equivalence
    $\Fun^\L(\Cscr,\Sp)\simeq\Fun^{\ex}(\Cscr^\omega,\Sp)$.
    Moreover, the (restricted) spectral co-Yoneda embedding
    $h:\Cscr^{\op}\to\Fun(\Cscr^\omega,\Sp)$ preserves limits and factors through
    the full subcategory $\Fun^{\ex}(\Cscr^\omega,\Sp)\subseteq\Fun(\Cscr^\omega,\Sp)$.
    This gives the desired limit-preserving functor $\Cscr^{\op}\to\Fun^\L(\Cscr,\Sp)$.
    It is then routine to verify the triangle identities, so that $\Cscr$ is dualizable with dual $\Fun^\L(\Cscr,\Sp)$.
    For example, consider the composition
    $$\Cscr\xrightarrow{\id_\Cscr\otimes\mathrm{coev}}\Cscr\otimes\Fun^{\L}(\Cscr,\Sp)\otimes\Cscr\xrightarrow{\mathrm{ev}\otimes\id_\Cscr}\Cscr,$$
    which we can write as the composition
    $$\Cscr\rightarrow\Fun^{\lim}(\Cscr^{\op},\Cscr\otimes\Fun^{\L}(\Cscr,\Sp))\xrightarrow{\mathrm{ev}}\Fun^{\lim}(\Cscr^{\op},\Sp).$$
    By definition of the coevaluation map, the composition is the Yoneda
    embedding $\Cscr\rightarrow\Fun^{\lim}(\Cscr^{\op},\Sp)$. Since the natural equivalence
    $\Fun^{\lim}(\Cscr^{\op},\Sp)\we\Cscr$ takes the representable functor
    $h(x)$ to $x$, we see that the composition is equivalent to the
    identity. The argument for the dual is similar.
\end{proof}

\begin{lemma}\label{lem:cg}
    If $\Cscr$ and $\Dscr$ are compactly generated stable
    $\infty$-categories, then $\Cscr\otimes\Dscr$ is compactly generated by
    objects the $x\otimes y$, where $x$ and $y$ range over compact generators of
    $\Cscr$ and $\Dscr$, respectively.
\end{lemma}

\begin{proof}
    We show first that these objects are compact.
    The mapping spectrum functor
    $$(\Cscr\otimes\Dscr)(x\otimes
    y,-):\Cscr\otimes\Dscr\rightarrow\Mod_{\SS}$$ preserves
    filtered colimits if and only if it admits a right adjoint by the adjoint
    functor theorem, and when this is the case it
    represents an object in $\Fun^{\L}(\Cscr\otimes\Dscr,\Mod_{\SS})$. By the
    universal property of $\Cscr\otimes\Dscr$ this occurs if and only if
    $\Cscr(x,-)\otimes_\SS\Dscr(y,-):\Cscr\times\Dscr\rightarrow\Mod_{\SS}$
    preserves filtered colimits in each variable, using the fact that if
    if $x,x'\in\Cscr^\omega$ and $y,y'\in\Dscr^\omega$,
    then $$(\Cscr\otimes\Dscr)(x\otimes y,x'\otimes
    y')\we\Cscr(x,x')\otimes_{\SS}\Dscr(y,y').$$
    This happens if and only if $x$ and $y$ are both compact.

    By definition, $\Cscr\otimes\Dscr$ is the universal stable presentable
    $\infty$-category equipped with a functor
    $\Cscr\times\Dscr\rightarrow\Cscr\otimes\Dscr$ preserving small colimits in
    each variable (see~\cite{htt}*{Remark~5.5.3.9}). Suppose that $F:\Cscr\times\Dscr\rightarrow\Escr$ is a functor
    preserving colimits in each variable such that $F(x,y)=0$ for all
    $x\in\Cscr^\omega$ and $y\in\Cscr^{\omega}$. By definition of $F$ and the
    fact that every object of $\Cscr$ (resp. $\Dscr$) can be written as a small
    colimit of objects of $\Cscr^\omega$ (resp. $\Dscr^\omega$), it follows
    that $F$ vanishes. It follows that the objects
    $$\{x\otimes y:x\in\Cscr^\omega,y\in\Dscr^\omega\}$$ generate
    $\Cscr\otimes\Dscr$.
\end{proof}

\begin{proposition}
    Let $\Cscr$ be dualizable object of $\Pr^\L_\st$.
    Then for any fully faithful functor $\Ascr\to\Bscr$ in $\Pr^\L$, the
    induced functor $\Ascr\otimes\Cscr\to\Bscr\otimes\Cscr$ is fully faithful.
\end{proposition}

\begin{proof}
Let $\Dscr$ be a dual of $\Cscr$. We have a commutative diagram
\[
\xymatrix{\Ascr\otimes\Cscr\ar[r]\ar[d] & \Fun^\L(\Dscr,\Ascr)\ar[r]\ar[d] & \Fun(\Dscr,\Ascr)\ar[d]\\
\Bscr\otimes\Cscr\ar[r] & \Fun^\L(\Dscr,\Bscr)\ar[r] & \Fun(\Dscr,\Bscr),}
\]
in which the left hand horizontal maps are equivalences and the right hand
horizontal maps are fully faithful. (Technically, $\Fun(\Dscr,-)$ lands in a
higher universe, but we can restrict to the $\kappa$-continuous functors for
any $\kappa$ such that $\Dscr$ is $\kappa$-compactly generated.)
Moreover, the right hand vertical map is fully faithful since $\Ascr\to\Bscr$
is, by hypothesis, so it follows that each of the other vertical maps is fully faithful.
\end{proof}

Note that we did not actually use the fact that $\Cscr$ was stable in the proof
of the above proposition; the same argument works for $\Cscr$ dualizable in
$\Pr^\L$.
Unfortunately, there are not so many dualizable objects of $\Pr^\L$, but as
soon as we pass to $\Pr^\L_\st$, we obtain a vast supply by
Proposition~\ref{prop:dual}.

A \df{localization sequence} in $\Pr^{\L}_{\st}$ is a cofiber sequence
$\Bscr\rightarrow\Cscr\rightarrow\Dscr$ such that $\Bscr\rightarrow\Cscr$ is
fully faithful. The stable presentable $\infty$-category $\Dscr$ in this case is
equivalent to the usual Bousfield localization of $\Cscr$ at the arrows with
cofiber in $\Bscr$ by~\cite{bgt1}*{Proposition~5.6}.

\begin{example}
    The prototypical localization sequence arises from a quasi-compact and
    quasi-separated scheme $X$ together with a quasi-compact open subscheme
    $U\subseteq X$ with complement $Z$. In this case, the functor $\Dscr(X)\rightarrow\Dscr(U)$ is
    a localization with kernel $\Dscr_Z(X)$, the $\infty$-category of complexes
    of $\Oscr_X$-modules with quasi-coherent cohomology sheaves supported
    set-theoretically on $Z$. Hence,
    $\Dscr_Z(X)\rightarrow\Dscr(X)\rightarrow\Dscr(U)$ is a localization
    sequence.
\end{example}

Since localization sequences in $\Pr^{\L}_{\st}$ are cofiber sequences and as
the tensor product on $\Pr^{\L}_{\st}$ preserves small colimits in each
variable by~\cite{ha}*{Remark~4.8.1.23}, the previous lemma shows that
localization sequences of stable presentable
$\infty$-categories are preserved by tensoring with a given compactly generated
stable $\infty$-category $\Escr$. Thus, we have proved the
following.

\begin{corollary}\label{cor:53}
    Let $\Bscr\rightarrow\Cscr\rightarrow\Dscr$ be a localization sequence of
    stable presentable $\infty$-categories. Then,
    $$\Bscr\otimes\Escr\rightarrow\Cscr\otimes\Escr\rightarrow\Dscr\otimes\Escr$$
    is a localization sequence for any compactly generated stable
    $\infty$-category $\Escr$.
\end{corollary}

\subsection{Negative $K$-theory via $\infty$-categories of automorphisms}\label{sub:universal}

In this section, we prove the following theorem, which verifies
Conjecture~\hyperlink{conj:b}{B} in many cases.

\begin{theorem}\label{thm:gscnh}
    If $E$ is a small stable $\infty$-category equipped with a bounded
    $t$-structure such that $E^{\heartsuit}$ is noetherian, then $\K_{-n}(E)=0$
    for $n\geq 1$.
\end{theorem}

Many of our arguments in the proof work in greater generality,
and we take care to isolate those parts that are truly special to
the situation of a noetherian heart.

\begin{definition}
    Throughout this section, $\SS[s]=\Sigma^\infty_+\NN$ denotes the free
    commutative $\SS$-algebra on the commutative monoid $\NN$. Note that
    $\SS[s]$ equivalent to the free $\EE_1$-ring spectrum on the sphere
    spectrum $\SS$.
    Similarly, $\SS[s^{\pm 1}]=\Sigma^\infty_+\ZZ$ is the free commutative
    $\SS$-algebra on the commutative monoid $\ZZ$, or, equivalently, the
    localization of $\SS[s]$ obtained by
    inverting $s\in\pi_0\SS[s]$. These are
    each flat over $\SS$ and have the expected ring of components; that is,
    $\pi_0(\SS[s])\iso\ZZ[s]$ and $\pi_0(\SS[s^{\pm 1}])\iso\ZZ[s^{\pm 1}]$,
    while $\pi_*(\SS[s])\iso(\pi_*\SS)[s]$ and $\pi_*(\SS[s^{\pm
    1}])\iso(\pi_*\SS)[s^{\pm 1}]$. {\bfseries Warning:} the commutative
    $\SS$-algebra $\SS[s]$ is not the free commutative (or $\EE_\infty$)
    algebra on a single element in degree $0$.
\end{definition}

\begin{notation}
    In general, we will use either the notation $\Mod_R$ or $\Dscr(R)$ for
    the stable presentable $\infty$-category of right $R$-modules for an
    $\EE_{\infty}$-ring spectrum $R$. Moreover, in the special cases of $\SS[s]$ and
    $\SS[s^{\pm 1}]$, we will use the suggestive notation $\Dscr(\AA^1)=\Mod_{\SS[s]}$ and
    $\Dscr(\Gm)=\Mod_{\SS[s^{\pm 1}]}$. Given a stable presentable
    $\infty$-category $\Cscr$, we write
    $$\Dscr(\AA^1,\Cscr)=\Dscr(\AA^1)\otimes\Cscr=\Mod_{\SS[s]}\otimes\Cscr,$$ and similarly for
    $\Dscr(\Gm,\Cscr)$.
\end{notation}

% The goal in this section is threefold. First, when $\Cscr$ is a stable presentable
% $\infty$-category with an accessible $t$-structure, then we will show that
% there is are induced $t$-structures on $\Dscr(\AA^1,\Cscr)$ and
% $\Dscr(\Gm,\Cscr)$. Second, in the special case that $\Cscr=\Ind(E)$ and $E$
% admits a bounded $t$-structure, we prove that $\Dscr(\AA^1,\Cscr)^\omega$ and
% $\Dscr(\Gm,\Cscr)^\omega$ admit bounded $t$-structures. Third, we prove the
% generalized Schlichting conjecture when $E$ has a bounded $t$-structure with a
% noetherian heart.

To begin, we show that $\Dscr(\AA^1,\Cscr)$ can be identified with the
$\infty$-category of endomorphisms in $\Cscr$, and that $\Dscr(\Gm,\Cscr)$ is
equivalent to the $\infty$-category of automorphisms in $\Cscr$.

\begin{definition}
    Given an $\infty$-category $\Cscr$, the functor category
    $$\Fun(\Delta^1/\partial\Delta^1,\Cscr)$$ is the \df{$\infty$-category of
    endomorphisms in $\Cscr$}. An object of the $\infty$-category of
    endomorphisms consists of a pair $(x,e)$ where $x$ is an object of $\Cscr$
    and $e:x\rightarrow x$ is an endomorphism. For example, if $\Cscr$ is
    additive, then $(x,0)$ (the object $x$ equipped with the zero endomorphism)
    and $(x,\id_x)$ are functorial sections of the
    forgetful functor $\Fun(\Delta^1/\partial\Delta^1,\Cscr)\rightarrow\Cscr$.
    The \df{$\infty$-category of
    automorphisms in $\Cscr$} is $$\Fun(S^1,\Cscr),$$ where $S^1\simeq B\ZZ$ is a Kan complex weakly equivalent to $\Delta^1/\partial\Delta^1$. The map of
    $\infty$-categories
    $\Delta^1/\partial\Delta^1\rightarrow S^1$ induces a fully faithful
    embedding
    $$\Fun(S^1,\Cscr)\rightarrow\Fun(\Delta^1/\partial\Delta^1,\Cscr)$$
    with essential image those endomorphisms $(x,e)$ such that $e:x\rightarrow
    x$ is an equivalence.
\end{definition}

\begin{proposition}
    If $\Cscr$ is a stable presentable $\infty$-category, then
    \begin{enumerate}
        \item[\emph{(i)}]
            $\Dscr(\AA^1,\Cscr)\we\Fun(\Delta^1/\partial\Delta^1,\Cscr)$, and
        \item[\emph{(ii)}] $\Dscr(\Gm,\Cscr)\we\Fun(S^1,\Cscr)$.
    \end{enumerate}
\end{proposition}

\begin{proof}
    We prove (i), the proof of (ii) being similar. We claim that there is a natural equivalence
    $$\Fun(\Delta^1/\partial\Delta^1,\Cscr)\we\Fun^{\L}(\Mod_{\SS[s]},\Cscr).$$
    It suffices to show that $\Mod_{\SS[s]}$ is the free stable presentable
    $\infty$-category generated by $\Delta^1/\partial\Delta^1$.
    This follows from the $(\Mod_\ast,\End)$ adjunction \cite{ag}*{Section 3.1}
    together with the fact that $\SS[s]\simeq\SS[\NN]$ is the free
    $\SS$-algebra on the monoid $\NN$, and that the nerve of $\NN$ (viewed as a
    category with one object) is a fibrant replacement for
    $\Delta^1/\partial\Delta^1$ in the Joyal model structure.
    %(Note that the underlying spectrum of the
    %free $\SS$-algebra on a monoid $M$ is $\Sigma^\infty_+M$, the $M$-fold coproduct
    %of $\SS$ in spectra.)
    For any $\SS$-algebra $R$ and any stable presentable $\infty$-category
    $\Cscr$, there is a natural equivalence
    $\Mod_{R^{\op}}\otimes\Cscr\we\Fun^{\L}(\Mod_R,\Cscr)$.
    Indeed, $\Mod_R$ is compactly generated, and hence dualizable by
    Proposition~\ref{prop:dual} with dual $\Mod_{R^{\op}}$.
    In particular, since $\SS[s]$ is an $\EE_\infty$-ring spectrum, the $\infty$-category of
    endomorphisms in a stable presentable $\infty$-category $\Cscr$ is equivalent
    to $\Mod_{\SS[s]}\otimes\Cscr$.
\end{proof}

%     More precisely, we claim that there are four natural equivalences
%     \begin{align*}
%     \Mod_{R^{\op}}\otimes\Cscr&\we\Fun^{\R}(\Mod_{R^{\op}}^{\op},\Cscr)\\
%         &\we\Fun^{\mathrm{ex}}(\Mod_{R^{\op}}^{\omega,\op},\Cscr)\we\Fun^{\mathrm{ex}}(\Mod_{R}^\omega,\Cscr)\we\Fun^{\L}(\Mod_{R},\Cscr).
%     \end{align*}
%     The first equivalence is~\cite{ha}*{Proposition~4.8.1.16}. The second
%     follows from the fact that $\Mod_{R^{\op}}^{\op}$ is generated
%     under limits by $\Mod_{R^{\op}}^{\op}$, that $\Cscr$ is closed under
%     limits, and that the functors in $\Fun^{\R}$ are right adjoints. The third
%     equivalence follows from the fact that $(R^{\op})^{\op}\we R$ and hence
%     that $\Mod_{R^{\op}}^{\omega,\op}\we\Mod_R^{\omega}$. Finally, the fourth
%     equivalence follows from the fact that $\Mod_R$ is generated under colimits
%     by $\Mod_R^\omega$.

% A typical object of $\Mod_{\SS[t]}\otimes\Cscr$ is denoted by
% $(x,e)$, determined by an object $x\in\Cscr$ and an endomorphism $e$ of $X$.

We focus now on the case where $\Cscr\we\Ind(E)$ is compactly generated by a
small stable $\infty$-category $E$.

\begin{lemma}
    If $\Cscr\we\Ind(E)$ is compactly generated, then $\Dscr(\AA^1,\Cscr)$ is
    compactly generated by the objects $\SS[s]\otimes x=:x[s]$ and
    $\Dscr(\Gm,\Cscr)$ is compactly generated by the objects
    $\SS[s^{\pm 1}]\otimes x=:x[s^{\pm 1}]$
    as $x$ ranges over the objects of $E$.
\end{lemma}

\begin{proof}
    This is a special case of Lemma~\ref{lem:cg}.
\end{proof}

\begin{lemma}
    If $\Cscr\we\Ind(E)$ is compactly generated, then
    the natural functor $\Dscr(\AA^1,\Cscr)\rightarrow\Dscr(\Gm,\Cscr)$ is a
    localization with kernel a compactly generated stable presentable
    $\infty$-category which we will denote $\Dscr_{\{0\}}(\AA^1,\Cscr)$. Moreover,
    $\Dscr_{\{0\}}(\AA^1,\Cscr)$ is compactly generated by the compact objects
    $(x,0)$ in $\Dscr(\AA^1,\Cscr)$ as $x$ ranges over the objects of $E$.
\end{lemma}

The fact that $\Dscr_{\{0\}}(\AA^1,\Cscr)$ is generated by compact objects that
are compact in $\Dscr(\AA^1,\Cscr)$ implies that the right adjoint
$\Dscr(\Gm,\Cscr)\rightarrow\Dscr(\AA^1,\Cscr)$ to the localization preserves
filtered colimits.

\begin{proof}
    By~\cite{ha}*{Proposition~7.2.4.17} or~\cite{ag}*{Proposition~6.9},
    we have a localization sequence
    $$\Mod_{\AA^1,\{0\}}\rightarrow\Mod_{\AA^1}\rightarrow\Mod_{\Gm}$$ of stable
    presentable $\infty$-categories, where $\Mod_{\AA^1,\{0\}}$ is the full
    subcategory of $\Mod_{\AA^1}$ consisting of $\SS[s]$-modules $M$ such that
    for every $x\in\pi_m(M)$ there exists a positive integer $N$ such that
    $s^N\cdot x=0$. Moreover, by~\cite{ag}*{Proposition~6.9}, $\Mod_{\AA^1,\{0\}}$ is compactly generated by
    the object $\SS$ when viewed as an $\SS[s]$-module; in particular, since
    $\SS\we\mathrm{cofib}\left(\SS[s]\xrightarrow{t}\SS[s]\right)$ is compact as an $\SS[s]$-module, $\Mod_{\AA^1,\{0\}}$ is compactly
    generated by compact objects of $\Mod_{\AA^1}$. By tensoring with
    $\Ind(E)$, we obtain the localization sequence we want by
    Corollary~\ref{cor:53}. The object $\SS\otimes x$ is by definition $x$ with
    the zero endomorphism.
\end{proof}

% \begin{lemma}
%     The full subcategory $\Mod_{\AA^1,\{0\}}^E\subseteq\Mod_{\AA^1}^E$  is generated by the compact objects
%     $(x,0)$, where $x\in E$ is equipped with the $0$ endomorphism.
% \end{lemma}
% 
% \begin{proof}
%     This can be seen
%     for example by an easy generalization of~\cite{ag}*{Proposition~6.9}.
%     Alternatively, we can note that $\SS$ with a trivial action of $t$ is a
%     compact generator of $\Mod_{\AA^1,\{0\}}$, so that the objects $\SS\otimes
%     x=(x,0)$ generate $\Mod_{\AA^1,\{0\}}^E=\Mod_{\AA^1,\{0\}}\otimes\Ind(E)$.
% \end{proof}

We turn to the problem of constructing $t$-structures on $\infty$-categories of
endomorphisms and automorphisms.

\begin{lemma}\label{lem:tstructureendomorphisms}
    Let $\Cscr=\Ind(E)$ be a compactly generated stable presentable $\infty$-category with a $t$-structure
    $(\Cscr_{\geq 0},\Cscr_{\leq 0})$. The full subcategory $\Dscr(\AA^1,\Cscr)_{\geq
    0}\subseteq\Dscr(\AA^1,\Cscr)$ of endomorphisms
    $(x,e)$ where $x\in\Cscr_{\geq 0}\subseteq\Cscr$ defines the
    non-negative part of a $t$-structure on $\Dscr(\AA^1,\Cscr)$, where
    $(y,f)\in\Dscr(\AA^1,\Cscr)_{\leq 0}$ if and only if
    $y\in\Cscr_{\leq 0}$. Moreover, the truncation functors induced by the
    $t$-structure on $\Dscr(\AA^1,\Cscr)$ preserve the full subcategory
    $\Dscr(\AA^1,\Cscr)\supseteq\Dscr(\Gm,\Cscr)$.
\end{lemma}

\begin{proof}
    Requirement (1) of Definition~\ref{def:t} is inherited from $\Cscr$.
    Since the truncations $\tau_{\geq 0}x$ and $\tau_{\leq 0}x$ are functorial,
    there is a cofiber sequence $$(\tau_{\geq 0}x,\tau_{\geq
    0}(e))\rightarrow(x,e)\rightarrow(\tau_{\leq -1}x,\tau_{\leq -1}(e)).$$
    This verifies requirement (3).
    As for (2), note that the forgetful
    functor $\Dscr(\AA^1,\Cscr)\rightarrow\Cscr$
    detects nullhomotopic maps. This means that if
    $x\in\Cscr_{\geq 0}$ and $y\in\Cscr_{\leq -1}$, then
    $$\Map_{\Dscr(\AA^1,\Cscr)}((x,e),(y,f))\we 0$$ for any endomorphisms $e$ of $x$ and $f$ of $y$.
    The first claim follows.

    If $(x,e)$ is an object of $\Dscr(\Gm,\Cscr)$, then $e$ is an automorphism
    of $x$, and hence $\tau_{\geq 0}(e)$ is an automorphism of $\tau_{\geq
    0}x$. So, the truncation functors preserve
    $\Dscr(\Gm,\Cscr)\subseteq\Dscr(\AA^1,\Cscr)$. This proves the second
    claim.
\end{proof}

\begin{proposition}\label{prop:tstructurea1}
    Let $\Cscr=\Ind(E)$ be a compactly generated stable presentable
    $\infty$-category with the $t$-structure induced (in the sense of
    Proposition~\ref{lem:textension}) by a bounded $t$-structure
    on $E$ such that $E^{\heartsuit}$ is noetherian.
    The $t$-structure on $\Dscr(\AA^1,\Cscr)$ of the previous lemma restricts to a
    bounded $t$-structure with noetherian heart on the full subcategory $\Dscr(\AA^1,\Cscr)^\omega$ of
    compact objects.
\end{proposition}

\begin{proof}
    Let $F\subseteq\Dscr(\AA^1,\Cscr)^\omega$ be the full subcategory of
    objects $x$ such that $\tau_{\geq n}x$ is compact for all $n$.
    It follows immediately that $F$ is idempotent complete. Moreover, $F$
    contains all objects of the form $(x[s],s)$ for $x\in E$ since $\tau_{\geq
    n}(x[s],s)\we(\tau_{\geq n}x,\tau_{\geq n}(s))$ and since $\tau_{\geq n}x$
    is in $E$ if $x$ is in $E$. Therefore, if $F$ is stable, the inclusion
    $F\rightarrow\Dscr(\AA^1,\Cscr)^\omega$ is an equivalence. By definition,
    $F$ is closed under suspension and desuspension. Hence, by~\cite{ha}*{Lemma~1.1.3.3},
    it is enough to show that $F$ is closed under taking cofibers.

    Hence, given a cofiber sequence $x\rightarrow y\rightarrow c$ in
    $\Dscr(\AA^1,\Cscr)^\omega$ with $x,y\in F$, we must show that $\tau_{\geq
    0}c$ is compact. Let $d$ be the cofiber of
    $\tau_{\geq 0}y\rightarrow\tau_{\geq 0}y$, so that $d$ fits into a second
    cofiber sequence $d\rightarrow\tau_{\geq 0}c\rightarrow \pi$, where
    $\pi\in\Dscr(\AA^1,\Cscr)^{\heartsuit}$ is the image of
    $\pi_0c\rightarrow\pi_{-1}x$. As $d$ is compact by the hypothesis on $x$
    and $y$, it is enough to show that $\pi$ is compact in
    $\Dscr(\AA^1,\Cscr)$.

    Let $A=E^{\heartsuit}$, and let $A[s]$ be the full subcategory of compact
    objects in $\Dscr(\AA^1,\Cscr)^\heartsuit$. Note that this is well-defined
    because $\left(\Mod_{\SS[s]}^\cn\otimes\Cscr_{\geq
    0}\right)^\heartsuit\we\Mod_{\SS[s]}^\heartsuit\otimes\Cscr^\heartsuit$
    by definition of the tensor product of Grothendieck abelian categories
    in~\cite{sag}. In general, there is no reason
    for $A[s]$ to be an abelian category. However, this is implied by the fact
    that $A$ is noetherian in our case, as this ensures that the kernel of a
    map between finitely presented objects is again finitely presented.
    Moreover, it is a consequence that
    $\pi_nx\in A[s]$ whenever $x\in\Dscr(\AA^1,\Cscr)^\omega$.
    Now, we claim that $A[s]$ is noetherian. To see this, note that
    $\Dscr(\AA^1,\Cscr)$ is compactly generated by the objects $(y[s],s)$ where
    $y\in E$. It follows that every object of $A[s]$ is obtained in finitely
    many steps (consisting of taking kernels, cokernels, extensions, sums, and
    summands) from objects of the form $(x[s],s)$, where $x\in A$. Since $A$ is
    noetherian, every such $x$ is noetherian and~\cite{swan}*{Theorem~3.5} implies that
    $(x[s],s)$ is noetherian. This is effectively an easy generalization of the Hilbert basis
    theorem in algebra. Noetherianity implies that $A[s]$ is a Serre subcategory of
    $\Dscr(\AA^1,\Cscr)^\heartsuit$, so that $\pi\subseteq\pi_{-1}x$ is in
    $A[s]$. We are reduced to proving that if $\pi$ is an object of $A[s]$,
    then $\pi$ is compact as an object of $\Dscr(\AA^1,\Cscr)$.

    It is convenient for the rest of the proof to write $s$ for the
    endomorphism of any object of $\Dscr(\AA^1,\Cscr)$.
    Let $F_i\pi=\ker(s^i:\pi\rightarrow\pi)$ for $i\geq 0$. This is an increasing filtration
    on $\pi$, which stabilizes at some $F_N$ for $N\geq 0$ since $A[s]$ is
    noetherian. Each $F_i\pi/F_{i-1}\pi$ is in fact an object of $A$ as it is a
    finitely presented object of $\Dscr(\AA^1,\Cscr)^\heartsuit$ such that $s$
    acts as zero. So, inductively, $F_N\pi$ is compact. Let $\tau$ be the
    quotient $\pi/\F_N\pi$. The endomorphism $s$ acts injectively on $\tau$ by
    construction. To see that $\tau$ is compact, choose a surjection
    $a[s]\rightarrow\tau$ such that $a\in A$. Let $\sigma$ be the kernel, and
    let $\sigma_i\subseteq a\cdot s^i$ be the intersection of the kernel and
    $a\cdot s^i\subseteq a[s]$ (viewed as an object of $\Ind(A)$). Then,
    $s:\sigma_i\rightarrow\sigma_{i+1}$ and $\sigma\iso\bigoplus_{i\geq
    0}\sigma_i$. Moreover, $s:\sigma_i\rightarrow\sigma_{i+1}$ is an
    isomorphism for all $i$. The injectivity follows from the fact that $s$ acts injectively on
    $a[s]$, while the surjectivity follows (via the snake lemma) from the fact that $s$ acts
    injectively on $\tau$ and $s:a\cdot s^i\to a\cdot s^{i+1}$ is surjective. It follows that $\sigma\iso\sigma_0[s]$. But, since
    $\sigma_0\subseteq a$, it follows that $\sigma$ is compact. Therefore,
    $\tau$ is compact.

    Now, to see that the $t$-structure is bounded, it is enough to see that
    each object $(x[s],s)$ is bounded for $x\in E$. This is the case by
    construction. Finally, we have already mentioned that
    $\Dscr(\AA^1,\Cscr)^{\omega,\heartsuit}\we A[s]$ is noetherian.
\end{proof}

\begin{remark}
    Noetherianity is used in a couple primary locations in the proof. The first
    is to check that $\pi_{-1}x$ is finitely presented and that $\pi$ is
    therefore itself in $A[s]$. The second is to guarantee that the filtration
    $F_\bullet\pi$ stabilizes. We return in the next sections to the problem of
    weakening the noetherian hypothesis.
\end{remark}

\begin{lemma}
    Let $\Cscr=\Ind(E)$ be a compactly generated stable presentable
    $\infty$-category with a $t$-structure induced (in the sense of
    Proposition~\ref{lem:textension}) by a bounded $t$-structure
    on $E$ such that $E^{\heartsuit}$ is noetherian.
    Then, the $t$-structure on $\Dscr(\AA^1,\Cscr)^\omega$ respects
    $\Dscr_{\{0\}}(\AA^1,\Cscr)^\omega$.
\end{lemma}

\begin{proof}
    Note that we did not prove in general that the $t$-structure on
    $\Dscr(\AA^1,\Cscr)$ restricts to a $t$-structure on
    $\Dscr_{\{0\}}(\AA^1,\Cscr)$. But, this is true under the noetherianity
    condition for the compact objects. Indeed, an object $x$ of
    $\Dscr(\AA^1,\Cscr)^\omega$ is contained in the subcategory
    $\Dscr_{\{0\}}(\AA^1,\Cscr)^\omega$ if and only if $s^N$ acts
    nullhomotopically on $x$ for some $N\geq 0$. If $s^N$ does act
    nullhomotopically on $x$, then it does so on $\tau_{\geq 0}x$ as well,
    which shows that if $x\in\Dscr_{\{0\}}(\AA^1,\Cscr)^\omega$, then so is
    $\tau_{\geq 0}x$.
\end{proof}

\begin{corollary}\label{cor:tstructuregm}
    Let $\Cscr=\Ind(E)$ be a compactly generated stable presentable
    $\infty$-category with a $t$-structure induced (in the sense of
    Proposition~\ref{lem:textension}) by a bounded $t$-structure
    on $E$ such that $E^{\heartsuit}$ is noetherian. Then, there is a bounded
    $t$-structure on $\Dscr(\Gm,\Cscr)^\omega$ with noetherian heart.
\end{corollary}

\begin{proof}
    The subcategory
    $$\Dscr_{\{0\}}(\AA^1,\Cscr)^{\omega,\heartsuit}\subseteq\Dscr(\AA^1,\Cscr)^{\omega,\heartsuit}$$
    is in fact a Serre subcategory. Indeed, if $\tau\subseteq\sigma$ is a
    subobject where $s$ acts nilpotently on $\sigma$, then $s$ acts nilpotently
    on $\tau$ as well. Using (i) implies (ii) in
    Proposition~\ref{prop:tlocalization}, we see that there is an induced
    $t$-structure on $\Dscr(\Gm,\Cscr)^\omega$ and that the functor
    $\Dscr(\AA^1,\Cscr)^\omega\rightarrow\Dscr(\Gm,\Cscr)^\omega$ is $t$-exact. 
    Since every object of $\Dscr(\Gm,\Cscr)^\omega$ is a retract of an object
    in the image of the localization functor and since the $t$-structure on
    $\Dscr(\AA^1,\Cscr)^\omega$ is bounded, it follows that the $t$-structure
    on $\Dscr(\Gm,\Cscr)^\omega$ is bounded too. The abelian category
    $\Dscr(\Gm,\Cscr)^{\omega,\heartsuit}$ is noetherian because it is
    equivalent to the localization of the noetherian abelian category
    $\Dscr(\AA^1,\Cscr)^{\omega,\heartsuit}$ by the Serre subcategory
    $\Dscr_{\{0\}}(\AA^1,\Cscr)^{\omega,\heartsuit}$.
\end{proof}

% \begin{lemma}
%     The full subcategory $\Mod_{\Gm}^E\subseteq\Mod_{\AA^1}^E$ is closed under
%     truncations in the $t$-structure of the previous lemma.
% \end{lemma}
% 
% \begin{proof}
%     This is clear because if $t$ acts invertibly on $x$, then $\tau_{\geq
%     0}(t)$ acts invertibly on $\tau_{\geq 0}x$ and $\tau_{\leq 0}(t)$ acts
%     invertibly on $\tau_{\leq 0}x$.
% \end{proof}

% \begin{lemma}\label{lem:inducedt}
%     The induced $t$-structure on $\Mod_{\Gm}^E$ respects compact objects.
% \end{lemma}
% 
% \begin{proof}
%     This is also clear because the stable subcategory of compact objects is generated by the
%     $t$-localizations of $(x[t],t)$ where
%     $x\in E$. But, $$\tau_{\geq 0}(x[t],t)\we((\tau_{\geq 0}x)[t],t),$$
%     which is also a $t$-localization of a compact object.
% \end{proof}
% 
% \begin{lemma}
%     The induced $t$-structure on $\left(\Mod_{\Gm}^E\right)^\omega$ is bounded.
% \end{lemma}
% 
% \begin{proof}
%     This follows from the boundedness on $E$ and because the generators as in the proof of the previous lemma are all
%     bounded.
% \end{proof}
% 
% 

\begin{proof}[Proof of Theorem~\ref{thm:gscnh}]
    Let $\Cscr=\Ind(E)$.
    Applying $K$-theory to the exact sequence
    $$\Dscr_{\{0\}}(\AA^1,\Cscr)^\omega\rightarrow\Dscr(\AA^1,\Cscr)^\omega\rightarrow\Dscr(\Gm,\Cscr)^\omega,$$
    we obtain a cofiber sequence $$\K_{\{0\}}(\AA^1,\Cscr)\rightarrow\K(\AA^1,\Cscr)\rightarrow\K(\Gm,\Cscr)$$ of
    nonconnective $K$-theory spectra.

    Consider the exact functor $i:\Cscr\rightarrow\Dscr_{\{0\}}(\AA^1,\Cscr)$ given by
    $i(x)\we(x,0)$. Since $(x[s],s)\xrightarrow{s}(x[s],s)\rightarrow(x,0)$ is
    a cofiber sequence in $\Dscr(\AA^1,\Cscr)$, we see that $(x,0)$ is compact
    if $x$ is. Hence, $i$ restricts to a functor
    $E\rightarrow\Dscr_{\{0\}}(\AA^1,\Cscr)^\omega$, also denoted $i$. Moreover, the additivity theorem, applied to this same cofiber sequence, viewed as a cofiber sequence of functors $E\to\Dscr(\AA^1,\Cscr)^\omega$,
    %via the inclusion $\Dscr_{\{0\}}(\AA^1,\Cscr)^\omega\subseteq\Dscr(\AA^1,\Cscr)^\omega$, 
    induces a nullhomotopic map $\K(E)\rightarrow\K(\AA^1,\Cscr)$. The
    underlying object functor
    $u:\Dscr_{\{0\}}(\AA^1,\Cscr)^\omega\rightarrow E$ sending $(x,e)$ to $x$
    gives $u\circ i\we\id_{E}$.
    It follows that $\K(E)$ is a summand of $\K_{\{0\}}(\AA^1,\Cscr)$ and that this summand maps trivially to
    $\K(\AA^1,\Cscr)$. 

    Now, suppose that $\K_{-m}(F)=0$ for all $1\leq m\leq n$ and all stable
    $\infty$-categories $F$ which admit bounded $t$-structures with noetherian
    hearts. The remarks above prove that $\K_{-n-1}(E)$ is a subquotient
    of $\K_{-n}(\Gm,\Cscr)$. By Corollary~\ref{cor:tstructuregm}, there is a
    bounded $t$-structure on $\Dscr(\Gm,\Cscr)^\omega$ with noetherian heart.
    Hence, $\K_{-n}(\Gm,\Cscr)=0$ by the inductive hypothesis and so $\K_{-n-1}(E)=0$ as well.
\end{proof}

\subsection{The nonconnective theorem of the heart}\label{sub:abelian}

In this section we prove Conjecture~\hyperlink{conj:c}{C} in the case of a
noetherian heart.

\begin{theorem}[Nonconnective theorem of the heart]\label{thm:heart}
    If $E$ is a small stable $\infty$-category with a bounded $t$-structure
    such that $E^\heartsuit$ is noetherian,
    then the natural map $$\K(E^\heartsuit)\xrightarrow{\we}\K(E)$$ is an equivalence.
\end{theorem}

To give the theorem content, we must define $\K(A)$ when $A$ is an abelian
category, show that this agrees with other definitions in the literature, and
define the map $\K(E^\heartsuit)\rightarrow\K(E)$.
We will use the terminology and results about prestable $\infty$-categories
of~\cite{sag}*{Appendix~C},
which in turn follows work of Krause~\cite{krause} on homotopy categories of
injective complexes.

\begin{lemma}
    If $A$ is a small abelian category, then
    $\Ind(A)$ is a Grothendieck abelian category, the Yoneda embedding
    $A\rightarrow\Ind(A)$ is exact, and the natural map
    $A\rightarrow\Ind(A)^\omega$ is an equivalence.
\end{lemma}

\begin{proof}
    Since $A$ has finite colimits, $\Ind(A)$ is presentable. Moreover, it is
    not difficult to see that $\Ind(A)$ is abelian. To see that filtered
    colimits preserve monomorphisms, use that filtered colimits preserve finite
    limits. Yoneda is always right exact,
    and it preserves finite colimits that exist in $A$. This proves exactness.
    The last claim follows because $A$ is idempotent complete.
\end{proof}

\begin{definition}
    Let $\check{\Dscr}(\Ind(A))$ denote the unseparated derived
    $\infty$-category of $\Ind(A)$ as defined in~\cite{sag}*{Section~C.5.8}. It
    is the dg nerve of the dg category of complexes of injective objects in
    $\Ind(A)$. In particular, by~\cite{ha}*{Remark~1.3.2.3}, $\Ho(\check{\Dscr}(\Ind(A)))$ is the homotopy
    category of injectives as studied by Krause~\cite{krause}.
    There is a right complete $t$-structure on $\check{\Dscr}(\Ind(A))$, and
    $\check{\Dscr}(\Ind(A))_{\geq 0}$ is anticomplete
    (see~\cite{sag}*{Section~C.5.5}) with an important universal property: it
    is initial among Grothendieck prestable $\infty$-categories
    $\Cscr$ with $\Cscr^\heart\we\Ind(A)$ (see~\cite{sag}*{Corollary~C.5.8.9}).
\end{definition}

\begin{definition}\label{def:bounded}
    Let $A$ be a small abelian category. We define the bounded derived
    $\infty$-category of $A$ to be
    $\Dscr^b(A)=\check{\Dscr}(\Ind(A))^\omega$. In particular, $\Dscr^b(A)$
    is a small idempotent complete stable $\infty$-category.
\end{definition}

\begin{lemma}
    The $\infty$-category $\check{\Dscr}(\Ind(A))$ is compactly generated by
    $\Dscr^b(A)$.
\end{lemma}

\begin{proof}
    This is the content of~\cite{krause}*{Theorem~4.9}. In the setting of~\cite{sag}*{Appendix~C},
    we invoke the fact that $\check{\Dscr}(\Ind(A))_{\geq 0}$ is coherent
    (by~\cite{sag}*{Corollary~C.6.5.9}) and anticomplete to conclude that
    $\check{\Dscr}(\Ind(A))_{\geq 0}$ is compactly generated
    by~\cite{sag}*{Theorem~C.6.7.1}. Since $\check{\Dscr}(\Ind(A))$ is right
    complete, a compact object of $\check{\Dscr}(\Ind(A))_{\geq 0}$ is compact
    when viewed in $\check{\Dscr}(\Ind(A))\we\Sp(\check{\Dscr}(\Ind(A))_{\geq
    0})$. Let $y\in\check{\Dscr}(\Ind(A))$. We have to show that if
    $\Map_{\check{\Dscr}(\Ind(A))}(x,y)$ for all $x$ in
    $\check{\Dscr}(\Ind(A))_{\geq n}^\omega$ and all $n$, then $y\we 0$. But, if this
    condition is satisfied, then $y\in\check{\Dscr}(\Ind(A))_{\leq n}$ for all
    $n$. Since $\check{\Dscr}(\Ind(A))$ is right separated, $y\we 0$.
\end{proof}

\begin{lemma}
    The canonical $t$-structure on $\check{\Dscr}(\Ind(A))$ restricts to a bounded
    $t$-structure on $\Dscr^b(A)$ with heart equivalent to $A$.
\end{lemma}

\begin{proof}
    This follows from~\cite{sag}*{Theorem~C.6.7.1}.
\end{proof}

\begin{lemma}\label{lem:defbounded}
    Let $\Dscr'\subseteq\Dscr(\Ind(A))$ denote the full subcategory of objects
    $x$ such that $\H_n(x)\in A\subseteq\Ind(A)$ for all $n$ and such that
    $\H_n(x)$ is non-zero for at most finitely many $n\in\ZZ$. Then, the map
    $\check{\Dscr}(\Ind(A))\rightarrow\Dscr(\Ind(A))$ restricts to an
    equivalence $\Dscr^b(A)\rightarrow\Dscr'$.
\end{lemma}

\begin{proof}
    The claim can be checked at the level of homotopy categories, which
    is the other part of~\cite{krause}*{Theorem~4.9}.
\end{proof}

\begin{definition}
    If $A$ is a small abelian category, then we define $\K(A)=\K(\Dscr^b(A))$,
    the nonconnective $K$-theory spectrum of the idempotent complete stable
    $\infty$-category $\Dscr^b(A)$ (defined in~\cite{bgt1}).
\end{definition}

We want to construct ``the natural map $\K(E^\heartsuit)\rightarrow\K(E)$'' of
the statement of Theorem~\ref{thm:heart}.

\begin{proposition}
    Let $E$ be a small stable $\infty$-category with a bounded $t$-structure.
    Then, there is a natural $t$-exact functor
    $\Dscr^b(E^\heartsuit)\rightarrow E$ inducing an equivalence on hearts.
\end{proposition}

\begin{proof}
    To define the natural map in the statement of Theorem~\ref{thm:heart}, let $E$
    be a small stable $\infty$-category with a bounded $t$-structure, and let
    $A=E^\heartsuit$.
    By~\cite{sag}*{Corollary~C.5.8.9}, there is a left exact functor
    $\check{\Dscr}(\Ind(A))_{\geq 0}\rightarrow\Ind(E)_{\geq 0}$ inducing the
    equivalence $\Ind(A)\we\Ind(E)^\heartsuit$.

%     The unseparated derived $\infty$-category $\check{\Dscr}(\Ind(A))$ is the
%     universal right complete, anticomplete, $0$-complicial stable presentable
%     $\infty$-category with heart equivalent to $\Ind(A)$.
%     See~\cite{sag}*{Appendix~C} for definitions. We have already seen that $\Ind(A)$
%     is right complete. In particular,
%     $\check{\Dscr}(\Ind(A))\we\Sp(\check{\Dscr}(\Ind(A))_{\geq 0})$ and
%     $\Ind(E)\we\Sp(\Ind(E)_{\geq 0})$, where $\Sp(\Cscr_{\geq 0})$ denotes the
%     stable presentable $\infty$-category of spectrum objects in a Grothendieck
%     prestable $\infty$-category.

    By~\cite{sag}*{Proposition~C.3.2.1}, the induced functor
    $F:\check{\Dscr}(\Ind(A))\rightarrow\Ind(E)$ is $t$-exact and induces an
    equivalence on hearts. It suffices to
    check that $F$ preserves compact objects. By Lemma~\ref{lem:defbounded}, every compact object of
    $\check{\Dscr}(\Ind(E))$ is a finite iterated fiber of maps between shifts of
    objects in $A\subseteq\check{\Dscr}(\Ind(E))^\heartsuit$. Thus, it suffices to
    show that $F(x)\in E$ when $x\in A$. But, this follows from hypothesis.
% 
%     In any case, it induces a
%     $t$-exact functor $F^+:\check{\Dscr}^+(\Ind(A))\rightarrow\Ind^+(E)$. But, by the
%     definition of the $t$-structure on $\check{\Dscr}(\Ind(A))$
%     in~\cite{sag}*{Section~C.5.8}, we have that $\check{\Dscr}(\Ind(A))_{\leq
%     0}\we\Dscr(\Ind(A))_{\leq 0}$ so that
%     $\check{\Dscr}^+(\Ind(A))\we\Dscr^+(\Ind(A))$, the derived $\infty$-category of
%     homologically-bounded above complexes of objects in $\Ind(A)$.
%     By Lemma~\ref{lem:defbounded}, objects of $A$ are compact when viewed in
%     $\check{\Dscr}(\Ind(A))$. Since these map under $F$ to objects of
%     $E^\heart\subseteq\Ind(E)^\omega$, it follows that $F$ preserves compact
%     objects: every compact object is an iterated cofiber of maps between
%     shifts of objets of $A\subseteq\check{\Dscr}(\Ind(A))^\heartsuit$.
\end{proof}

\begin{corollary}
    Let $E$ be a small stable $\infty$-category with a bounded $t$-structure.
    Then, there is a natural map $\K(E^\heartsuit)\rightarrow\K(E)$ of
    nonconnective $K$-theory spectra.
\end{corollary}

\begin{proof}
    Apply $\K$ to the exact functor $\Dscr^b(E^\heartsuit)\rightarrow E$.
\end{proof}

With this in mind, we turn to the proof of the theorem.

\begin{proof}[Proof of Theorem~\ref{thm:heart}]
    The first step is to use Barwick's theorem of
    the heart to prove that the induced map
    $\K^{\cn}(E^\heartsuit)\rightarrow\K^{\cn}(E)$ of
    connective $K$-theory is an equivalence. Philosophically, this \emph{is}
    Barwick's theorem, but we have defined $K$-theory in terms of stable
    $\infty$-categories instead of using exact $\infty$-categories.

    Consider the commutative triangle
    \begin{equation*}
        \xymatrix{
            &E^\heartsuit\ar[ld]\ar[rd]&\\
            \Dscr^b(E^\heartsuit)\ar[rr]&&E
        }
    \end{equation*}
    of Waldhausen $\infty$-categories in the sense of~\cite{barwick}. For the
    moment, denote by
    $\Kbar$ the (connective) $K$-theory spectrum of a Waldhausen
    $\infty$-category, as constructed in~\cite{barwick-ktheory}. Then,~\cite{barwick}*{Theorem~6.1} shows that the
    $\Kbar(E^\heartsuit)\rightarrow\Kbar(\Dscr^b(E^\heartsuit))$ and
    $\Kbar(E^\heartsuit)\rightarrow\Kbar(E)$ are equivalences. Hence,
    $\Kbar(E^\heartsuit)\rightarrow\Kbar(E)$ is an equivalence.

    By~\cite{barwick-ktheory}*{Corollary~10.6},
    $\Kbar(\Dscr^b(E^\heartsuit))$ and $\Kbar(E)$ are equivalent to the
    Waldhausen $K$-theory of suitable Waldhausen categories, and these are in
    turn equivalent to $\K^{\cn}(\Dscr^b(E^\heartsuit))$ and $\K^{\cn}(E)$,
    respectively, by~\cite{bgt1}*{Theorem~7.8}. This proves the result in
    connective $K$-theory.
    
    Now, in the situation of the theorem, both $\Dscr^b(E^\heartsuit)$ and $E$
    are have bounded $t$-structures with noetherian hearts. It follows from
    Theorem~\ref{thm:gscnh} that $\K_{-n}(E^\heartsuit)=\K_{-n}(E)=0$ for $n\geq 1$.
    This completes the proof.
\end{proof}

\begin{remark}
    If $E$ is a small stable $\infty$-category with a bounded
    $t$-structure, then $\K^{\cn}(E)\we\K^{\cn}(\Dscr^b(E^\heartsuit))$
    is equivalent to the Quillen $K$-theory
    of $E^\heartsuit$ viewed as an exact category. This follows from the
    Gillet-Waldhausen theorem~\cite{thomason-trobaugh}*{Theorem~1.11.7} and a
    theorem of Waldhausen~\cite{thomason-trobaugh}*{Theorem~1.11.2}. In the
    end, the theorem of the heart is a generalization of the Gillet-Waldhausen
    theorem.
\end{remark}

\begin{remark}
    Note that the negative $K$-groups  of a small abelian category $\Ascr$,
    defined in this paper as $\K_{-n}(\Dscr^b(\Ascr))$ for $n\geq 1$, agree with the negative
    $K$-groups of $\Ascr$ as defined by Schlichting~\cite{schlichting}. To see
    this, denote the latter by $\K_{-n}^\S(\Ascr)$ for the moment.

    As it is not necessary for our paper, we only illustrate the argument.
    Recall that a Frobenius pair is a pair $(\Escr,\Escr_0)$ of small Frobenius
    categories where $\Escr_0$ is a full subcategory of $\Escr_0$ such that
    the embedding $\Escr_0\rightarrow\Escr$ preserves projective (or equivalently
    injective) objects. A morphism of Frobenius pairs $(\Escr,\Escr_0)\rightarrow(\Fscr,\Fscr_0)$
    consists of a functor $\Phi:\Escr\rightarrow\Fscr$ such that
    $\Phi(M)\in\Fscr_0$ if $M\in\Escr_0$. (See for example~\cite{schlichting}.)
    Let $\Frob$ denote the $\infty$-category of Frobenius pairs.
    Using the functor
    $\Dscr_{\sing}$
    defined in Appendix~\ref{sec:frobenius}, we obtain a functor
    $\Dscr_{\sing}:\Frob\rightarrow\Cat_{\infty}^{\perf}$ defined by letting
    $$\Dscr_{\sing}(\Escr,\Escr_0)={\left(\Dscr_{\sing}(\Escr)/\Dscr_{\sing}(\Escr_0)\right)}^\sim,$$
    the idempotent completion of the Verdier quotient.
    Note that $\Dscr_{\sing}(\Escr_0)\rightarrow\Dscr_{\sing}(\Escr)$ is fully
    faithful because any map $\Escr_0$ that factors through a projective in
    $\Escr$ also factors through a projective in $\Escr_0$.

    By the definition of an exact sequence in $\Frob$ given
    in~\cite{schlichting}, $\Dscr_{\sing}$ sends exact sequences to
    localization sequences. Moreover, if $(\Escr,\Escr_0)$ is a flasque
    Frobenius pair, meaning that there is an endofunctor $T$ of the pair such
    that $T\we T\oplus\id$, then $\Dscr_{\sing}(\Escr,\Escr_0)$ is flasque.
    Using these facts, and the fact that $K_0$ can be computed either in
    $\Frob$ or in $\Cat_{\infty}^\perf$, it follows
    that $\K_{-n}^{\S}(A)\iso\K_{-n}(A)$.
\end{remark}

\subsection{Counterexamples using non-stably coherent rings}\label{sub:counterexamples}

While our proof of Theorem~\ref{thm:gscnh} uses crucially the hypothesis that
the small stable $\infty$-category $E$ has a bounded $t$-structure with noetherian
heart, much of the proof works for a general  $E$ with any bounded
$t$-structure. In particular, the existence of the sequence
$\K_{\{0\}}(\AA^1,\Cscr)\rightarrow\K(\AA^1,\Cscr)\rightarrow\K(\Gm,\Cscr)$, where
$\Cscr=\Ind(E)$, exists without any hypothesis on $E$ except that it be small
and stable, as does the fact that $\K(E)$ itself is a summand of
$\K_{\{0\}}(\AA^1,\Cscr)$. The fact that $\K_{-1}(E)=0$ whenever $E$ admits a
bounded $t$-structure provides additional
strength to the assertion that $\K_{-n}(E)$ should be zero for all $n\geq 1$.

The noetherian hypothesis is used to prove that the $t$-structure of
Lemma~\ref{lem:tstructureendomorphisms} on $\Dscr(\Gm,\Cscr)$ restricts to a
(bounded) $t$-structure on $\Dscr(\Gm,\Cscr)^\omega$. This leads to the inductive step.
One may ask if this is true in general, with a different proof.
This is not the case.

% Having proved that $\K_{-1}(E)=0$ for small idemptotent-complete stable
% $\infty$-categories with bounded $t$-structures, to prove our generalization of
% Schlichting's conjecture, we must find an inductive construction that produces
% new deloopings of $E$ which also have bounded $t$-structures.
% 
% A first idea is to mimic the approach of Bass, Quillen, and Schlichting, by
% studing certain categories of endomorphisms and automorphisms. We will see
% below that in the non-noetherian case, the naive way of making this
% generalization does not work, so a more subtle attach is necessary.
% 
% Consider the localization sequence $$\D(\text{$\AA^1$ on
% $\{0\}$};\Ind(E))\rightarrow\D(\AA^1,\Ind(E))\rightarrow\D(\Gm,\Ind(E)).$$
% Here, $\D(\AA^1,\Ind(E))=\Mod_{\AA^1}\otimes\Ind(E)$ is to be viewed as a kind
% of derived category of $\AA^1$ with coefficients in $\Ind(E)$.
% After taking compact objects, we obtain the fiber sequence
% $$\K^{\nil}(E)\rightarrow\K^{\AA^1}(E)\rightarrow\K^{\Gm}(E).$$ Moreover,
% $\K(E)$ is a summand of $\K^{\nil}(E)$, $\K_{-n}(E)$ is in the image of the
% boundary map $\K_{-n+1}^{\Gm}(E)\rightarrow\K_{-n}^{\nil}(E)$ for all $n$.
% Hence, if we can prove that $\Perf_{\Gm}^E$, the $\infty$-category of compact
% objects in $\Mod_{\Gm}^E$ has a bounded $t$-structure, then this will give
% inductive vanishing of the negative $K$-theory of $E$.

Let $R$ be an ordinary ring. A finitely presented right $R$-module $M$ is coherent if every
finitely generated submodule $N\subseteq M$ is finitely presented. The ring $R$
is \df{right coherent} if $R$ is coherent as a right $R$-module. We say that
$R$ is \df{right regular} if every finitely presented right $R$-module has
finite projective dimension. Finally, we call $R$ \df{right regular coherent} if it
is both right coherent and right regular.

\begin{theorem}\label{thm:km1coherent}
    If $R$ is a right regular coherent ring, then $\K_{-1}(R)=0$.
\end{theorem}

\begin{proof}
    The right coherence of $R$ means that the category
    $\Mod_R^{\heartsuit,\omega}$ of finitely presented
    $R$-modules is abelian and hence that
    $\Dscr^b(R)=\Dscr^b(\Mod_R^{\heartsuit,\omega})$ is a well-defined small stable
    presentable $\infty$-category. The right regularity of $R$ means that
    the natural map $\Mod_R^\omega\rightarrow\Dscr^b(R)$ is an equivalence. Hence,
    $\K(R)\we\K(\Dscr^b(R))$. We can conclude in either of two ways. We can appeal
    to Theorem~\ref{thm:minus1} using that the equivalence induces a bounded
    $t$-structure on $\Mod_R^\omega$. Or, we can use that $\K_{-1}(\Dscr^b(R))=0$,
    as proved by Schlichting~\cite{schlichting}*{Theorem 6}.
\end{proof}

Note that if $R$ is an ordinary ring, then $\Dscr(\Gm,\Mod_R)\we\Mod_{R[s^{\pm
1}]}$ and the $t$-structure on $\Dscr(\Gm,\Mod_R)$ induced by that on $\Mod_R$
via Lemma~\ref{lem:tstructureendomorphisms} agrees with the standard
$t$-structure on $\Mod_{R[s^{\pm 1}]}$.

\begin{proposition}
    Suppose that $R$ is right regular coherent but that $R[s^{\pm 1}]$ is not right
    coherent. Then, the $t$-structure on $\Mod_{R[s^{\pm 1}]}$ does not
    restrict to a $t$-structure on compact objects.
\end{proposition}

\begin{proof}
    Note that $R[s^{\pm 1}]$ is right regular. Let
    $I\subseteq R[s^{\pm 1}]$ be a finitely generated ideal that is not finitely presented, and let
    $P:R[s^{\pm 1}]^n\rightarrow R[s^{\pm 1}]$ be a chain complex in degrees $1$
    and $0$, with $d(R[s^{\pm 1}]^n)=I$. By our
    choice of $I$, $\H_1(P)$ is not a finitely generated $R[s^{\pm 1}]$-module. If the $t$-structure on
    $\Mod_{R[s^{\pm 1}]}$ restricted to a $t$-structure on the compact objects, then $\H_1(P)$ would
    have to be perfect, and hence finitely presented, a contradiction.
\end{proof}

\begin{example}
    Glaz presents an example of Soublin in~\cite{glaz}*{Example~7.3.13} showing
    that the ring $$R=\prod_{m\in\ZZ}\QQ[\![x,y]\!]$$ is right regular coherent, but that $R[s]$ is not coherent.
    By~\cite{glaz}*{Theorem~8.2.4(2)}, $R[s^{\pm 1}]$ is not coherent either.
\end{example}

This example implies that the strategy used in the previous
section to prove Conjecture~\hyperlink{conj:b}{B} in the case of a noetherian heart
cannot work for general small stable $\infty$-categories $E$ equipped with
bounded $t$-structures.

% Amusingly, however, it might be the case that $\K_{-n}(R)=0$ for all $n$ in the
% Soublin example anyways. Indeed, Carlsson proved in~\cite{carlsson-products}
% that connective $K$-theory commutes with infinite products. If nonconnective
% $K$-theory does as well, then one would have $$\K_{-n}(R)\iso\prod_{m\geq
% 0}\K_{-n}(\QQ[\![x,y]\!])=0$$ for $n\geq 1$, since $\QQ[\![x,y]\!]$ is regular
% noetherian.

Another strategy is to replace $\Dscr(\AA^1,\Cscr)^\omega$ by some
stable $\infty$-category that does have a bounded $t$-structure. For
example, one can consider the abelian closure
$\widetilde{\coh}(\AA^1,\Cscr)$ in
$\Dscr(\AA^1,\Cscr)^\heartsuit$ of the additive category consisting of
$\pi_nx$ as $x$ ranges over all compact objects of $\Dscr(\AA^1,\Cscr)$.
Let $\Dscr_{\widetilde{\coh}}^b(\AA^1,\Cscr)\subseteq\Dscr(\AA^1,\Cscr)$ be
the full subcategory of bounded objects $x$ such that
$\pi_nx\in\widetilde{\coh}(\AA^1,\Cscr)$ for all $n$. This is a small
stable $\infty$-category and the induced $t$-structure is bounded.

Now, consider the localization sequence
$$\Dscr_{\{0\}}(\AA^1,\Cscr)\cap\Dscr_{\widetilde{\coh}}^b(\AA^1,\Cscr)\rightarrow\Dscr_{\widetilde{\coh}}^b(\AA^1,\Cscr)\rightarrow\left(\Dscr_{\widetilde{\coh}}^b(\AA^1,\Cscr)/\Dscr_{\{0\}}(\AA^1,\Cscr)\cap\Dscr_{\widetilde{\coh}}^b(\AA^1,\Cscr)\right)^\sim.$$
Let $\Dscr_{\{0\},\widetilde{\coh}}^b(\AA^1,\Cscr)$ denote the left-hand
side. For this to play the role of the localization sequence
$\Dscr_{\{0\}}(\AA^1,\Cscr)^\omega\rightarrow\Dscr(\AA^1,\Cscr)^\omega\rightarrow\Dscr(\Gm,\Cscr)^\omega$,
we need to guarantee two things:
\begin{enumerate}
    \item[(1)] $\K(E)$ is a summand of
        $\K(\Dscr_{\{0\},\widetilde{\coh}}^b(\AA^1,\Cscr))$;
    \item[(2)] the abelian subcategory
        $$\Dscr_{\{0\},\widetilde{\coh}}^b(\AA^1,\Cscr)^\heartsuit\subseteq\Dscr_{\widetilde{\coh}}^b(\AA^1,\Cscr)^\heartsuit$$
        is Serre.
\end{enumerate}
Condition (2) is easier and is in fact always true.
Condition (1) would follow if $x$ is compact when
$(x,e)$ is an object of $\Dscr_{\{0\},\widetilde{\coh}}^b(\AA^1,\Cscr)$.

\subsection{Stable coherence}\label{sub:fj}

The next theorem was known to Bass and Gersten.

\begin{theorem}\label{thm:stablycoherent}
    Suppose that $R$ is a right regular coherent ring such that $R[s_1,\ldots,s_n]$
    is right coherent. Then, $\K_{-n-1}(R)=0$.
\end{theorem}

\begin{proof}
    It is enough to note that in this case $R[s_1^{\pm 1},\ldots,s_n^{\pm 1}]$
    is right coherent by~\cite{glaz}*{Theorem~8.2.4(2)}. Hence,
    $\K_{-1}(R[s_1^{\pm 1},\ldots,s_n^{\pm 1}])=0$. As $\K_{-n-1}(R)$ is a
    subquotient of $\K_{-1}(R[s_1^{\pm 1},\ldots,s_n^{\pm 1}])$, the result
    follows from Theorem~\ref{thm:km1coherent}.
\end{proof}

The classical proof (due to Bass~\cite{bass-problems}*{Section~2})
uses a specific inductive presentation of $\K_{-n-1}(R)$,
namely as the cokernel of 
$$\K_{-n}(R[s])\oplus\K_{-n}(R[s^{-1}])\rightarrow\K_{-n}(R[s^{\pm 1}]).$$
(See~\cite{thomason-trobaugh}*{Section~6}.) In particular, $\K_{-n-1}(R)$ is a
quotient of $\K_{-1}(R[s_1^{\pm 1},\ldots,s_n^{\pm 1}])$, which vanishes if
$R[s_1^{\pm 1},\ldots,s_n^{\pm 1}]$ is right regular coherent.

One reason to prefer our proof is that it extends immediately to small abelian
categories $A$ such that $A[s_1,\ldots,s_n]$ is abelian, with notation as in
the proof of Proposition~\ref{prop:tstructurea1}. We know of no analogous
general result using Bass' methods in the literature. We restate this result separately.

\begin{theorem}\label{thm:stablycoherenta}
    Let $A$ be a small abelian category such that $A[s_1,\ldots,s_n]$ is abelian.
    Then, $\K_{-n-1}(A)=0$.
\end{theorem}

It seems that Schlichting's paper is very close to establishing a result
like this. However, the proof given of~\cite{schlichting}*{Lemma 8} relies on the structure of
injective modules in a noetherian abelian category to establish the long
exact sequence in $K$-groups allowing one to conclude that that
$\K_{-n-1}(A)$ is a subquotient of $\K_{-1}(A[s_1^{\pm 1},\ldots,s_n^{\pm
1}])$. So, that paper allows one to prove the theorem for $A$
noetherian (\cite{schlichting}*{Remark 7}), but not in this stably coherent setting.

We close this section with a discussion relating these vanishing results and the
$K$-theoretic Farrell-Jones conjecture. None of these results are new (as they
all follow from Theorem~\ref{thm:stablycoherent}), but it
serves to illustrate the importance of
Conjectures~\hyperlink{conj:a}{A},~\hyperlink{conj:b}{B},
and~\hyperlink{conj:c}{C} in the non-noetherian case.

% Recall, for example from~\cite{luck-reich}*{Conjecture 58}, that the
% $K$-theoretic Farrell-Jones conjecture asserts that the natural map
% $$\colim_{V\subseteq G}\K(R[V])\rightarrow\K(R[G])$$ is an equivalence when $R$
% is a ring, $G$ is a group, and the colimit is over the category of virtually
% cyclic subgroups of $G$.
The most interesting cases of the conjecture are when $R=\ZZ$ or when $R$
is an arbitrary regular noetherian commutative ring.
Farrell and Jones~\cite{farrell-jones} proved that $\K_{-n}(\ZZ[V])=0$ for $n\geq 2$.
If the Farrell-Jones conjecture holds for $G$, then
it follows from the homotopy colimit spectral sequence
that $\K_{-n}(\ZZ[G])=0$ for $n\geq 2$ as well. In many cases it is suspected
that $K_{-n}(R[G])=0$ for all $n\geq 1$. For example, this follows from the
Farrell-Jones conjecture when the orders of all finite subgroups of $G$ are
invertible in $R$ (see~\cite{luck-reich}*{Conjecture~79}).
% As highlighted
% in~\cite{luck-reich}*{Section 2.2.2}, the vanishing of the negative $K$-groups
% of $\ZZ[G]$ when $G$ is torsion-free is important in its own right, and is
% related to the classification of bounded $h$-cobordisms over high-dimensional
% manifolds. See~\cite{luck-reich}*{Theorem~10}.
Our application to this problem is via a class of groups studied in this
setting by Waldhausen~\cite{waldhausen-free}. Say that a group $G$ is \df{regular
coherent} (resp. {\bf noetherian}) if $R[G]$
is right regular coherent (resp. {\bf noetherian}) for any regular noetherian
commutative ring $R$.

\begin{theorem}
    Let $R$ be a regular noetherian commutative ring, and let $G$ be a regular
    coherent group. Then, $\K_{-n}(R[G])=0$ for $n\geq 1$.
\end{theorem}

\begin{proof}
    The key point is that $R[G][s]\iso (R[s])[G]$, so that
    under the hypotheses, $R[G]$ is stably coherent. The result follows from
    Theorem~\ref{thm:stablycoherent}.
\end{proof}

The Farrell-Jones conjecture is known in many cases, and hence the vanishing of
negative $K$-theory is known in many cases over the integers by the result of
Farrell and Jones. For a table of known results on the
Farrell-Jones conjecture, see~\cite{luck-reich}*{Section~2.6.3}.
% Of course,
% that reference eleven years old at the time of writing, and new results have
% been obtained since.

\begin{example}\label{ex:fj}
    Many groups are regular and coherent. The following list is transcribed
    from~\cite{waldhausen-free}. The group $G$ is regular coherent if it is
    (1) a free group, (2) a free abelian group, (3) a polycyclic group,
    (4) a torsion-free one-relator group, (5) a group of the form $\pi_1M$ where $M$ is a $2$-manifold not
    homeomorphic to $\RR\PP^2$,
    (6) a sufficiently large $3$-manifold group,
    (7) a group of the form $\pi_1M$ where $M$ is a submanifold of $S^3$,
    (8) a subgroup of a group of one of the above types, or (9) a filtered
    colimit of inclusions thereof. In particular, for all of these groups, $\K_{-n}(R[G])=0$ for $n\geq 1$.
    Example (8) is particularly interesting as regular coherence passes to
    subgroups by~\cite{waldhausen-free}*{Theorem~19.1} even though this is not known for the Farrell-Jones conjecture.
\end{example}

\subsection{Serre cones of abelian categories}\label{sub:serrecones}

We have seen that the straightforward generalization of Schlichting's inductive
strategy to prove vanishing of negative $K$-theory of noetherian abelian
categories founders because of the failure of the Serre subcategory condition on the
hearts, even though though the weak Serre subcategory condition always holds.

% The problem is the simultaneous satisfaction of two requirements for the main
% localization sequence used in the argument. In order to run the
% inductive argument we want
% \begin{enumerate}
%     \item   an exact functor $i:E\rightarrow F$ with an exact section $s$,
%     \item   a fully faithful exact functor $j:F\rightarrow G$ such that $j\circ
%         i:E\rightarrow G$ is $K$-acyclic,
%     \item   bounded $t$-structures on $F$, $G$, and $G/F$.
% \end{enumerate}
% The resulting localization sequence
% $F\rightarrow G\rightarrow G/F$ in $\Cat_{\infty}^{\perf}$ together with the
% $K$-acyclicity of $j\circ i$ implies that $\K_{-n}(E)$ is a subquotient of
% $\K_{-n+1}(G/F)$. If we can
% find such a localization sequence, then
% $\K_{-2}(E)$ vanishes because $\K_{-1}(G/F)$ vanishes. Schlichting
% uses the nilpotent localization sequence, which works for noetherian or more
% generally in our situation for stably coherent abelian categories.
% 
% It seems to us that
% the easiest way to produce such a sequence is for $F$ and $G$ to have bounded
% $t$-structures, for the functor $F\rightarrow G$ to be $t$-exact, and for
% $F^{\heart}$ to be a Serre subcategory of $G^{\heart}$. Then, appealing
% Proposition~\ref{prop:tlocalization}, there is an induced (bounded) $t$-structure on $G/F$.
% This suggests the following question, which would be enough to prove
% Schlichting's conjecture.

Part of the subtlety of Schlichting's conjecture is that the negative
$K$-theory of an abelian category is defined using derived categories. To date,
there is no definition internal to abelian categories. Given a small idempotent
complete stable
$\infty$-category $E$ and an uncountable regular cardinal $\kappa$, let $\Sigma_\kappa(E)$ be the cofiber in
$\Cat_{\infty}^\perf$ fitting into the exact sequence
$$E\rightarrow\Ind(E)^\kappa\rightarrow\Sigma_\kappa(E).$$ This allows us to define
negative $K$-theory inductively as $\K_{-n}(E)=\K_0(\Sigma^{(n)}_\kappa(E))$
since $\K(\Ind(E)^\kappa)\we 0$.

No similar construction is known for the negative $K$-theory of abelian
categories. It would be possible given a positive answer to the next question.

\begin{question}\label{q:serrecone}
    Suppose that $A$ is a small abelian category. Does there exist an exact
    fully faithful inclusion of abelian categories $A\subseteq B$ such that
    $\K(B)\we 0$ and such that $A$ is Serre in $B$?
\end{question}

For example, $B$ might be closed under countable coproducts, which implies the
$K$-acyclicity condition. One natural guess would be to take a category
$\Ind(A)^\kappa$ of $\kappa$-compact objects for an uncountable cardinal
$\kappa$. However, $A\subseteq\Ind(A)^\kappa$ is not typically Serre.

\begin{example}
    Let $R$ be a non-noetherian coherent ring, and let
    $\coh_R\subseteq\Mod_R^{\heartsuit,\kappa}$
    be the full subcategory of coherent right $R$-modules inside all
    $\kappa$-compact $R$-modules (for some regular uncountable cardinal $\kappa$). Then, $R$ itself
    has subobjects (specifically, non-finitely generated ideals) in
    $\Mod_R^{\heartsuit,\kappa}$ not contained in $\coh_R$.
\end{example}

\begin{remark}
    Jacob Lurie informed us that the previous example extends to say that it is not
    generally true that a small abelian category $A$ admits a fully faithful
    exact inclusion $A\subseteq B$ where $B$ is closed under countable
    coproducts and $A$ is Serre inside of $B$. Indeed, if $B$ has countable
    coproducts, then for any object $M$ of $B$, the lattice $\mathrm{Sub}(M)$ of
    subobjects of $M$ is closed under countable joins. This property will be
    true for any Serre subcategory. An example where is not true is as follows.
    Consider the category $\coh_R$ of coherent
    modules for $R=k[x_1,x_2,\ldots]$, the polynomial ring on countably many
    variables over some field $k$. This ring is coherent but not noetherian, so
    that $\coh_R$ is abelian. However, the union of
    the ideals $(x_1)\subseteq (x_1,x_2)\subseteq\cdots$ is not coherent (or
    even finitely generated). So, the lattice of coherent subobjects of $R$ is not closed under countable joins.
    In particular, there is no Serre
    embedding $\coh_R\subseteq B$ for any $B$ closed under countable coproducts.
\end{remark}

\begin{proposition}
    Conjecture~\hyperlink{conj:c}{C} and a positive answer to
    Question~\ref{q:serrecone} imply
    Conjectures~\hyperlink{conj:a}{A} and~\hyperlink{conj:b}{B}.
\end{proposition}

\begin{proof}
    For $n\geq 1$, let $A(n)$ denote the statement that $\K_{-n}(A)=0$ for all
    small abelian categories $A$, and let $B(n)$ denote the statement that
    $\K_{-n}(E)=0$ for every small stable $\infty$-category $E$ with a bounded
    $t$-structure. Since we are assuming Conjecture~\hyperlink{conj:c}{C}, $A(n)$ if and
    only if $B(n)$. So, assume $B(n)$ for some $n\geq 1$. It suffices to prove
    $A(n+1)$.

    Let $A$ be a small abelian category, and let $A\subseteq B$
    denote the abelian category guaranteed by the hypothesis of the
    proposition. Conjecture~\hyperlink{conj:c}{C} implies that
    $\K(A)=\K(\Dscr^b(A))\we\K(\Dscr^b_A(B))$. (Note that even this special case
    of Conjecture~\hyperlink{conj:c}{C} is open: see~\cite{weibel-kbook}*{Open
    Problem~V.5.3}.)
    Consider the localization sequence
    $$\Dscr^b_A(B)\rightarrow\Dscr^b(B)\rightarrow\left(\Dscr^b(B)/\Dscr^b_A(B)\right)^\sim.$$
    By our choice of $B$, $\K(\Dscr^b(B))\we 0$, so
    $$\K_{-n}((\Dscr^b(B)/\Dscr^b_A(B))^\sim)\rightarrow\K_{-n-1}(\Dscr_A^b(B))\iso\K_{-n-1}(A)$$
    is surjective.
    The quotient $(\Dscr^b(B)/\Dscr^b_A(B))^\sim$ has a bounded $t$-structure by Proposition~\ref{prop:tlocalization}.
    Thus, by $B(n)$, $\K_{-n-1}(A)=0$.
\end{proof}

% \begin{remark}
%     The key idea we are trying to suggest here is that the definition of
%     negative $K$-theory of abelian categories, as it is currently conceived,
%     passes through the derived category. It is only there that we can construct
%     the $K$-theoretic suspension functors. Indeed, for any small stable
%     idempotent complete
%     $\infty$-category $E$ and any uncountable regular cardinal $\kappa$, there is an exact sequence of small
%     stable idempotent complete $\infty$-categories $E\rightarrow
%     F_\kappa(E)\rightarrow\Sigma_\kappa(E)$, where $F_\kappa(E)$ can be taken
%     to be $\Ind(E)^\kappa$. If Question~\ref{q:serrecone} has a positive
%     answer, then the localization sequence $A\rightarrow B\rightarrow B/A$ of abelian categories
%     could play a similar role.
% \end{remark}

\begin{remark}
    As a final philosophical remark, note that negative $K$-theory exists
    because of the need to idempotent complete when constructing a
    localization of stable $\infty$-categories. Since abelian categories are
    idempotent complete, the sequences $A\rightarrow B\rightarrow B/A$ are
    already exact when $A\subseteq B$ is Serre. In particular, the induced map
    $\K_0(B)\rightarrow\K_0(B/A)$ is always surjective for such a localization
    sequence. It follows that, to the extent it exists along the lines
    of~\cite{bgt1}, the universal localizing invariant \emph{of
    abelian categories} should actually be connective $K$-theory. This does not
    imply Schlichting's conjecture by itself, but it would provide some
    evidence.
\end{remark}

\section{Some applications}\label{sec:negative}

We present here applications of the vanishing results above to the $K$-theory
of dg algebras and of ring spectra.
% In particular, Corollary~\ref{cor:KUintro}
% is proved in Section~\ref{sub:nperiodic}.

\subsection{Negative $K$-theory of dg algebras}\label{sub:dg}

Our first result has also been
proved independently by Denis-Charles Cisinski in unpublished work.

\begin{theorem}
    Let $k$ be a commutative ring, and let $A$ be a cohomological dg $k$-algebra such that
    $\H^0(A)$ is semisimple, $\H^i(A)$ is a finitely generated
    right $\H^0(A)$-module for all $i$, and $\H^i(A)=0$ for $i<0$. Then, $\K_{-n}(A)=0$ for
    $n\geq 1$.
\end{theorem}

\begin{proof}
    Keller and Nicol\'as prove in~\cite{keller-nicolas}*{Theorem~7.1} that
    under these hypotheses, $\Mod_A^\omega$ admits a bounded $t$-structure
    whose heart is a length category. Recall that a length category is a small
    abelian category in which every object has finite length. In particular, it
    is noetherian. The result follows now from Theorem~\ref{thm:heart}.
\end{proof}

\begin{example}
    Suppose that $k$ is a field and $R$ is a noetherian local commutative
    $k$-algebra with maximal ideal $\mathfrak{m}$. Then, the derived
    endomorphism algebra $A$ of $R/\mathfrak{m}$, which computes
    $\Ext^*_R(R/\mathfrak{m},R/\mathfrak{m})$ satisfies the hypotheses of the
    theorem, and hence the negative $K$-theory of $A$ vanishes. We can see this
    in another way. Let $\Ascr\subseteq\Mod_R^{\heartsuit,\omega}$ be the full
    subcategory of finitely presented $R$-modules supported set theoretically
    on $\Spec R/\mathfrak{m}\subseteq\Spec R$. Then, $\Dscr^b(\Ascr)$ is a
    fully subcategory of $\Dscr^b(R)$ and $R/\mathfrak{m}$ is a compact
    generator. Hence, $\Mod_A^\omega\we\Dscr^b(\Ascr)$. So, since $R$ is
    noetherian, $\Ascr$ is noetherian, and the fact that $\K_{-n}(A)=0$ for
    $n\geq 1$ follows from Schlichting's theorem.
\end{example}

\begin{example}
    If $k$ is a field and $X$ is a smooth proper $k$-scheme, then the algebraic
    de Rham complex, which computes the algebraic de Rham cohomology
    $\H^*_{\dR}(X/k)$, satisfies the conditions of the theorem.
\end{example}

\subsection{Negative $K$-theory of periodic and related ring spectra}\label{sub:nperiodic}

Let $R$ be a connective ring spectrum. 
A right $R$-module $M$ is \df{$\pi_*$-finitely presented} if $\bigoplus_n\pi_nM$ is
a finitely presented (right) $\pi_0R$-module. In particular, this means that
$M$ is bounded and that each $\pi_nM$ is a finitely presented
$\pi_0R$-module.
% \footnote{In~\cite{barwick-lawson}, $\pi_*$-finitely presented
% $R$-modules are called \emph{coherent}. We prefer our terminology because
% a right coherent ring spectrum $R$ in the sense
% of~\cite{ha}*{Definition~7.2.5.16} need not be $\pi_*$-finitely presented
% as a right $R$-module spectrum.}
A discrete ring $R$ is said to be \df{right noetherian} if every submodule of a finitely generated $R$-module is
finitely generated. A connective ring spectrum $R$ is \df{right noetherian} if
$\pi_0R$ is right noetherian and if $\pi_nR$ is finitely generated as a right
$\pi_0R$-module for all $n\in\NN$.

Following~\cite{mcconnell-robson}, a discrete ring $R$ is said to be \df{right regular} if every
finitely generated discrete (right) $R$-module has finite projective dimension. A
connective ring
spectrum $R$ will be said to be \df{right regular} if $\pi_0R$ is right regular and
if each $\pi_*$-finitely presented (right) $R$-module spectrum $M$ is compact.
A connective ring
spectrum $R$ will be called \df{right regular noetherian} if it is right
noetherian and right regular.

For the purposes of this section,
a map $R\rightarrow S$ of ring spectra will be called a \df{localization} if the induced map
$\Mod_R\rightarrow\Mod_S$ is a localization with kernel generated by a
compact object, or equivalently by a finite set of compact objects.

The next result extends those of Barwick and Lawson in~\cite{barwick-lawson}.

\begin{proposition}\label{prop:bl}
    Let $R$ be a right regular noetherian ring spectrum.
    Suppose that $R\rightarrow S$ is a localization of $R$ such that for a
    compact $R$-module $M$, $S\otimes_R
    M\we 0$ if and only if $M$ is $\pi_*$-finitely presented. Then, there is a fiber sequence
    $$\K(\pi_0R)\rightarrow\K(R)\rightarrow\K(S)$$  of
    nonconnective $K$-theory spectra.
\end{proposition}

\begin{proof}
    Let $\Mod_{R}^{\pi_*\textrm{-fp}}\subseteq\Mod_R^{\omega}$ be the full
    subcategory of $\pi_*$-finitely presented $R$-modules. The localization
    theorem in algebraic $K$-theory gives a fiber sequence
    $$\K(\Mod_R^{\pfp})\rightarrow\K(R)\rightarrow\K(S).$$
    Since $R$ is connective, there is a bounded $t$-structure on
    $\Mod_R^{\pfp}$ with noetherian heart (the category of finitely presented
    discrete right $R$-modules). The result follows from
    Theorem~\ref{thm:heart}.
% 
%     In~\cite{barwick-lawson}*{Corollary~2.3}, it is shown that the natural map
%     $\K(\pi_0R)\rightarrow\K(\Mod_R^{\pfp})$ is an equivalence on connective
%     covers.\footnote{Barwick and Lawson assume that $\pi_0R$ is a regular \emph{commutative}
%     ring, but it is enough to assume that $R$ is right regular noetherian.}
%     But, $\K_{-n}(\Mod_R^{\pfp})=0$ for $n\geq 1$ by
%     Theorem~\ref{thm:gscnh}, and $\K_{-n}(\pi_0R)=0$ for $n\geq 1$ by
%     Schlichting's result~\cite{schlichting} since $\pi_0R$ is right regular noetherian.
\end{proof}

\begin{corollary}
    If $R$ is a right regular noetherian ring spectrum and $R\rightarrow S$ is a localization of
    $R$ such that for a compact $R$-module $M$, $S\otimes_RM\we 0$ if and only
    if $M$ is $\pi_*$-finitely presented, then
    $\K_{-n}(S)=0$ for all $n\geq 1$.
\end{corollary}

\begin{proof}
    Indeed, $\K_{-n}(\pi_0R)=0$ for $n\geq 1$ since $R$ is right regular
    noetherian. Moreover, $\K_{-n}(R)\iso\K_{-n}(\pi_0R)$ for $n\geq 1$
    by~\cite{bgt1}*{Theorem~9.53}.
\end{proof}

There are many examples of regular ring spectra admitting localizations
satisfying the condition of the theorem. The consequences for negative
$K$-theory are new and
require the methods of this paper.

\begin{example}\label{ex:negative}
    \begin{enumerate}
        \item   If $R$ is a ring spectrum with $\pi_*R\iso\pi_0R[u]$ where
            $|u|=2m>0$ and $\pi_0R$ is right regular noetherian, then $R\rightarrow R[u^{-1}]$
            satisfies the conditions of the theorem. In particular, if $S$ is
            an even periodic ring spectrum with $\pi_0S$ right regular
            noetherian, then $\K_{-n}(S)=0$ for $n\geq 1$.
%             (Corollary~\ref{cor:periodicintro}.)
        \item   In particular, $\K_{-n}(\KU)=0$ for $n\geq 1$. This extends
            the theorem of Blumberg and Mandell~\cite{blumberg-mandell}.
%             (Corollaries~\ref{cor:kuintro} and~\ref{cor:KUintro}.)
        \item   Similarly, $\K_{-n}(\E_m)=0$, $\K_{-n}(\K_m)=0$, and
            $\K_{-n}(\K(m))=0$ for $n\geq 1$ and $m\geq 0$, where
            $\E_m$ is the Morava $E$-theory spectrum, $\K_m$ is the $2$-periodic
            Morava $K$-theory spectrum, and $\K(m)$ is the $2(p^m-1)$-periodic
            Morava $K$-theory spectrum.
        \item   Barwick and Lawson show in~\cite{barwick-lawson} that $\ko$ is
            right regular noetherian, and that
            $\ko\rightarrow\KO$ satisfies the hypothesis of the theorem. Hence,
            $\K_{-n}(\KO)=0$ for $n\geq 1$.
%             (Corollaries~\ref{cor:kointro} and~\ref{cor:KOintro}.)
        \item   They also show that $\mathrm{tmf}$ is right regular noetherian, and that
            $\mathrm{tmf}\rightarrow\mathrm{Tmf}$ satisfies the hypothesis of
            Proposition~\ref{prop:bl}. Therefore, $\K_{-n}(\mathrm{Tmf})=0$ for
            $n\geq 1$.
%             (Corollaries~\ref{cor:kointro} and~\ref{cor:KOintro}.)
    \end{enumerate}
\end{example}

% \begin{remark}
%     \todo{Delete this once it's known that it really doesn't work.}
%     There is another proof that the negative $K$-groups vanish in these examples.
%     Despite its simplicity, it has not appeared in print before now. We give a
%     sliver of the argument, leaving the general case to the interested reader.
%     We prove that $\K_{-2}(\KU)=0$. Consider the diagram
%     \begin{equation*}
%         \xymatrix{
%         \Dscr_{\beta=0,\{0\}}(\ku[t])^\omega\ar[r]\ar[d]   & \Dscr_{\{0\}}(\ku[t])^\omega\ar[r]\ar[d]  & \Dscr_{\{0\}}(\KU[t])^\omega\ar[d]\\
%         \Dscr_{\beta=0}(\ku[t])^\omega\ar[r]\ar[d]   & \Dscr(\ku[t])^\omega\ar[r]\ar[d]  & \Dscr(\KU[t])^\omega\ar[d]\\
%         \Dscr_{\beta=0}(\ku[t^{\pm 1}])^\omega\ar[r]   & \Dscr(\ku[t^{\pm
%         1}])^\omega\ar[r]  & \Dscr(\KU[t^{\pm 1}])^\omega\\
%         }
%     \end{equation*}
%     in which the rows are exact sequences obtained by inverting $\beta$, a
%     generator of $\pi_2\ku$, and the columns are exact sequences obtained by
%     inverting $t$. As in the proof of Theorem~\ref{thm:gscnh}, using the right
%     column, we see that $\K_{-2}(\KU)$ is a
%     summand of $\K_{-2}(\Dscr_{\{0\}}(\KU[t])^\omega)$, which is a subquotient of
%     $\KU_{-1}(\Dscr_{\beta=0,\{0\}}(\ku[t]^\omega))$ using the top row and the fact
% \end{remark}

\begin{example}
    Not all periodic ring spectra concentrated in even degrees satisfy the
    hypotheses of Example~\ref{ex:negative}(1). For example, the Johnson-Wilson theories $\E(m)$
    with $m\geq 2$ have $$\pi_*\E(m)=\ZZ_{(p)}[v_1,\ldots,v_{m-1},v_m^{\pm
    1}],$$ where $|v_i|=2p^i-2$. Hence, they are periodic with period $2(p^m-1)$,
    but they are not concentrated in multiples of this degree. We do not know
    if $\K_{-n}(\E(m))=0$ for $m\geq 2$ and $n\geq 1$. 
\end{example}

\subsection{Negative $K$-theory of cochain algebras}\label{sub:ncochains}

In a different direction, we consider cochain algebras.

\begin{theorem}\label{thm:cochain}
    Let $X$ be a compact space and $R$ a regular noetherian
    discrete commutative ring.
    There is an equivalence $\oplus_{x\in\pi_0X}\K(R)\we\K(\C^*(X,R))$ of nonconnective $K$-theory
    spectra. In particular, $\K_{-n}(\C^*(X,R))=0$ for $n\geq 1$.
\end{theorem}

\begin{proof}
    It is enough to consider the case when $X$ is connected, so that $X\simeq B\Omega X$.
    Let $$\Loc_X(\Mod_R)\we\Fun(X^{\op},\Mod_R)\we\Mod_{\C_*(\Omega X,R)}$$
    be the $\infty$-category of local systems on $X$ with coefficients in
    the stable $\infty$-category $\Mod_R$ of complexes of $R$-modules. Since
    the endomorphism algebra of the constant local system on $R$ is
    $\C^*(X,R)$, there is a fully faithful
    functor $$\Mod_{\C^*(X,R)}\rightarrow\Mod_{\C_*(\Omega
    X,R)}.$$ As $R$ is connective, so is $\C_*(\Omega X,R)$, and hence there is an induced $t$-structure on
    $\Loc_X(\Mod_R)$.

    If $X$ is compact (in the $\infty$-category of spaces), then $R$ is compact when viewed as a
    $\C_*(\Omega X,R)$-module (for example by~\cite{dgi}*{Proposition~5.3}). But, $R$ corresponds to $\C^*(X,R)$ under the
    functor above. It follows that $\Mod_{\C^*(X,R)}\rightarrow\Loc_X(\Mod_R)$
    sends compact objects to bounded objects with respect to the $t$-structure on
    $\Loc_X(\Mod_R)$. Moreover, the $t$-structure restricts to a $t$-structure on
    $\Mod_{\C^*(X,R)}$ by Mathew's
    description~\cite{mathew-galois}*{Proposition~7.8} of the essential image as the
    ind-unipotent modules over $\C_*(\Omega X,R)$, a condition which depends only
    on the action of $\pi_1X$ on the homotopy groups of the $R$-module of the underlying local system.
    Hence, $\Mod_{\C^*(X,R)}^\omega$ has a bounded $t$-structure, with heart
    easily seen to be the
    abelian category of finitely presented $R$-modules.

    The theorem now follows immediately from the nonconnective theorem of the
    heart (Theorem~\ref{thm:heart}) and
    the fact that $\K_{-n}(R)=0$ for $n\geq 1$.
\end{proof}

\appendix

\section{Frobenius nerves}\label{sec:frobenius}

We examine an $\infty$-categorical model of the stable category of a Frobenius
category. This material is used in the main body of the paper
to verify that Schlichting's definition of the negative $K$-theory of a
small abelian category $A$ agrees with the negative $K$-theory of the small
stable $\infty$-category $\Dscr^b(A)$, as defined in~\cite{bgt1}.
% Section~\ref{sub:gorenstein} observes a generalization of the theorem of Buchweitz
% on singularity categories of Gorenstein rings. This is not needed elsewhere in
% the paper but provides another application of the nerve construction below and
% another source of questions about negative $K$-theory.

Let $\Escr$ be a small exact category in the sense of Quillen~\cite{quillen}. We
will identify $\Escr$ with a full subcategory of
$\Ascr=\Fun^{\lex}(\Escr^{\op},\Mod_{\ZZ}^\heartsuit)$, the category of left
exact additive functors $\Escr^{\op}\rightarrow\Mod_{\ZZ}^{\heartsuit}$. The
(Yoneda) embedding $\Escr\rightarrow\Ascr$ is exact and reflects exactness.
Moreover, $\Escr$ is closed under extensions in $\Ascr$. If, additionally,
$\Escr$ is idempotent complete, then $\Escr$ is closed under taking kernels of
epimorphisms in $\Ascr$. See~\cite{thomason-trobaugh}*{Proposition~A.7.16}.
In other words, $\Escr$ satisfies
hypothesis~\cite{thomason-trobaugh}*{1.11.3.1}, the key assumption needed for
the Gillet-Waldhausen theorem~\cite{thomason-trobaugh}*{Theorem~1.11.7}. We
refer to~\cite{thomason-trobaugh}*{Appendix A} in general for details about the
Gabriel-Quillen embedding.

Mimicking the definitions in
an abelian category, we say that an object $P$ of $\Escr$ is {\bf projective} if
for every admissible epi $M\rfib N$ the induced map
$\Hom_{\Escr}(P,M)\rightarrow\Hom_\Escr(P,N)$ is surjective. Dually, an object
$I$ of $\Escr$
is {\bf injective} if $\Hom_\Escr(N,I)\rightarrow\Hom_\Escr(M,I)$ is surjective for
every admissible mono $M\rcof N$ in $\Escr$.

We say that $\Escr$ has {\bf enough projectives} if for every object $M$ of
$\Escr$ there is an admissible epi $P\rfib M$ where $P$ is projective.
Let $\Escr^{\proj}$ denote the full subcategory of projective objects of
$\Escr$. Similarly, $\Escr$
has {\bf enough injectives} if for every object $M$ of $\Escr$ there is an
admissible mono $M\rcof I$ where $I$ is injective.

A {\bf Frobenius category} is an exact category which has enough injectives and
projectives and an object of $\Escr$ is projective if and only if it is
injective. 

\begin{construction}
    If $\Escr$ is a Frobenius category, the {\bf stable category}
    $\underline{\Escr}$ of
    $\Escr$ has the same objects as $\Escr$ with morphisms
    $\Hom_{\underline{\Escr}}(M,N)$ the
    quotient of $\Hom_\Escr(M,N)$ by the subgroup of morphisms $f:M\rightarrow N$
    factoring through a projective (or equivalently injective) object of $\Escr$.
\end{construction}

\begin{remark}
    The stable category $\underline{\Escr}$ of a Frobenius category $\Escr$ is
    triangulated. This was first observed by
    Happel~\cite{happel}*{Theorem~9.4} following ideas of 
    A. Heller~\cite{heller}. The loopspace of an object $M$ is obtained by taking
    an admissible exact sequence $\Omega M\rcof P\rfib M$ with $P$ projective. Then,
    $\Omega M$ is isomorphic to $M[-1]$ in $\underline{\Escr}$. We will write
    $\Omega_nM$
    for the $n$-fold iteration $\Omega\cdots\Omega M$. Note that $\Omega_nM$ is not
    in general a well-defined endofunctor of $\Escr$, but that it defines an
    endofunctor of $\underline{\Escr}$.
\end{remark}

Let $\Escr$ be an idempotent complete exact category. In this section, we will associate to
$\Escr$ a stable $\infty$-category $\Dscr_{\sing}(\Escr)$, the singularity
$\infty$-category of $\Escr$, and show that its homotopy category is naturally
equivalent to $\underline{\Escr}$ when $\Escr$ is Frobenius.

A special case of such a construction can be extracted from Hovey~\cite{hovey}*{Section~2.2}.
A right noetherian ring $R$ is {\bf quasi-Frobenius} if $R$ is injective as a right $R$-module.
See~\cite{curtis-reiner}*{Section~58}. In this case, the category $\Mod_R^{\heartsuit}$ of right $R$-modules
is Frobenius, and Hovey constructs a model category structure on $\Mod_R^{\heartsuit}$
whose homotopy category is equivalent to the stable category of
$\Mod_R^{\heartsuit}$. Hovey's construction does not seem to generalize because a small Frobenius category
need not embed into a Grothendieck abelian category which is also Frobenius.
Specifically, there are examples where the Gabriel-Quillen embedding $\Escr\rightarrow\Ascr$ does not
preserve injectives. Hence, we take a different approach.

\begin{example}
    Let $k$ be a field and let $G$ be a locally finite group that is not
    finite (such as $\QQ/\ZZ$).
    Then, $k[G]$ is not (right) self-injective by Renault~\cite{renault}, so in
    particular $\Mod_{k[G]}^\heartsuit$ is not Frobenius.
    On the other hand,
    $G$ is the filtered colimit of its finite subgroups, and hence $k[G]$ is
    the filtered colimit of Frobenius sub-algebras along flat transition maps.
    In particular, $k[G]$ is coherent, for example
    by~\cite{glaz}*{Theorem~2.3.3}. It follows
    that the category of finitely presented (right) $k[G]$-modules is abelian.
    It is not hard to check that $k[G]$ is injective in
    $\Mod_{k[G]}^{\heartsuit,\omega}$, which shows that the category of
    finitely presented $k[G]$-modules is Frobenius.
\end{example}

Let $\Escr$ be an exact category.
Let $\Ch^-(\Escr)$ denote the category of bounded below chain complexes of objects in
$\Escr$. This is a dg category and the dg nerve $\N_{\dg}(\Ch^-(\Escr))$
of~\cite{ha}*{Construction~1.3.1.6} is a stable $\infty$-category
by~\cite{ha}*{Proposition~1.3.2.10}. The homotopy category of
$\N_{\dg}(\Ch^-(\Escr))$ is the category of bounded chain complexes up to chain
homotopy. For simplicity, we will write $\Ch_{\dg}^-(\Escr)$ for
$\N_{\dg}(\Ch^-(\Escr))$.

\begin{definition}
    A complex $X$ in $\Ch^-(\Escr)$ is \df{acyclic in degree $n$} if there is a
    factorization of the differential $X_n\leftarrow X_{n+1}:d_{n+1}$ into
    $$X_n\lcof Z_n\lfib X_{n+1}$$ where $X_n\lcof Z_n$ is an admissible mono
    and a kernel for $d_n$ and $X_{n+1} \rfib Z_{n}$ is an admissible epi and a 
    cokernel for $d_{n+2}$. The complex $X$ is \df{acyclic} if it is acyclic in
    degree $n$ for all $n\in\ZZ$.
\end{definition}

Consider the following full stable subcategories of
$\Ch_{\dg}^-(\Escr)$:
\begin{enumerate}
    \item[(i)] $\Ch^b_{\dg}(\Escr)$, the dg nerve of the dg category of
        bounded complexes in $\Escr$;
    \item[(ii)] $\Ac^-_{\dg}(\Escr)$ and $\Ac^b_{\dg}(\Escr)$, the dg nerve of the dg category of
        acyclic bounded below and bounded complexes, respectively, in $\Escr$;
    \item[(iii)] $\Ch^{-,b}_{\dg}(\Escr)$ the dg nerve of the dg category of
        bounded below complexes in $\Escr$ which are acyclic in all
        sufficiently high degrees.
\end{enumerate}

\begin{remark}
    If $\Escr\subseteq\Fscr$ is fully faithful, then
    $\Ch^-(\Escr)\rightarrow\Ch^-(\Fscr)$ is fully faithful, which leads to
    fully faithful maps between all of the subcategories above.
\end{remark}

\begin{lemma}\label{lem:idemomni}
    Let $\Escr$ be an idempotent complete exact category.
    \begin{enumerate}
        \item[\emph{(a)}]   The stable $\infty$-categories $\Ac^b_{\dg}(\Escr)$
            and $\Ac^-_{\dg}(\Escr)$ are idempotent complete in
            $\Ch^b_{\dg}(\Escr)$ and $\Ch^-_{\dg}(\Escr)$, respectively.
        \item[\emph{(b)}]   Any chain complex in $\Ch^-_{\dg}(\Escr)$ chain homotopy
            equivalent to an acyclic chain complex is itself acyclic. In other
            words, $\Ac^b_{\dg}(\Escr)$ and $\Ac^-_{\dg}$ are closed under equivalence in
            $\Ch^b_{\dg}(\Escr)$ and $\Ch^-_{\dg}(\Escr)$, respectively.
    \end{enumerate}
\end{lemma}

\begin{proof}
    Let $\Escr\rightarrow\Ascr$ denote the Gabriel-Quillen embedding.
    We prove first that if $X\in\Ch_{\dg}^-(\Escr)$, then $X$ is acyclic if and
    only if $\H_*(X)=0$ when $X$ is viewed as a complex of objects in $\Ascr$.
    If $X$ is acyclic, then $\H_*(X)=0$ by definition. We can suppose that $X$
    is of the form $0\leftarrow X_0\leftarrow X_1\leftarrow\cdots$. Since
    $\H_*(X)=0$, it follows that $X_0\leftarrow X_1$ is surjective. Since
    $\Escr$ is idempotent complete, it is closed under kernels of admissible
    epimorphisms in $\Ascr$. Hence, there is a factorization $X_1\lcof Z_1\lfib
    X_2$ where $Z_1$ is the kernel of $X_0\leftarrow X_1$ and the cokernel of
    $X_1\leftarrow X_2$. By induction, the claim follows.

    Now, part (a) follows immediately. Indeed, if $X\we Y\oplus Z$ in
    $\Ch^b_{\dg}(\Escr)$ (resp. $\Ch^-_{\dg}(\Escr)$) where $X$ is acyclic, then $\H_*(X)=0$, so
    $\H_*(Y)=\H_*(Z)=0$. So, $Y$ and $Z$ are bounded (resp. bounded below) complexes with vanishing
    homology. By the previous paragraph, they are acyclic.

    Part (b) follows as well, since if $X\rightarrow Y$ is an equivalence in
    $\Ch^-_{\dg}(\Escr)$ with $X$ acyclic, we find that $\H_*(Y)=0$, so $Y$ is
    acyclic.
\end{proof}

\begin{definition}
    Let $\Escr$ be an idempotent complete exact category.
    The \df{bounded derived $\infty$-category} $\Dscr_{\dg}^b(\Escr)$ of
    $\Escr$ is the Verdier quotient
    $$\Ch^b_{\dg}(\Escr)/\Ac^b_{\dg}(\Escr).$$ The homotopy category of
    $\Dscr_{\dg}^b(\Escr)$ is equivalent to the usual bounded derived category of
    $\Escr$. Similarly, the \df{bounded below derived $\infty$-category} $\Dscr_{\dg}^-(\Escr)$ is the Verdier quotient
    $$\Ch^-_{\dg}(\Escr)/\Ac^-_{\dg}(\Escr),$$ while the \df{homologically
        bounded derived
    $\infty$-category}
    $\Dscr_{\dg}^{-,b}(\Escr)$ is the Verdier quotient
    $$\Ch^{-,b}_{\dg}(\Escr)/\Ac^-_{\dg}(\Escr).$$ A map in
    $\Ch^-_{\dg}(\Escr)$ which is an equivalence in any of these derived
    $\infty$-categories is called a {\bf quasi-isomorphism}. By
    Lemma~\ref{lem:idemomni}(b), these are precisely the maps in
    $\Ch^-_{\dg}(\Escr)$ whose cones are acyclic.
\end{definition}

\begin{remark}
    If $\Ascr$ is a small abelian category viewed as an exact category in the
    usual way, then $\Dscr^b_{\dg}(\Ascr)\we\Dscr^b(\Ascr)$, where
    $\Dscr^b(\Ascr)$ is defined as in Section~\ref{sub:abelian} as
$\check{\Dscr}(\Ind(\Ascr))^\omega$.
\end{remark}

\begin{proposition}
    The natural functor
    $\Dscr^b_{\dg}(\Escr)\rightarrow\Dscr^{-,b}_{\dg}(\Escr)$ is an equivalence
    and the natural functor $\Dscr^{-,b}_{\dg}(\Escr)\rightarrow\Dscr^-_{\dg}(\Escr)$
    is fully faithful.
\end{proposition}

\begin{proof}
    The second functor is trivially fully faithful since
    $\Ac^-_{\dg}(\Escr)\subseteq\Ch^{-,b}_{\dg}(\Escr)$. We prove that the
    composition is fully faithful. For this, it suffices to verify Verdier's
    criterion~\cite{verdier}*{Proposition~II.2.3.5}. Thus, if $f:M\rightarrow X$
    is a map in $\Ch^-_{\dg}(\Escr)$ with $M$ in $\Ac^-_{\dg}(\Escr)$ and $X$
    in $\Ch^b_{\dg}(\Escr)$, we show that $f$ factors through $M\rightarrow M'$
    where $M'$ is in $\Ac^b_{\dg}(\Escr)$. Choose $n$ such that $X_i=0$ for
    $i\geq n$. Since $M$ is acyclic, the good truncation $\tau_{\leq n}M$
    exists in $\Ch^b_{\dg}(\Escr)$ and is also acyclic. The map
    $M\rightarrow X$ factors through $M\rightarrow\tau_{\leq n}M$ since $X_n=0$.

    To see essential surjectivity, let $X$ be in $\Ch^{-,b}_{\dg}(\Escr)$ and
    choose $n$ such that $X$ is acyclic in degrees $n$ and higher. Then, the
    good truncation $\tau_{\leq n}X$ exists in $\Ch^b_{\dg}(\Escr)$ and $X\rightarrow\tau_{\leq n}X$ is
    a quasi-isomorphism because the cone has zero homology and is hence
    acyclic by the argument in the proof of Lemma~\ref{lem:idemomni}.
\end{proof}

\begin{theorem}[Balmer-Schlichting~\cite{balmer-schlichting}*{Theorem~2.8}]\label{thm:bs}
    If $\Escr$ is idempotent complete, then the derived $\infty$-category
    $\Dscr^b_{\dg}(\Escr)$ is idempotent complete.
\end{theorem}

\begin{proof}
    This can be checked on the homotopy category, which is done
    in~\cite{balmer-schlichting}.
\end{proof}

\begin{lemma}\label{lem:contractible}
    Any complex $P$ in $\Ac^-_{\dg}(\Escr)\cap\Ch^-_{\dg}(\Escr^\proj)$ is
    contractible.
\end{lemma}

\begin{proof}
    This follows immediately from the projectivity of the terms of $P$.
\end{proof}

\begin{remark}
    It follows that $\Ch^b_{\dg}(\Escr^\proj)\we\Dscr^b_{\dg}(\Escr^\proj)$ and
    $\Ch^-_{\dg}(\Escr^\proj)\we\Dscr^-_{\dg}(\Escr^\proj)$, since
    the acyclic complexes are already equivalent to zero in
    $\Ch^b_{\dg}(\Escr^\proj)$.
\end{remark}

\begin{corollary}
    The stable $\infty$-category $\Ch^b_{\dg}(\Escr^\proj)$ is idempotent
    complete.
\end{corollary}

\begin{proof}
    This is a special case of Theorem~\ref{thm:bs}.
\end{proof}

\begin{lemma}\label{lem:chainhomotopic}
    If $X$ is in $\Ac^-_{\dg}(\Escr)$ and $P$ is in
    $\Ch^-_{\dg}(\Escr^\proj)$, then any map $f:P\rightarrow X$ is chain
    homotopic to zero.
\end{lemma}

\begin{proof}
    We assume that $P_n=0$ for $n\leq -1$. Let $s_n:P_n\rightarrow X_{n+1}$ be
    the zero map for $n\leq -1$. Assume that $s_n$ has been constructed for
    $n\leq N-1$ such that $f_i=d_{i+1}^X\circ s_i+s_{i-1}\circ d_{i}^P$ for
    $i\leq N-1$. Then,
    \begin{align*}
        d_N^X\circ(f_N-s_{N-1}\circ d_N^P)  &= d_{N-1}^X\circ f_N-d_n^X\circ s_{N-1}\circ d_N^P\\
        &=d_{N-1}^X\circ f_N-(f_{N-1}-s_{N-2}\circ d_{N-1}^P)\circ d_{N}^P\\
        &=d_{N-1}^X\circ f_N-f_{N-1}\circ d_N^p\\
        &=0
    \end{align*}
    since $f$ is a map of chain complexes. It follows from the acyclicity of
    $X$ that $f_N-s_{N-1}\circ
    d_N^p$ factors through the admissible mono $Z_{N}\rcof X_N$. Since $X_{N+1}\rightarrow
    Z_N$ is an admissible epi, there is a lift of $f_N-s_{N-1}\circ d_N^p$ to
    $s_N:P_N\rightarrow X_{N+1}$ such that $d_{N+1}^X\circ s_N=f_N-s_{N-1}\circ
    d_N^P$. By induction, this proves the existence of a contracting homotopy
    for $f$.
\end{proof}

\begin{proposition}
    The functors $\Ch^b_{\dg}(\Escr^\proj)\rightarrow\Dscr^b_{\dg}(\Escr)$
    and $\Ch^-_{\dg}(\Escr^\proj)\rightarrow\Dscr^-_{\dg}(\Escr)$
    are fully faithful.
\end{proposition}

\begin{proof}
    We use Verdier's criterion~\cite{verdier}*{Proposition~II.2.3.5}, which
    says in our case that if every map $P\rightarrow X$ with $P$ in
    $\Ch^-_{\dg}(\Escr^\proj)$ and $X$ in $\Ac^-_{\dg}(\Escr)$ factors through
    a map $X'\rightarrow X$ where $X'$ is in
    $\Ch^-_{\dg}(\Escr^\proj)\cap\Ac^-_{\dg}(\Escr)$, then
    $$\Ch^-_{\dg}(\Escr^\proj)/\Ch^-_{\dg}(\Escr^\proj)\cap\Ac^-_{\dg}(\Escr)\rightarrow\Dscr^-_{\dg}(\Escr)$$
    is fully faithful. But, Lemma~\ref{lem:chainhomotopic} says that in fact
    every such map factors through zero, so the criterion is satisfied. On the
    other hand, Lemma~\ref{lem:contractible} says every complex in 
    $\Ch^-_{\dg}(\Escr^\proj)\cap\Ac^-_{\dg}(\Escr)$ is already equivalent to
    zero, so that the conclusion of Verdier's criterion reduces to the statement
    of the proposition. The bounded case is similar, or follows from the fully
    faithfulness of
    $\Ch^b_{\dg}(\Escr^\proj)\rightarrow\Ch^-_{\dg}(\Escr^\proj)$ and
    $\Dscr^b_{\dg}(\Escr)\rightarrow\Dscr^-_{\dg}(\Escr)$.
\end{proof}

\begin{corollary}
    The natural functor
    $\Ch^-_{\dg}(\Escr^\proj)\rightarrow\Dscr^-_{\dg}(\Escr)$ is an
    equivalence.
\end{corollary}

\begin{proof}
    Thanks to the previous proposition it suffices to check essential
    surjectivity, which follows by taking projective resolutions.
\end{proof}

\begin{definition}
    Let $\Escr$ be an idempotent complete exact category. The {\bf singularity
    $\infty$-category} $\Dscr_{\sing}(\Escr)$ of $\Escr$ is the Verdier
    quotient $$\Dscr^b_{\dg}(\Escr)/\Ch^b_{\dg}(\Escr^\proj).$$ We will write
    $\Dscr_{\sing}$ for the induced functor from the $\infty$-category of exact
    categories and exact functors to $\Cat_{\infty}^\perf$.
\end{definition}

\begin{definition}
    Syzygys play a crucial role in the proof of the next theorem. Let $X$ in
    $\Ch^-_{\dg}(\Escr)$ be acyclic in degree $n-1$. Then, the $n$th syzygy
    $\Omega_nX$ is an object of $\Escr$, being the kernel of
    $d_{n-1}:X_{n-1}\rightarrow
    X_{n-2}$. Moreover, in this case, the brutal truncation $\sigma_{\geq
    n}X$ admits a canonical map to $\Omega_{n}X[n]$. When $X$ is acyclic,
    $\sigma_{\geq n}X\rightarrow\Omega_nX[n]$ is a quasi-isomorphism. Finally,
    if $X$ is a complex of projectives which is acyclic in degree $i$ for $i\geq n-1$, then
    $\Omega_{i}X\iso\Omega_{i-n}\Omega_nX$ in $\underline{\Escr}$.
\end{definition}

\begin{theorem}\label{thm:frobeniusstable}
    There is a natural equivalence $\Ho(\Dscr_{\sing}(\Escr))\we\underline{\Escr}$
    when $\Escr$ is an idempotent complete Frobenius category.
\end{theorem}

\begin{proof}
    The proof of this theorem is due in spirit to Buchweitz~\cite{buchweitz},
    though only a special case is given there. For simplicity, we
    avoid the comparison with the homotopy category of acyclic complexes of
    projectives, instead giving a direct argument for the equivalence.

    There is an evident composition of functors
    $\N(\Escr)\rightarrow\Ch^b_{\dg}(\Escr)\rightarrow\Dscr^b_{\dg}(\Escr)\rightarrow\Dscr_{\sing}(E)$,
    where the first is given by viewing an object of $\Escr$ as a chain complex
    concentrated in degree zero. This first functor is evidently fully
    faithful. The second and third functors are the Verdier quotient functors.

    Let $\Ch^{-,b}(\Escr^\proj)$ be the full subcategory of
    $\Ch^-(\Escr^\proj)$ consisting of {\bf homologically bounded complexes of
    projectives}, i.e., those complexes which are acyclic when viewed
    in $\Ch^-(\Escr)$ in all sufficiently high degrees. It is clear that the
    natural functor $\Ch^{-,b}_{\dg}(\Escr^\proj)\rightarrow\Dscr^-(\Escr)$
    induces an equivalence $\Ch^{-,b}_{\dg}(\Escr^\proj)\we\Dscr^{-,b}(\Escr)$.
    Hence, there are equivalences
    $$\Ch^{-,b}_{\dg}(\Escr^\proj)/\Ch^b_{\dg}(\Escr^\proj)\we\Dscr^{-,b}_{\dg}(\Escr)/\Ch^b_{\dg}(\Escr^\proj)\we\Dscr_\sing(\Escr).$$
    We are therefore free in our arguments to replace bounded complexes in
    $\Escr$ with homologically bounded complexes of projectives.

    We claim first that the functor
    $\underline{\Escr}\rightarrow\Dscr_{\sing}(\Escr)$ is essentially surjective.
    Pick a complex $X$ of $\Ch^b_{\dg}(\Escr)$, and let $P\rightarrow X$ be a
    quasi-isomorphism where $P$ is a bounded below complex of projectives.
    Choose $n\geq 0$ sufficiently large so that $P$ is acyclic in degree $i$
    for all $i\geq n$. In this case, the brutal truncation $\sigma_{\geq i}P$
    admits a quasi-isomorphism $\sigma_{\geq i}P\rightarrow\Omega_iP[i]$ for
    all $i>n$, where $\Omega_iP$ is some object of $\Escr$. Fix $i>n$
    and extend $\sigma_{\geq i}P$ to an unbounded acyclic complex $Q$ of
    projectives, by taking a projective co-resolution of $\Omega_iP$ (which
    exists because $\Escr$ is Frobenius). There is a diagram of morphisms
    $$X\leftarrow P\rightarrow\sigma_{\geq i}P=\sigma_{\geq i}Q\leftarrow
    \sigma_{\geq 0}Q\rightarrow \Omega_0Q$$ in $\Ch^{-,b}_{\dg}(\Escr)$. The outside arrows are
    quasi-isomorphisms and hence already equivalences in
    $\Dscr^{-,b}_{\dg}(\Escr)$. The inside arrows have cones in
    $\Ch^b_{\dg}(\Escr^\proj)$, and hence
    they become equivalences in $\Dscr_{\sing}(\Escr)$. But, this shows that
    $X\we\Omega_0Q$ in $\Dscr^{-,b}_{\dg}(\Escr)/\Ch^b_{\dg}(\Escr^\proj)\we\Dscr_{\sing}(\Escr)$.

    To finish the proof, it is enough to prove that
    $\underline{\Escr}\rightarrow\Ho(\Dscr_{\sing}(\Escr))$ is fully faithful.
    We check faithfulness and fullness separately.
    Let $M$ and $N$ be objects of $\Escr$. Since
    $\Ch^-_{\dg}(\Escr^\proj)\rightarrow\Dscr^-_{\dg}(\Escr)$ is fully
    faithful, by replacing $M$ and $N$ by projective resolutions,
    we see that
    $\Hom_\Escr(M,N)\rightarrow\pi_0\Map_{\Dscr^b_{\dg}(\Escr)}(M,N)$ is a
    bijection. Now, consider a diagram
    $$
    M\xleftarrow{s} X\xrightarrow{f} N
    $$
    in $\Ch^{-,b}_{\dg}(\Escr^\proj)$ 
    with $\cone(s)\in\Ch^b_{\dg}(\Escr^\proj)$. Then, $\Omega_n M\leftarrow \Omega_nX$ is an
    isomorphism up to projective summands for $n$ sufficiently large, so there is an induced map
    $\Omega_n M\rightarrow\Omega_n N$. Since $\Omega_n$ is an autoequivalence of
    $\underline{\Escr}$, fullness follows.
%     \marginpar{Explain that in $\underline{\Escr}$, presumably $\Omega\simeq\Omega_1$ (this is clear from the dual of the suspension operation $SM$ defined above) and $\Omega_n\simeq\Omega\Omega_{n-1}$.}

    To prove faithfulness, suppose that $f:M\rightarrow N$ maps to zero in
    $\pi_0\Map_{\Dscr_{\sing}}(M,N)$. Then, there is $X\xrightarrow{s}M$ such
    that $\cone(s)$ is quasi-isomorphic to an object of
    $\Ch^b_{\dg}(\Escr^\proj)$ and $f\circ s$ is zero in $\Dscr^b_{\dg}(\Escr)$.
    Working with bounded below complexes of projectives, we can assume in fact
    that $f\circ s$ is nullhomotopic in $\Ch^{-,b}_{\dg}(\Escr^\proj)$. In this
    case, $M\rightarrow N$ factors through $X\rightarrow\cone(s)$. A
    sufficiently high syzygy of $\cone(s)$ is projective, so this means that
    $\Omega_n f$ factors through a projective, and hence is zero in
    $\underline{\Escr}$. Again using that $\Omega_n$ is an autoequivalence, we
    find that $f=0$ in $\underline{\Escr}$, as desired.
\end{proof}

\begin{example}
    In general $\Dscr_{\sing}(\Escr)$ is not idempotent complete, and hence
    neither is $\underline{\Escr}$. It is enough to find a Gorenstein noetherian
    commutative ring $R$ with $\K_{-1}(R)\neq 0$, since in this case there is an isomorphism
    $$\K_0(\widetilde{\underline{\Escr}})/\K_0(\underline{\Escr})\iso\K_{-1}(R)$$
    (as $\K_{-1}(\Dscr^b_{\dg}(R))=0$ by Schlichting's theorem).
    The complete intersection $R=\ZZ[x_0,x_1]/(x_0x_1(1-x_0-x_1))$ works. In
    this case, $\K_{-1}(R)\iso\ZZ$, as can be checked
    from~\cite{weibel-isolated}.
\end{example}

Let $R$ be a noetherian commutative ring. The abelian category
$\Mod_R^{\heartsuit,\omega}$ of finitely presented discrete $R$-modules is
exact, and its negative $K$-theory vanishes by Schlichting. Hence,
$\K_{-n}(\Dscr_{\sing}(\Mod_R^{\heartsuit,\omega}))\rightarrow\K_{-n-1}(R)$ is
a surjection for $n\geq 0$ and an isomorphism for $n\geq 1$. In this way, the
singularity category supports the negative $K$-theory of $R$ and gives one
measurement of the singularities of $R$ itself.

When $R$ is not noetherian, the question of whether or not this connection
continues is precisely bound up in Schlichting's conjecture. For example, if
$R$ is merely coherent, then it is no longer known in general that
$\K_{-n}(\Dscr_{\sing}(\Mod_R^{\heartsuit,\omega}))\rightarrow\K_{-n-1}(R)$ is
an isomorphism for $n\geq 1$. This would follow from
Conjecture~\hyperlink{conj:a}{A}.

%%%%%%%%%%%%%%%%%%%%%
%%% End material. %%%
%%%%%%%%%%%%%%%%%%%%%

\vspace{20pt}
\scriptsize
\noindent
Benjamin Antieau\\
University of Illinois at Chicago\\
Department of Mathematics, Statistics, and Computer Science\\
851 South Morgan Street, Chicago, IL 60607\\
\texttt{antieau@math.uic.edu}

\vspace{10pt}
\noindent
David Gepner\\
School of Mathematics and Statistics\\
The University of Melbourne\\
813 Swanston Street\\
Parkville VIC 3010 Australia\\
\texttt{david.gepner@unimelb.edu.au}

\vspace{10pt}
\noindent
Jeremiah Heller\\
University of Illinois at Urbana-Champaign\\
Department of Mathematics\\
1409 W. Green Street, Urbana, IL 61801\\
\texttt{jbheller@math.uiuc.edu}

\end{document}